\documentclass[10pt, letterpaper]{article}
\usepackage[utf8]{inputenc}

\usepackage{amsmath,amsthm,color}
\usepackage{amssymb}  
\usepackage{amsfonts}
\usepackage{graphicx}
\usepackage{dsfont}
\usepackage{psfrag}
\usepackage{subfigure,epstopdf}
\usepackage{multirow}
\usepackage{float}
\usepackage{cite}
\usepackage{bbm}
\usepackage{savesym}
\usepackage{float}
\usepackage{etoolbox}
\usepackage[ruled,vlined]{algorithm2e}
\usepackage{algorithmic}
\usepackage{tikz}
\usetikzlibrary{positioning,chains,fit,shapes,calc}
\usetikzlibrary{backgrounds}
\usetikzlibrary{arrows.meta}
\usetikzlibrary{shapes,snakes}
\usepackage{tabularx}
\usepackage[font=small]{caption}

\definecolor{ao}{rgb}{0.55, 0.71, 0.0}
\definecolor{bleudefrance}{rgb}{0.19, 0.55, 0.91}
\definecolor{dimgray}{rgb}{0.41, 0.41, 0.41}    
\definecolor{mediumorchid}{rgb}{0.73, 0.33, 0.83}
\definecolor{mediumtealblue}{rgb}{0.0, 0.33, 0.71}
\definecolor{harvestgold}{rgb}{0.85, 0.57, 0.0}
\definecolor{blue(pigment)}{rgb}{0.2, 0.2, 0.6}
\definecolor{forestgreen(traditional)}{rgb}{0.27, 0.35, 0.27}
\definecolor{cadmiumred}{rgb}{0.89, 0.0, 0.13}
\definecolor{orange(webcolor)}{rgb}{1.0, 0.5, 0.0}

\newtheorem{theorem}{Theorem}
\newtheorem{corollary}{Corollary}

\newtheorem{lemma}{Lemma}
\newtheorem{definition}{Definition}

\title{A scalable distributed dynamical systems approach to compute the strongly connected components and diameter of networks}
\date{\today}

\author{$\text{Emily A. Reed}^{*}$,
$\text{Guilherme Ramos}$\footnote{Both authors contributed equally.
{\scriptsize E. Reed is with the Ming Hsieh Electrical and Computer Engineering Department at the University of Southern California, USA. G. Ramos is with the Department of Electrical and Computer Engineering, Faculty of Engineering, University of Porto, Portugal. P. Bogdan is a faculty member at the University of Southern California in the Ming Hsieh Electrical and Computer Engineering Department. S. Pequito is a faculty member at the Delft University of Technology in the Delft Center for Systems and Control.
 This work was supported in part by FCT project POCI-01-0145-FEDER-031411-HARMONY, National Science Foundation GRFP DGE-1842487, Career Award CPS/CNS-1453860, CCF-1837131, MCB-1936775, CNS-1932620, CMMI-1936624, CMMI 1936578, the University of Southern California Annenberg Fellowship, USC WiSE Top-Off Fellowship, the DARPA Young Faculty Award and DARPA Director Award N66001-17-1-4044. The views, opinions, and/or findings contained in this article are those of the authors and should not be interpreted as representing the official views or policies, either expressed or implied by the Defense Advanced Research Projects Agency, the Department of Defense or the National Science Foundation.}}, Paul Bogdan, S\'ergio Pequito}

\begin{document}
\maketitle

{\centering\footnotesize \emph{\textbf{Dedicated to Professor Cristina and Amilcar Sernadas.}}\par}

\begin{abstract}
Finding strongly connected components (SCCs) and the diameter of a directed network play a key role in a variety of discrete optimization problems, and subsequently, machine learning and control theory problems. On the one hand, SCCs are used in solving the 2-satisfiability problem, which has applications in clustering, scheduling, and visualization. On the other hand, the diameter has applications in network learning and discovery problems enabling efficient internet routing and searches, as well as identifying faults in the power grid. 

In this paper, we leverage consensus-based principles to find the SCCs in a scalable and distributed fashion with a computational complexity of $\mathcal{O}\left(Dd_{\text{in-degree}}^{\max}\right)$, where $D$ is the (finite) diameter of the network and $d_{\text{in-degree}}^{\max}$ is the maximum in-degree of the network. Additionally, we prove that our algorithm terminates in $D+1$ iterations, which allows us to retrieve the diameter of the network. We illustrate the performance of our algorithm on several random networks, including Erdős-Rényi, Barabási-Albert, and \mbox{Watts-Strogatz} networks. 
\end{abstract}


%

\section{Introduction}\label{sec:intro}

Strongly connected components (SCCs) are important in solving problems in clustering, scheduling, and visualization \cite{ramnath2004dynamic,even1975complexity,poon1998polynomial}, as well as in the context of control theory, including structural systems \cite{ramos2020structural} and distributed control \cite{bullo2009distributed}. The diameter is important in improving internet search engines  \cite{albert1999diameter}, quantifying the multifractal geometry of complex networks \cite{xue2017reliable}, and identifying faults in both the power grid \cite{zongxiang2004cascading} and multiprocessor systems \cite{kuhl1980distributed}. 

Nowadays, the networks associated with data are becoming increasingly larger, which demands \emph{scalable} and \emph{distributed} algorithms that enable an efficient determination of both the SCCs and diameter of such networks.

Identifying the different SCCs in a directed network (directed \mbox{graph -- digraph} for short) leads to a unique decomposition of the digraph $\mathcal G=(\mathcal V, \mathcal E)$, where $\mathcal V$ denotes the nodes and $\mathcal E$ the set of directed edges. We may find this decomposition, for instance, using the classic algorithm by Tarjan~\cite{tarjan1972depth}, which employs a single pass of depth-first search and whose computational complexity is $\mathcal{O}(|\mathcal V|+|\mathcal E|)$. It is worth mentioning that depending on the network sparsity, the effective computational complexity is $\mathcal O(|\mathcal V|^2)$, since $\mathcal E\subset \left( \mathcal V \times \mathcal V\right )$. Similar to Tarjan's algorithm, Dijkstra introduced the path-based algorithm to find strongly connected components and also runs in linear time (i.e., $\mathcal{O}(|\mathcal V|+|\mathcal E|)$) \cite{dijkstra1976discipline}. Finally, Kosaraju's algorithm uses two passes of depth-first search but is also upper-bounded by $\mathcal{O}(|\mathcal V|+|\mathcal E|)$ \cite{sharir1981strong}. 

Most of the newly proposed algorithms for finding the SCCs have similar computational complexity~\cite{hsu2017comparative}. A possible alternative is to develop better data structure algorithms that are suitable for parallelization, which can then lead to implementations with computational complexity equal to $\mathcal O(|\mathcal V| \log \left ( |\mathcal V|\right ))$~\cite{fleischer2000identifying} -- see also~\cite{barnat2011distributed} for an overview of different parallelized algorithms for SCC decomposition. 

The above-mentioned  solutions require knowledge of the overall structure of the system digraph, which may not be suitable for neither control systems nor for large-scale applications in machine learning, including social networks. Subsequently, we propose a scalable distributed algorithm to determine the SCCs that relies solely on control systems tools, specifically max-consensus-like dynamics. Furthermore, our algorithm converges in $D+1$ iterations and thereby enables us to determine the diameter $D$ of the network. State-of-the-art methods to determine the diameter of a directed network include the Floyd-Warshall algorithm, which has a complexity of $\mathcal{O}(|V|^3)$~\cite{floydwarshall}.

\textbf{Main contributions:}  

\begin{itemize}
    \item Provide a scalable distributed algorithm to find the strongly connected components of a directed graph with computational time-complexity~$\mathcal{O}\left(Dd_{\text{in-degree}}^{\max}\right)$; 
    \item Determine the finite diameter of a directed graph with computational time-complexity $\mathcal{O}\left(Dd_{\text{in-degree}}^{\max}+|\mathcal{V}|\right)$;
    \item Provide numerical evidence of the performance of our algorithm on random networks including Erdős-Rényi, Barabási-Albert, and Watts-Strogatz.
\end{itemize}
 
\subsection{Preliminaries and Terminology}\label{sec:notation}
Consider a directed graph (digraph) $\mathcal{G}=(\mathcal{V},\mathcal{E})$ where $\mathcal{V}$ is the set of vertices with $|\mathcal{V}| = N$, and $\mathcal{E}\subset \mathcal V\times \mathcal V$ is the set of edges, where the maximum number of edges is $|\mathcal{E}| = |\mathcal{V}\times \mathcal{V}| = N^2$. 
Given $\mathcal{G}=(\mathcal{V},\mathcal{E})$, the \textit{in-degree} of a vertex $v\in\mathcal V$ is $d_{\text{in-degree}}(v)=|\{(u,v)\,:\,(u,v)\in\mathcal E\}|$, and we denote the maximum in-degree of $\mathcal G$ by $d_{\text{in-degree}}^{\max}=\displaystyle\mathop{\max}_{v\in\mathcal V}d_{\text{in-degree}}(v)$. 
Moreover, given a vertex $v\in\mathcal V$, we define the set of its \textit{in-neighbors} as $\mathcal N_{v}^-=\{u\,:\,(u,v)\in\mathcal E\}$. 

A \textit{walk} in a digraph is any sequence of edges where the last vertex in one edge is the beginning of the next edge, except for the beginning vertex of the first edge and the ending vertex of the last edge. 
Notice that a walk does not exclude the repetition of vertices. 
In contrast, a \textit{path}, is a walk where the same vertex is not the beginning or ending of two different edges in the sequence. The size of the path is the number of edges that constitute it.  
If the beginning and ending vertex of a path is the same, then we obtain a \textit{cycle}.
Additionally, a \textit{sub-digraph} $\mathcal G_s=(\mathcal V',\mathcal E')$ is described as any sub-collection of vertices $\mathcal V'\subset \mathcal V$ and the edges $\mathcal E'\subset \mathcal{V}'\times\mathcal{V}'$ between them. 
If a subgraph has the property that there exists a path between any two pairs of vertices, then it is a \textit{strongly connected (di)graph}. 
The maximal strongly connected subgraph forms a \emph{strongly connected component (SCC)}, and any digraph can be uniquely decomposed into SCCs. A digraph $\mathcal{G}' = (\mathcal{V}',\mathcal{E}')$ is said to \textit{span} $\mathcal{G} = (\mathcal{V},\mathcal{E})$, denoted by $\mathcal{G}' = \text{span}(\mathcal{G})$, if $\mathcal{V}'=\mathcal{V}$ and $\mathcal{E}' \subseteq \mathcal{E}$.

Finally, given a digraph $\mathcal G=(\mathcal V,\mathcal E)$, we define its \emph{finite digraph diameter} $D$ as the size of the longest shortest path between any pair of vertices in $\mathcal V$, for the pairs such that such a path exists.  


\section{Problem Statement}\label{sec:prob}

We propose to address the following two problems.

\vspace{0.2cm}
$(\mathbf{P}_1)$: Given a digraph $\mathcal G=(\mathcal V,\mathcal E)$,  determine the unique decomposition of $m\in \mathbb{N}$ strongly connected components by finding the maximal subgraphs $\mathcal G_s = (\mathcal V_s,\mathcal E_s), s=1,\cdots,m$, where each subgraph is a SCC such that $\mathcal V_s\cap \mathcal V_q= \emptyset$ for $s\neq q$ with $q=1,\ldots,m$, $\mathcal V_s,\mathcal V_q\subset\mathcal V$, $\mathcal E_s\subset (\mathcal E\cap  (\mathcal{V}_s\times \mathcal{V}_s))$, and $\displaystyle\bigcup_{s=1}^m \mathcal G_s\equiv \left(\cup_{s=1}^m \mathcal{V}_s,\cup_{s=1}^m \mathcal{E}_s\right)=\text{span}(\mathcal G)$. 
\vspace{0.3cm}

$(\mathbf{P}_2)$: Given a digraph $\mathcal G=(\mathcal V,\mathcal E)$, determine the finite digraph diameter $D$. 
\vspace{0.2cm}

Next, we provide the solution to the above problems in both a centralized and distributed fashion that enables a scalable approach to determine the different SCCs and the finite digraph diameter of a given network. Notice that a graph has a unique decomposition into $m$ SCCs, but we do not require a priori the knowledge of such number.

\section{A scalable distributed dynamical systems approach to compute the strongly connected components and finite diameter of networks}

To determine a solution to ($\mathbf{P}_1$) and ($\mathbf{P}_2$), we leverage a \mbox{max-consensus-like} protocol. 

\begin{definition}\label{def:max_consensus}\cite{cortes2008distributed} 
Consider $\mathcal{G}=(\mathcal{V},\mathcal{E})$, where each vertex $v_i\in\mathcal{V},i=1,\ldots,N$, has an associated state $y_i[k]\in\mathbb{R}$ at any time $k\in\mathbb{N}$. Then, we have the following \emph{max-consensus-like} update rule 
\begin{equation}\label{eq:max_consensus}
    y_{i}[k+1] = \max\limits_{v_j\in \mathcal{N}_{v_i}^{-}\cup \{v_i\}} y_{j}[k],
\end{equation}
for each node $v_i$, where $\mathcal{N}_{v_i}^{-}$ denotes all of the nodes $v_j$ such that there is an edge $(v_j,v_i)\in\mathcal{E}$. We simply say that \textit{consensus} is achieved if there exists an instance of time $h$ such that for all $h'\geq h$,   $y_{i}[h']=y_{j}[h']$, for all $v_i,v_j\in\mathcal{V},i=\{1,\dots,N\}, j=\{1,\dots,N\}$ and for all initial conditions $y[0]=[y_{1}[0]^{\intercal} \dots y_{n}[0]^{\intercal}]^{\intercal}$.  \hfill $\circ$
\end{definition}

Definition~\ref{def:max_consensus} is similar to the max-consensus update, but in addition to considering the information from the neighbors, Definition~\ref{def:max_consensus} also considers the information from the node itself. Furthermore, it is worth emphasizing that from Definition~\ref{def:max_consensus} it follows that every node only needs be able to receive information from its in-neighbors, (i.e., the nodes connected to it). Hence, each node only needs the local information, which is pertinent to distributed algorithms.

Next, we present Algorithm \ref{algorithm4}, which can be used in finding the solutions to ($\mathbf{P}_1$) as well as ($\mathbf{P}_2$).

\begin{algorithm}[thpb]
\SetAlgoLined
\SetKwInput{KwData}{Input} 
\KwData{ $\mathcal{N}_{v_i}^{-}$, which is the set of in-neighbors of node $v_i$}
\SetKwInput{KwResult}{Output}
 \KwResult{Each node $v_i$ obtains a set  $\mathcal{S}_{i}^{*}$, which contains the nodes belonging to the same SCC as node $v_i$, and a scalar $k_i$, which is the one more than the number of iterations} 
\textbf{Initialization:}
Set $\mathcal S_{i}^*=\emptyset$, $k_i=0$; $x_{i}[0]=\{i\}$; $y_{i}[0]=1$; $z_{i}[0]=\{\};$ and  $w_{i}[0]=\text{\textsc{False}}$; \\
\While{$w_i[k_i]==\text{\textsc{False}}$}{
\textbf{Step 1:} $$x_i[k_i+1]=\displaystyle\bigcup_{v_j\in\mathcal N^-_{v_i}\cup\{v_i\}}x_{j}[k_i]$$
\textbf{Step 2:} $$y_i[k_i+1]=\displaystyle\mathop{\max}\{\displaystyle\mathop{\max}_{v_j\in\mathcal N^-_{v_i}}|x_{j}[k_i]|,|x_{i}[k_i+1]|\}$$
\textbf{Step 3:} $$z_i[k_i+1]=\left\{v_j\,:\,y_i[k_i+1]=y_j[k_i]\land  v_j\in \hspace{-4mm}\bigcup_{v_l\in\mathcal N^-_{v_i}\cup\{v_i\}}\hspace{-4mm}x_l[k_i]\right\}$$
\textbf{Step 4:} $$w_i[k_i+1]=\left(y_i[k_i+1]==y_i[k_i]\right)$$
\textbf{Step 5:} $k_i=k_i+1$
}
\textbf{Step 6:} 
Set $S^\ast_{i}=z_i[k_i]$


 \caption{Find the SCCs distributively }\label{algorithm4}
\end{algorithm}

Algorithm \ref{algorithm4} is performed on each node $v_i$ and obtains a set $\mathcal{S}_{i}^{*}$, which consists of the nodes that belong to the same SCC as node $v_i$, and a scalar $k_i$, which is one more than the number of iterations. 

Briefly speaking, Algorithm \ref{algorithm4} works as follows. For each node $v_i$, we first find the set of nodes that have a directed path ending in node $v_i$. 
Next, we compare the size of this set with the size of the sets of their neighboring nodes that are connected to node $v_i$.
Finally, we add the nodes contained in the same SCC as node $v_i$ to the set $ \mathcal  S_i^{*}$.

More specifically, Algorithm~\ref{algorithm4} starts by initializing the local (i.e., at node $v_i$) sets and parameters for the algorithm. 
At each iteration of the algorithm, Step 1 finds the set of state `ids' (or, equivalently, nodes' indices) that form directed paths that end in node $v_i$. Step 2 records the maximum size of the sets of directed paths to node $v_i$. Step 3 determines the nodes that are contained in the same SCC as node $v_i$. In Step 4, if the maximum size of the set of directed paths to $v_i$ has been obtained, then an indication to end the algorithm for node~$v_i$ is provided. Step 5 tracks the iterations, which is important for finding the finite digraph diameter. The algorithm terminates when no new information is received. Lastly, Step 6 sets $\mathcal S_i^{*}$, which is the set of nodes that are contained in the same SCC as node $v_i$. 

The following lemma will be important in proving the correctness of Algorithm~\ref{algorithm4}.
\begin{lemma}\label{lemma:findingSCCs}
If, for any two nodes $v_i$ and $v_j$, we have that $y_i[k_i+1]=y_j[k_i]$ and $v_j\in\displaystyle\mathop{\bigcup}_{v_l\in\mathcal{N}_{v_i}^{-}\cup\{v_i\}} x_{l}[k_i]$, then $v_i$ and $v_j$ are in the same SCC. 
\end{lemma}
\begin{proof}
Suppose for a contradiction that $y_i[k_i+1] = y_j[k_i]$ and $v_j\in\displaystyle\mathop{\bigcup}_{v_l\in\mathcal{N}_{v_i}^{-}\cup\{v_i\}} x_{l}[k_i]$, but $v_i$ and $v_j$ are not in the same SCC. This would mean that there is neither a direct path from $v_i$ to $v_j$ nor from $v_j$ to $v_i$. However, if $v_j\in\cup_{v_l\in\mathcal{N}_{i}^{-}\cup\{i\}} x_{l}[k_i]$, then $v_j$ can reach node $v_i$, so there is a direct path from $v_j$ to $v_i$. Furthermore, if $y_i[k_i+1]=y_j[k_i]$, then there must also be a direct path from $v_i$ to $v_j$ or we would have that $y_i[k_i+1]>y_j[k_i]$. Therefore, there is a direct path from $v_i$ to $v_j$ and from $v_j$ to $v_i$, so $v_i$ and $v_j$ must be in the same SCC.
\end{proof}

The next lemma is important in providing a stopping criteria for the algorithm. 
\begin{lemma}\label{lemma:stopping_criteria}
If $y_{i}[k_i+1] = y_i[k_i]$, for all $i=1,\dots,N$, then all of the SCCs have been found.
\end{lemma}
\begin{proof}
At each iteration of the algorithm, the number of elements in set $x_i$ either increases or it remains the same. If from one iteration to the next, the number of elements in $x_i$ stays the same for all nodes $v_i$, then every node has received all of the information that it possibly can. Furthermore, each node knows the other nodes that can reach it as captured by the set $x_i$. Since $y_i$ finds the maximum set of nodes that can reach node $v_i$ by comparing the size of set $x_j$ among all of  neighbors and itself, then if $y_i$ remains the same from one iteration to the next, then it is clear that the number of elements in $x_i$ also remains the same. Hence, the network cannot communicate any new information. This means that the set $z_i$ will also not change on the next iteration, and since $z_i$ captures the SCC containing $v_i$, then all of the SCCs must have been found. 
\end{proof}

In the following theorem, we prove the correctness of Algorithm~\ref{algorithm4}. 

\begin{theorem}\label{theorem:alg}
Let $\mathcal S_i^{*}$ be the set of nodes that results after Algorithm~\ref{algorithm4} is executed on node $v_i\in\mathcal{V}$. Then, $\displaystyle\mathop{\bigcup}_{i=\{1,\ldots,N\}}( \mathcal S_i^{*},(\mathcal S_i^{*}\times \mathcal S_i^{*})\cap\mathcal E)$ is a solution to~$\mathbf{P}_1$. \hfill $\circ$
\end{theorem}

\begin{proof}
The algorithm iterates until $y_i[k_i+1] = y_i[k_i]$, for all $i=1,\ldots,N$, at which point all of the SCCs have been found--see Lemma~\ref{lemma:stopping_criteria}. At each iteration, Step 1 forms the set of nodes that reach node $v_i$ and is recorded in set $x_{i}$. Step 2 finds the maximum cardinality of this set by comparing the size of $x_i$ among its neighbors and itself. Step 3 finds the set of nodes that are contained in the same SCC as node $v_i$ -- see Lemma~\ref{lemma:findingSCCs} and is recorded in set $z_i$. 
Step 4 determines if the maximum set of nodes that can reach node $v_i$ has been found and serves as a stopping criteria for the algorithm. Step 5 tracks the iterations, and step 6 records the set of nodes that are contained in the same SCC as node $v_i$ in the set $\mathcal S_i^{*}$. Finally, the SCCs are formed in the following subgraphs $\mathcal{G}_{i}^{*}=( \mathcal S_i^{*},(\mathcal S_i^{*}\times \mathcal S_i^{*})\cap\mathcal{E})$ as mentioned in the statement of Theorem \ref{theorem:alg}. Any duplicate subgraphs of SCCs are eliminated by taking the union of all the subgraphs $\displaystyle\bigcup_{i=\{1,\ldots,N\}}\mathcal{G}_{i}^{*}$. Hence, we obtain the SCCs of $\mathcal{G}(\mathcal{V},\mathcal{E})$. 
\end{proof}



Next, we provide the computational time-complexity for Algorithm~\ref{algorithm4}.

\begin{theorem}\label{theorem:complexity}
Algorithm \ref{algorithm4} has computational time-complexity $\mathcal{O}\left(NDd_{\text{in-degree}}^{\max}\right)$, where $N$ is the number of vertices, $D$ is the (finite) diameter of the network (i.e., the longest shortest path) and $d_{\text{in-degree}}^{\max}$ is the maximum in-degree of the network. \hfill $\circ$ 
\end{theorem}
\begin{proof}
Algorithm \ref{algorithm4} executes for all nodes $v_i\in\mathcal{V}, i=1,\dots, N$. Furthermore, Algorithm \ref{algorithm4} contains a single while loop, which is upper-bounded by the diameter since $y_{i}$ finds the longest shortest path to node $v_i$. The steps inside the while loop (i.e., steps 1, 2, and 3) are upper-bounded by the maximum in-degree of the network since we examine all of the in-neighbors. Finally, steps 4, 5, and 6 are upper-bounded by a constant. Hence, the computational time-complexity is $\mathcal{O}\left(NDd_{\text{in-degree}}^{\max}\right)$, where $N$ is the number of nodes, $D$ is the finite digraph diameter of the network, and $d_{\text{in-degree}}^{\max}$ is the maximum in-degree of the network.
\end{proof}

The following result demonstrates the scability of Algorithm~\ref{algorithm4}.

\begin{corollary}\label{corollary:distributed_complexity}
Algorithm \ref{algorithm4} can be implemented in a distributed fashion and is scalable with computational time-complexity $\mathcal{O}\left(Dd_{\text{in-degree}}^{\max}\right)$. \hfill $\circ$
\end{corollary}

\begin{proof}
This readily follows from Theorem~\ref{theorem:complexity} and from noticing that Steps 1-6 can be performed locally for each node $v_i$, where $i=1,\dots,N$. Therefore, the algorithm can be computed in a distributed fashion reducing the need to account for $N$ in the computational complexity in Theorem~\ref{theorem:complexity}.
\end{proof}

The space-complexity for performing Algorithm~\ref{algorithm4} on each node $v_i$, where $i=1,\ldots,N$, in a distributed fashion is $\mathcal{O}(|\mathcal{V}_i|)$, where $\mathcal{G}_i=(\mathcal{V}_i,\mathcal{E}_i)$ is the SCC that node $v_i$ belongs to.

In the next result, we give a solution to $(\mathbf{P}_2)$.
\begin{theorem}\label{theorem:diameter}
After executing Algorithm \ref{algorithm4} on every node $v_i\in\mathcal{V}$, where $i=\{1,\ldots,N\}$, we have that $D=\displaystyle\mathop{\max}_{v_i\in\mathcal V} k_i-2$ is a solution to $(\mathbf{P}_2)$.
\end{theorem}
\begin{proof}
We will show that Algorithm \ref{algorithm4} converges after $D+1$ iterations, where $D$ is the finite digraph diameter of the input digraph. From Lemma~\ref{lemma:findingSCCs}, the algorithm terminates when all of the SCCs have been found, which occurs when no new information is being received by any node from its neighbors (or itself) at a subsequent time step. If we assume that the digraph has diameter $D$, this implies that there exists a pair of nodes $u$ and $v$ such that the size of the shortest path between $u$ and $v$ is $D$. 
Suppose that node $v$ receives the information of node $u$ in $k<D+1$ iterations where the information travels to the neighbors of each node in one iteration. 
Then, there must be another path from $u$ to $v$ with $k-1$ edges, which contradicts the fact that the shortest path between $u$ and $v$ has size $D$. 

Now, suppose that the algorithm only converges after $k>D+1$ iterations. 
This means that there is information from a node $u$ that only reaches a node $v$ after $k$ iterations. 
However, since information is sent to the neighbors at each iteration, then the shortest path between $u$ and $v$ has size $k-1$, which again contradicts the fact that the longest shortest finite path between two nodes has size $D$. Therefore, the algorithm converges in $D+1$ iterations. Hence, the diameter will be one less than the maximum number of iterations among all nodes. Since Step 5 increments $k_i$ before terminating, then, $k_i$ denotes the number of iterations (for $v_i$) plus one. Conveniently, $D=\displaystyle\mathop{\max}_{v_i\in\mathcal V} k_i-2$ finds precisely the maximum number of iterations plus one among all nodes $v_i$ (i.e., $\max_{v_i\in\mathcal{V}}k_i)$ and subtracts two to obtain the finite digraph diameter $D$. 
\end{proof}

We emphasize that computing the finite digraph diameter requires the number of iterations for each node. Next, we give the computational time-complexity for computing the finite digraph diameter. 

\begin{theorem}\label{theorem:complexity_diameter}
Computing the finite digraph diameter requires a computational time-complexity of $\mathcal{O}\left(Dd_{\text{in-degree}}^{\max}+N\right)$. 
\end{theorem}
\begin{proof}
Following from Theorem~\ref{theorem:diameter}, we obtain the finite digraph diameter by executing Algorithm~\ref{algorithm4} on every single node $v_i\in\mathcal{V}$. Hence, from Corollary~\ref{corollary:distributed_complexity}, we see that Algorithm~\ref{algorithm4} has a computational time-complexity of $\mathcal{O}\left(Dd^{\max}_{\text{in-degree}}\right)$ when executed distributively. The final term $N$ in the complexity is added because we must determine the maximum number of iterations among all of the nodes~$v_i$, where $i=1,\dots,N$. 
\end{proof}

Finally, we explore the average computational time-complexity in some special random networks. 
\begin{corollary}\label{corollary:erdos_renyi_complexity}
For an Erdős–Rényi network with $N$ nodes and $m$ edges, the expected computational time-complexity of Algorithm~\ref{algorithm4} is $\mathcal{O}\left(\left(\frac{\log(N)-\gamma}{\log(2m/N)}+\frac{1}{2}\right)\frac{2m}{N}\right)$, where $\gamma$ is the Euler-Mascheroni constant.
\end{corollary}
\begin{proof}
The average degree of an Erdős–Rényi network is $\frac{2m}{N}$, and the average path length is $\frac{\log(N)-\gamma}{\log(2m/N)}+\frac{1}{2}$ \cite{avgpathlengthER}. Thus, by Corollary~\ref{corollary:distributed_complexity}, the average time-complexity is $\mathcal{O}\left(\left(\frac{\log(N)-\gamma}{\log(2m/N)}+\frac{1}{2}\right)\frac{2m}{N}\right)$, where~$\gamma$ is the Euler-Mascheroni constant.
\end{proof}

\begin{corollary}\label{corollary:barabasi_albert_complexity}
For a Barabási-Albert network with $N$ nodes and $m$ edges added to a new vertex at each step, the expected time-complexity of Algorithm \ref{algorithm4} is $\mathcal{O}\left(2m\left(\frac{\log(N)-\log(m/2)-1-\gamma}{\log(\log(N))+\log(m/2)}+\frac{3}{2}\right)\right)$. 
\end{corollary}
\begin{proof}
Since the average degree is $2m$, and the average path length is $\frac{\log(N)-\log(m/2)-1-\gamma}{\log(\log(N))+\log(m/2)}+\frac{3}{2}$ \cite{avgpathlengthER}, by Corollary \ref{corollary:distributed_complexity}, the average \mbox{time-complexity} reduces to $\mathcal{O}\left(2m\left(\frac{\log(N)-\log(m/2)-1-\gamma}{\log(\log(N))+\log(m/2)}+\frac{3}{2}\right)\right)$.
\end{proof}

\begin{corollary}\label{corollary:watts_strogatz_complexity}
For a Watts-Strogatz network with $N$ nodes, $K$ edges per vertex, and rewiring probability $p$, the expected time-complexity of Algorithm \ref{algorithm4} is $\mathcal{O}(N/2)$ as $p \rightarrow 0$ and $\mathcal{O}\left(K\log(N)/\log(K)\right)$ as $p \rightarrow 1$. 
\end{corollary}
\begin{proof}
The Watts-Strogatz network has an average degree of $K$, and the average path length is $\frac{N}{2K}$ as $p\rightarrow 0$ and $\frac{\log(N)}{\log(K)}$as $p\rightarrow 1$ \cite{watts1998collective}. Hence, by Corollary~\ref{corollary:distributed_complexity}, the average time-complexity of Algorithm \ref{algorithm4} for the Watts-Strogatz network is $\mathcal{O}(N/2)$ as $p\rightarrow 0$ and $\mathcal{O}(K\log(N)/\log(K))$ as $p\rightarrow 1$.
\end{proof}


\section{Pedagogical Examples}\label{sec:examples}

In this section, we present several pedagogical examples to illustrate how Algorithm \ref{algorithm4} works and demonstrate its computational complexity. In what follows, when referring to each of the SCCs, we will only mention the indices of the nodes contained in that particular SCC (i.e., if $v_i\in\mathcal{V}_s$ then with some abuse of notation we refer to that node as $i\in\mathcal{V}_s$) as we are implicitly assuming that their edges are formed by  $\mathcal{E}_s = ((\mathcal{V}_s\times \mathcal{V}_s)\cap \mathcal{E)}$.

\subsection{Example 1}
Figure \ref{fig:example2} shows a network with six nodes that contains the following strongly connected components: $\{1,2\},\{3,4\},$ and $\{5,6\}$. 

\begin{figure}[H]
    \centering
\begin{tikzpicture}[scale=.6, transform shape,node distance=1.5cm]
\begin{scope}[every node/.style={circle,thick,draw},square/.style={regular polygon,regular polygon sides=4}]
\node (1) at (0.,0.) {\large $1$};
\node (2) at (2.2748,0.) {\large $2$};
\node (3) at (4.92837,0.) {\large $3$};
\node (4) at (7.68334,0.) {\large $4$};
\node (5) at (10.3395,0.) {\large $5$};
\node (6) at (12.617,0.) {\large $6$};
\end{scope}
\begin{scope}[>={Stealth[black]},
              every edge/.style={draw=black, thick}]
\path [->] (1) edge[bend right=15] node {} (2);
\path [->] (2) edge[bend right=15] node {} (1);
\path [->] (2) edge node {} (3);
\path [->] (3) edge[bend right=15] node {} (4);
\path [->] (4) edge[bend right=15] node {} (3);
\path [->] (4) edge node {} (5);
\path [->] (5) edge[bend right=15] node {} (6);
\path [->] (6) edge[bend right=15] node {} (5);
\end{scope}
\end{tikzpicture}
    \caption{This network has SCCs $\{5,6\},\{3,4\},$ and $\{1,2\}$.}
    \label{fig:example2}
\end{figure}
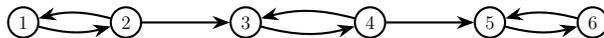
Table \ref{tab:example2} shows the trace of running Algorithm \ref{algorithm4} on Example \ref{fig:example2} for each node $v_i$, where $\mathcal{P}$ is the set of parameters for the algorithm and~$k$ is the total number of iterations. Table \ref{tab:example2} shows in column one that it takes six iterations ($k=6$) to identify the SCCs for Example \ref{fig:example2}. Here, the diameter of the network is 5, which is one less than the total required iterations and is consistent with the results in Theorem~\ref{theorem:diameter}.

\begin{table}[h!]
{\footnotesize
\centering
    \begin{tabular*}{\columnwidth}{@{\extracolsep{\fill}}c@{\extracolsep{\fill}}|@{\extracolsep{\fill}}c|@{\extracolsep{\fill}}c@{\extracolsep{\fill}}c@{\extracolsep{\fill}}c@{\extracolsep{\fill}}c@{\extracolsep{\fill}}c@{\extracolsep{\fill}}c@{\extracolsep{\fill}}}
 ${k}$ & ${\mathcal{P}}$ & ${v_1}$ & ${v_2}$ & ${v_3}$ & ${v_4}$ & ${v_5}$ & ${v_6}$ \\ \hline
 \multirow{4}{1em}{0} & \text{x[0]} & \{1\} & \{2\} & \{3\} & \{4\} & \{5\} & \{6\} \\
 & \text{y[0]} & 1 & 1 & 1 & 1 & 1 & 1 \\
 & \text{z[0]} & \{\} & \{\} & \{\} & \{\} & \{\} & \{\} \\
 & \text{w[0]} & \textsc{False} & \textsc{False} & \textsc{False} & \textsc{False} & \textsc{False} & \textsc{False} \\\hline
 
\multirow{4}{1em}{1} & \text{x[1]} & \{1,2\} & \{1,2\} & \{2,3,4\} & \{3,4\} & \{4,5,6\} & \{5,6\} \\
 & \text{y[1]} & 2 & 2 & 3 & 2 & 3 & 2 \\
 & \text{z[1]} & \{\} & \{\} & \{\} & \{\} & \{\} & \{\} \\
 & \text{w[1]} & \textsc{False} & \textsc{False} & \textsc{False} & \textsc{False} & \textsc{False} & \textsc{False} \\\hline

\multirow{4}{1em}{2} & \text{x[2]} & \{1,2\} & \{1,2\} & \{1,2,3,4\} & \{2,3,4\} & \{3,4,5,6\} & \{4,5,6\} \\
& \text{y[2]} & 2 & 2 & 4 & 3 & 4 & 3 \\
& \text{z[2]} & \{1,2\} & \{1,2\} & \{\} & \{3\} & \{\} & \{5\} \\
& \text{w[2]} & \textsc{True} & \textsc{True} & \textsc{False} & \textsc{False} & \textsc{False} & \textsc{False} \\\hline
\multirow{4}{1em}{3} & \text{x[3]} & \{1,2\} & \{1,2\} & \{1,2,3,4\} & \{1,2,3,4\} & \{2,3,4,5,6\} & \{3,4,5,6\} \\
& \text{y[3]} & 2 & 2 & 4 & 4 & 5 & 4 \\
& \text{z[3]} & \{1,2\} & \{1,2\} & \{3\} & \{3\} & \{\} & \{3,5\} \\
& \text{w[3]} & \textsc{True} & \textsc{True} & \textsc{True} & \textsc{False} & \textsc{False} & \textsc{False} \\\hline
\multirow{4}{1em}{4} & \text{x[4]} & \{1,2\} & \{1,2\} & \{1,2,3,4\} & \{1,2,3,4\} & \{1,2,3,4,5,6\} & \{2,3,4,5,6\} \\
& \text{y[4]} & 2 & 2 & 4 & 4 & 6 & 5 \\
& \text{z[4]} & \{1,2\} & \{1,2\} & \{3,4\} & \{3,4\} & \{\} & \{5\} \\
& \text{w[4]} & \textsc{True} & \textsc{True} & \textsc{True} & \textsc{True} & \textsc{False} & \textsc{False} \\\hline
\multirow{4}{1em}{5} & \text{x[5]} & \{1,2\} & \{1,2\} & \{1,2,3,4\} & \{1,2,3,4\} & \{1,2,3,4,5,6\} & \{1,2,3,4,5,6\} \\
& \text{y[5]} & 2 & 2 & 4 & 4 & 6 & 6 \\
& \text{z[5]} & \{1,2\} & \{1,2\} & \{3,4\} & \{3,4\} & \{5\} & \{5\} \\
& \text{w[5]} & \textsc{True} & \textsc{True} & \textsc{True} & \textsc{True} & \textsc{True} & \textsc{False} \\\hline
\multirow{4}{1em}{6} & \text{x[6]} & \{1,2\} & \{1,2\} & \{1,2,3,4\} & \{1,2,3,4\} & \{1,2,3,4,5,6\} & \{1,2,3,4,5,6\} \\
& \text{y[6]} & 2 & 2 & 4 & 4 & 6 & 6 \\
& \text{z[6]} & \{1,2\} & \{1,2\} & \{3,4\} & \{3,4\} & \{5,6\} & \{5,6\} \\
& \text{w[6]} & \textsc{True} & \textsc{True} & \textsc{True} & \textsc{True} & \textsc{True} & \textsc{True} \\\hline
    \end{tabular*}
    }
    \caption{This table enumerates the values of the parameters ($\mathcal{P}$) at each iteration ($k$) of Algorithm \ref{algorithm4} for all nodes $v_i$ when executed on Example~1. }
    \label{tab:example2}
\end{table}



\vspace{-0.2cm}
\subsection{Example 2: Complete Network}
Figure \ref{fig:complete} shows a complete network with five nodes, so there is a single SCC containing all of the nodes (i.e., $\{1,2,3,4,5\}$). In Table~\ref{tab:Complete}, we see that only two iterations are necessary as this is one more than the diameter of the network -- see Theorem~\ref{theorem:diameter}. 

\begin{figure}[H]
\centering
\begin{tikzpicture}[scale=.6, transform shape,node distance=1.5cm]
\begin{scope}[every node/.style={circle,thick,draw},square/.style={regular polygon,regular polygon sides=4}]
\node (1) at (-2.37764,0.772542) {\large $1$};
\node (2) at (-1.46946,-2.02254) {\large $2$};
\node (3) at (1.46946,-2.02254) {\large $3$};
\node (4) at (2.37764,0.772542) {\large $4$};
\node (5) at (0.,2.5) {\large $5$};
\end{scope}
\begin{scope}[>={Stealth[black]},
              every edge/.style={draw=black, thick}]
\path [->] (1) edge[bend right=15] node {} (2);
\path [->] (1) edge[bend right=15] node {} (3);
\path [->] (1) edge[bend right=15] node {} (4);
\path [->] (1) edge[bend right=15] node {} (5);
\path [->] (2) edge[bend right=15] node {} (1);
\path [->] (2) edge[bend right=15] node {} (3);
\path [->] (2) edge[bend right=15] node {} (4);
\path [->] (2) edge[bend right=15] node {} (5);
\path [->] (3) edge[bend right=15] node {} (1);
\path [->] (3) edge[bend right=15] node {} (2);
\path [->] (3) edge[bend right=15] node {} (4);
\path [->] (3) edge[bend right=15] node {} (5);
\path [->] (4) edge[bend right=15] node {} (1);
\path [->] (4) edge[bend right=15] node {} (2);
\path [->] (4) edge[bend right=15] node {} (3);
\path [->] (4) edge[bend right=15] node {} (5);
\path [->] (5) edge[bend right=15] node {} (1);
\path [->] (5) edge[bend right=15] node {} (2);
\path [->] (5) edge[bend right=15] node {} (3);
\path [->] (5) edge[bend right=15] node {} (4);
\end{scope}
\end{tikzpicture}    
\caption{The complete network contains a single SCC, which is made up of all of the nodes in the network (i.e., $\{1,2,3,4,5\}$).}
    \label{fig:complete}
\end{figure}
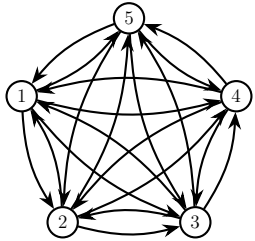

\begin{table}[h!]
    \centering
    {\footnotesize
    \begin{tabular*}{\columnwidth}{@{\extracolsep{\fill}}c@{\extracolsep{\fill}}|@{\extracolsep{\fill}}c|@{\extracolsep{\fill}}c@{\extracolsep{\fill}}c@{\extracolsep{\fill}}c@{\extracolsep{\fill}}c@{\extracolsep{\fill}}c}
 $k$ & $\mathcal{P}$ & $v_1$ & $v_2$ & $v_3$ & $v_4$ & $v_5$ \\ \hline
    \multirow{4}{1em}{0} & \text{x[0]} & \{1\} & \{2\} & \{3\} & \{4\} & \{5\} \\
& \text{y[0]} & 1 & 1 & 1 & 1 & 1 \\
& \text{z[0]} & \{\} & \{\} & \{\} & \{\} & \{\} \\
& \text{w[0]} & \textsc{False} & \textsc{False} & \textsc{False} & \textsc{False} & \textsc{False} \\\hline
\multirow{4}{1em}{1} & \text{x[1]} & \{1,2,3,4,5\} & \{1,2,3,4,5\} & \{1,2,3,4,5\} & \{1,2,3,4,5\} & \{1,2,3,4,5\} \\
& \text{y[1]} & 5 & 5 & 5 & 5 & 5 \\
& \text{z[1]} & \{\} & \{\} & \{\} & \{\} & \{\} \\
& \text{w[1]} & \textsc{False} & \textsc{False} & \textsc{False} & \textsc{False} & \textsc{False} \\\hline
\multirow{4}{1em}{2} & \text{x[2]} & \{1,2,3,4,5\} & \{1,2,3,4,5\} & \{1,2,3,4,5\} & \{1,2,3,4,5\} & \{1,2,3,4,5\} \\
& \text{y[2]} & 5 & 5 & 5 & 5 & 5 \\
& \text{z[2]} & \{1,2,3,4,5\} & \{1,2,3,4,5\} & \{1,2,3,4,5\} & \{1,2,3,4,5\} & \{1,2,3,4,5\} \\
& \text{w[2]} & \textsc{True} & \textsc{True} & \textsc{True} & \textsc{True} & \textsc{True} \\\hline
    \end{tabular*}
    }
    \caption{This table enumerates the values of the parameters ($\mathcal{P}$) at each iteration ($k$) of Algorithm \ref{algorithm4} for all nodes $v_i$ when executed on the complete network.}
    \label{tab:Complete}
\end{table}
\vspace{-0.6cm}
\subsection{Example 3: Tree}
Figure \ref{fig:Tree} shows a tree with nine nodes, so the SCCs are the individual nodes themselves (i.e., $\{1\},\{2\},\{3\},\{4\},\{5\},\{6\},\{7\},\{8\}$, and $\{9\}$). 

\begin{figure}[H]
    \centering
\begin{tikzpicture}[scale=.6, transform shape,node distance=1.5cm]
\begin{scope}[every node/.style={circle,thick,draw},square/.style={regular polygon,regular polygon sides=4}]
\node (1) at (5.,4.5) {\large $1$};
\node (2) at (3.,3.) {\large $2$};
\node (3) at (7.,3.) {\large $3$};
\node (4) at (1.,1.5) {\large $4$};
\node (5) at (3.,1.5) {\large $5$};
\node (6) at (5.,1.5) {\large $6$};
\node (7) at (7.,1.5) {\large $7$};
\node (8) at (0.,0.) {\large $8$};
\node (9) at (2.,0.) {\large $9$};
\end{scope}
\begin{scope}[>={Stealth[black]},every edge/.style={draw=black, thick}]
\path [->] (1) edge node {} (2);
\path [->] (1) edge node {} (3);
\path [->] (2) edge node {} (4);
\path [->] (2) edge node {} (5);
\path [->] (2) edge node {} (6);
\path [->] (3) edge node {} (7);
\path [->] (4) edge node {} (8);
\path [->] (4) edge node {} (9);
\end{scope}
\end{tikzpicture}
    \caption{The SCCs of the tree are the individual nodes themselves (i.e., $\{1\},\{2\},\{3\},\{4\},\{5\},\{6\},\{7\},\{8\},$ and $\{9\}$).}
    \label{fig:Tree}
\end{figure}
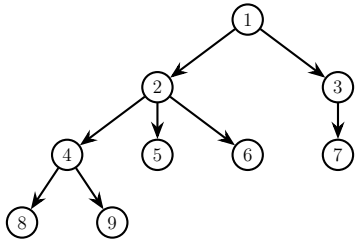

\begin{table}[h!]
    \centering
    \resizebox{.99\columnwidth}{!}{%
    \begin{tabular*}{\columnwidth}{@{\extracolsep{\fill}}c@{\extracolsep{\fill}}|c|@{\extracolsep{\fill}}c@{\extracolsep{\fill}}c@{\extracolsep{\fill}}c@{\extracolsep{\fill}}c@{\extracolsep{\fill}}c@{\extracolsep{\fill}}c@{\extracolsep{\fill}}c@{\extracolsep{\fill}}c@{\extracolsep{\fill}}c}
  $k$ & $\mathcal{P}$ & $v_1$ & $v_2$ & $v_3$ & $v_4$ & $v_5$ & $v_6$ & $v_7$ & $v_8$ & $v_9$ \\ \hline
\multirow{4}{1em}{0} & \text{x[0]} & \{1\} & \{2\} & \{3\} & \{4\} & \{5\} & \{6\} & \{7\} & \{8\} & \{9\} \\
& \text{y[0]} & 1 & 1 & 1 & 1 & 1 & 1 & 1 & 1 & 1 \\
& \text{z[0]} & \{\} & \{\} & \{\} & \{\} & \{\} & \{\} & \{\} & \{\} & \{\} \\
& \text{w[0]} & \textsc{False} & \textsc{False} & \textsc{False} & \textsc{False} & \textsc{False} & \textsc{False} & \textsc{False} & \textsc{False} & \textsc{False} \\\hline
\multirow{4}{1em}{1} & \text{x[1]} & \{1\} & \{1,2\} & \{1,3\} & \{2,4\} & \{2,5\} & \{2,6\} & \{3,7\} & \{4,8\} & \{4,9\} \\
& \text{y[1]} & 1 & 2 & 2 & 2 & 2 & 2 & 2 & 2 & 2 \\
& \text{z[1]} & \{1\} & \{\} & \{\} & \{\} & \{\} & \{\} & \{\} & \{\} & \{\} \\
& \text{w[1]} & \textsc{True} & \textsc{False} & \textsc{False} & \textsc{False} & \textsc{False} & \textsc{False} & \textsc{False} & \textsc{False} & \textsc{False} \\\hline
\multirow{4}{1em}{2} & \text{x[2]} & \{1\} & \{1,2\} & \{1,3\} & \{1,2,4\} & \{1,2,5\} & \{1,2,6\} & \{1,3,7\} & \{2,4,8\} & \{2,4,9\} \\
& \text{y[2]} & 1 & 2 & 2 & 3 & 3 & 3 & 3 & 3 & 3 \\
& \text{z[2]} & \{1\} & \{2\} & \{3\} & \{\} & \{\} & \{\} & \{\} & \{\} & \{\} \\
& \text{w[2]} & \textsc{True} & \textsc{True} & \textsc{True} & \textsc{False} & \textsc{False} & \textsc{False} & \textsc{False} & \textsc{False} & \textsc{False} \\\hline
\multirow{4}{1em}{3} & \text{x[3]} & \{1\} & \{1,2\} & \{1,3\} & \{1,2,4\} & \{1,2,5\} & \{1,2,6\} & \{1,3,7\} & \{1,2,4,8\} & \{1,2,4,9\} \\
& \text{y[3]} & 1 & 2 & 2 & 3 & 3 & 3 & 3 & 4 & 4 \\
& \text{z[3]} & \{1\} & \{2\} & \{3\} & \{4\} & \{5\} & \{6\} & \{7\} & \{\} & \{\} \\
& \text{w[3]} & \textsc{True} & \textsc{True} & \textsc{True} & \textsc{True} & \textsc{True} & \textsc{True} & \textsc{True} & \textsc{False} & \textsc{False} \\\hline
\multirow{4}{1em}{4} & \text{x[4]} & \{1\} & \{1,2\} & \{1,3\} & \{1,2,4\} & \{1,2,5\} & \{1,2,6\} & \{1,3,7\} & \{1,2,4,8\} & \{1,2,4,9\} \\
& \text{y[4]} & 1 & 2 & 2 & 3 & 3 & 3 & 3 & 4 & 4 \\
& \text{z[4]} & \{1\} & \{2\} & \{3\} & \{4\} & \{5\} & \{6\} & \{7\} & \{8\} & \{9\} \\
& \text{w[4]} & \textsc{True} & \textsc{True} & \textsc{True} & \textsc{True} & \textsc{True} & \textsc{True} & \textsc{True} & \textsc{True} & \textsc{True} \\\hline
    \end{tabular*}
    }
    \caption{This table enumerates the values of the parameters ($\mathcal{P}$) at each iteration ($k$) of Algorithm \ref{algorithm4} for all nodes $v_i$ when executed on the tree network.}
    \label{tab:Tree}
\end{table}
In Table \ref{tab:Tree}, we see that four iterations are required to identify the SCCs of the tree network, which is one more than the diameter of the network -- see Theorem~\ref{theorem:diameter}. 




\section{Simulation Results}\label{sec:sim}

In this section, we highlight the performance of our algorithm by contrasting it with one of the most frequently used algorithms among the current state-of-the-art. Specifically, we compare the run times of our algorithm against Kosaraju's algorithm \cite{sharir1981strong} on a series of random networks, including the Erdős–Rényi, Barabási–Albert, and Watts-Strogatz networks. We show how the run times of the two algorithms vary as different parameters of the networks change including the diameter, the maximum in-degree, the number of SCCs, and the number of nodes.  

We ran all the algorithms using Mathematica on a  Dell laptop with an Intel(R) Core(TM) i7-7500U CPU running at 2.70GHz with 12.0GB RAM. For each type of random network (i.e., Erdős-Rényi, Barabási-Albert, and Watts-Strogatz), we randomly generated 50 networks in the following manner. For five different sets of nodes, we randomly generated ten different networks, where the sets of nodes were 100, 200, 300, 400, and 500 nodes. Furthermore, to generate the random networks, we selected two different sets of parameters for each type of random network. 
\subsection{Erdős-Rényi}
\begin{figure}[H]
    \centering
    \subfigure[]{
    \includegraphics[width=.46\linewidth]{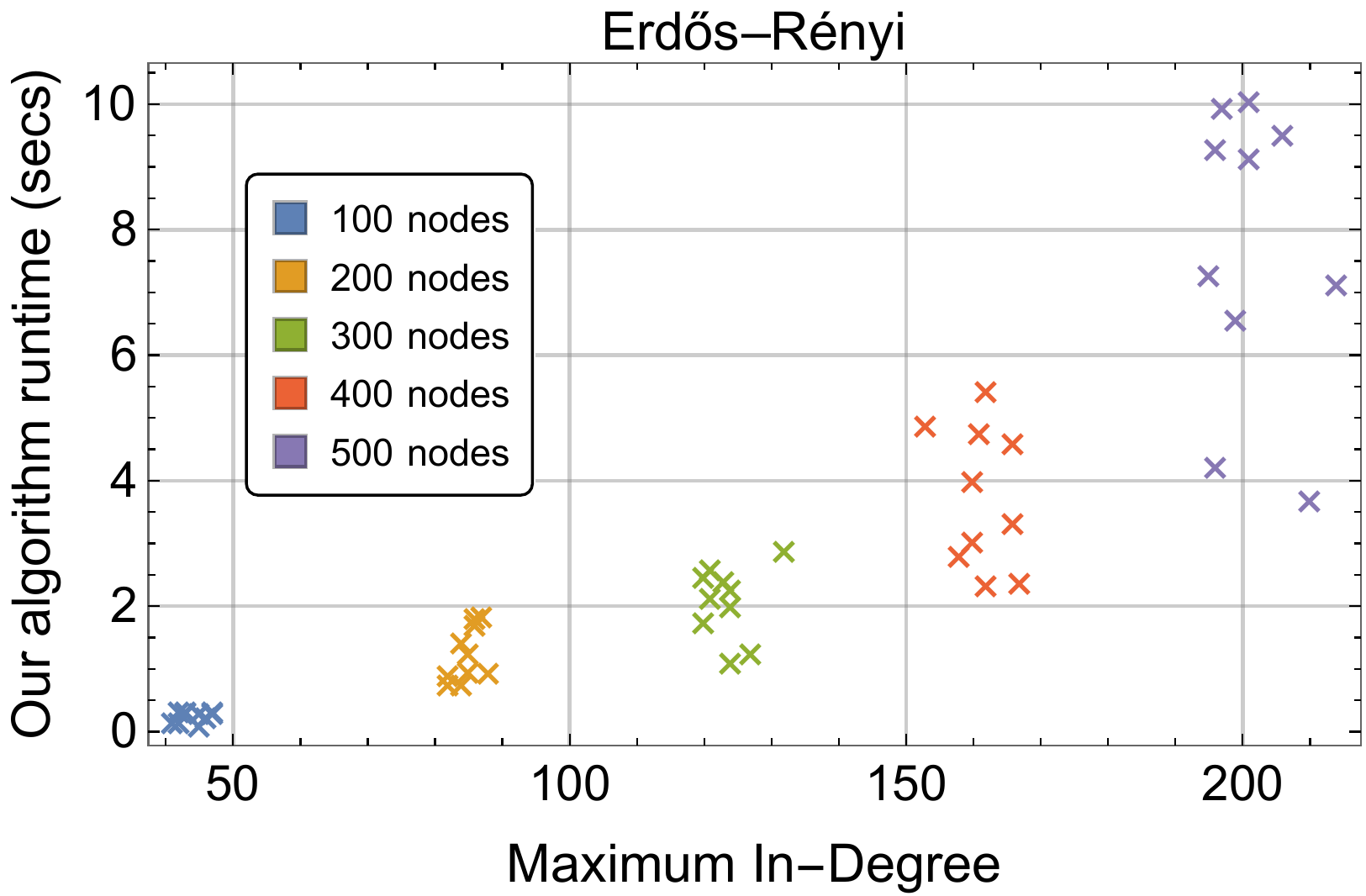}
    }
    \subfigure[]{
    \includegraphics[width=.46\linewidth]{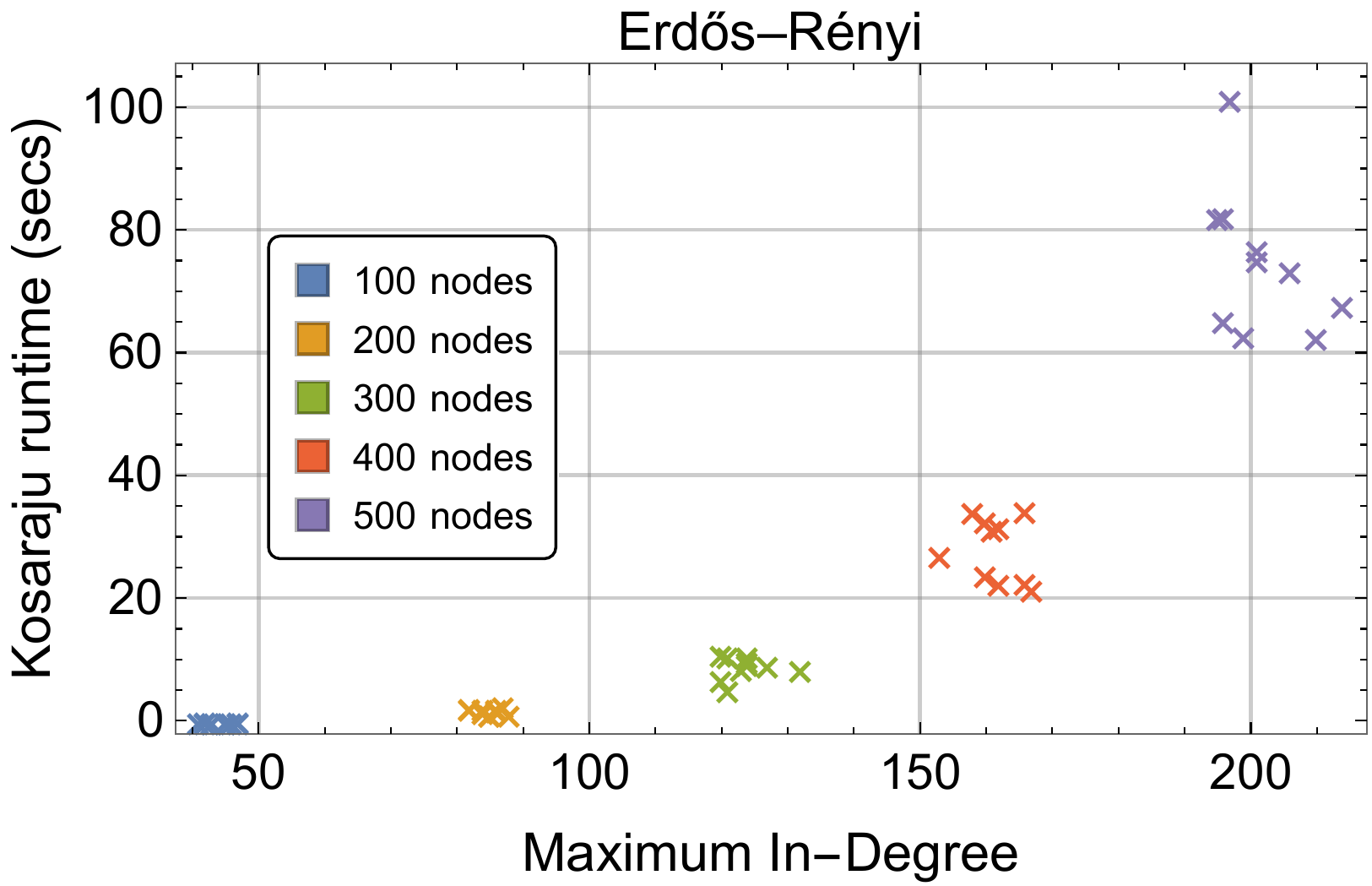}
    }
    \subfigure[]{
    \includegraphics[width=.46\linewidth]{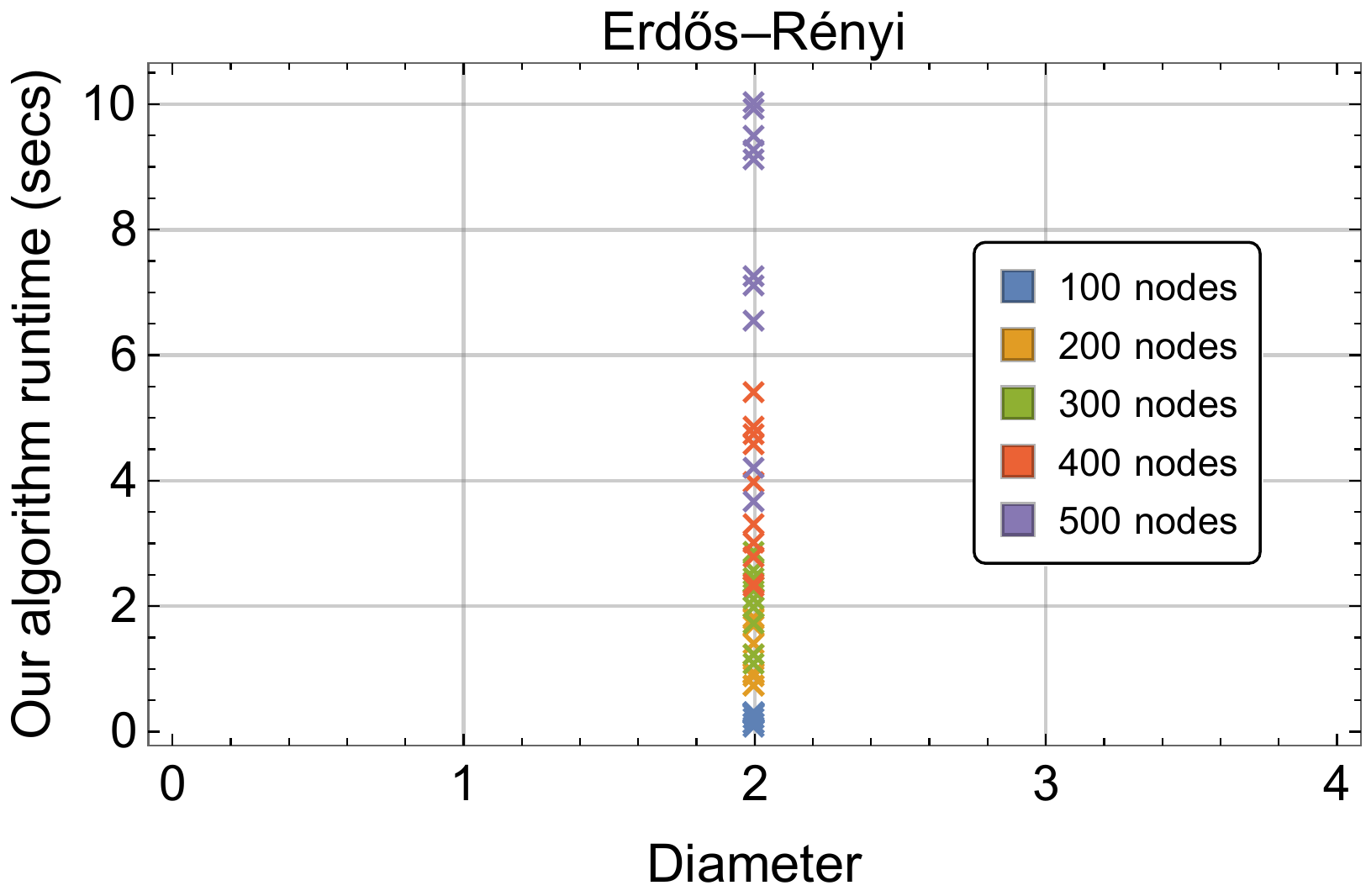}
    }
    \subfigure[]{
    \includegraphics[width=.46\linewidth]{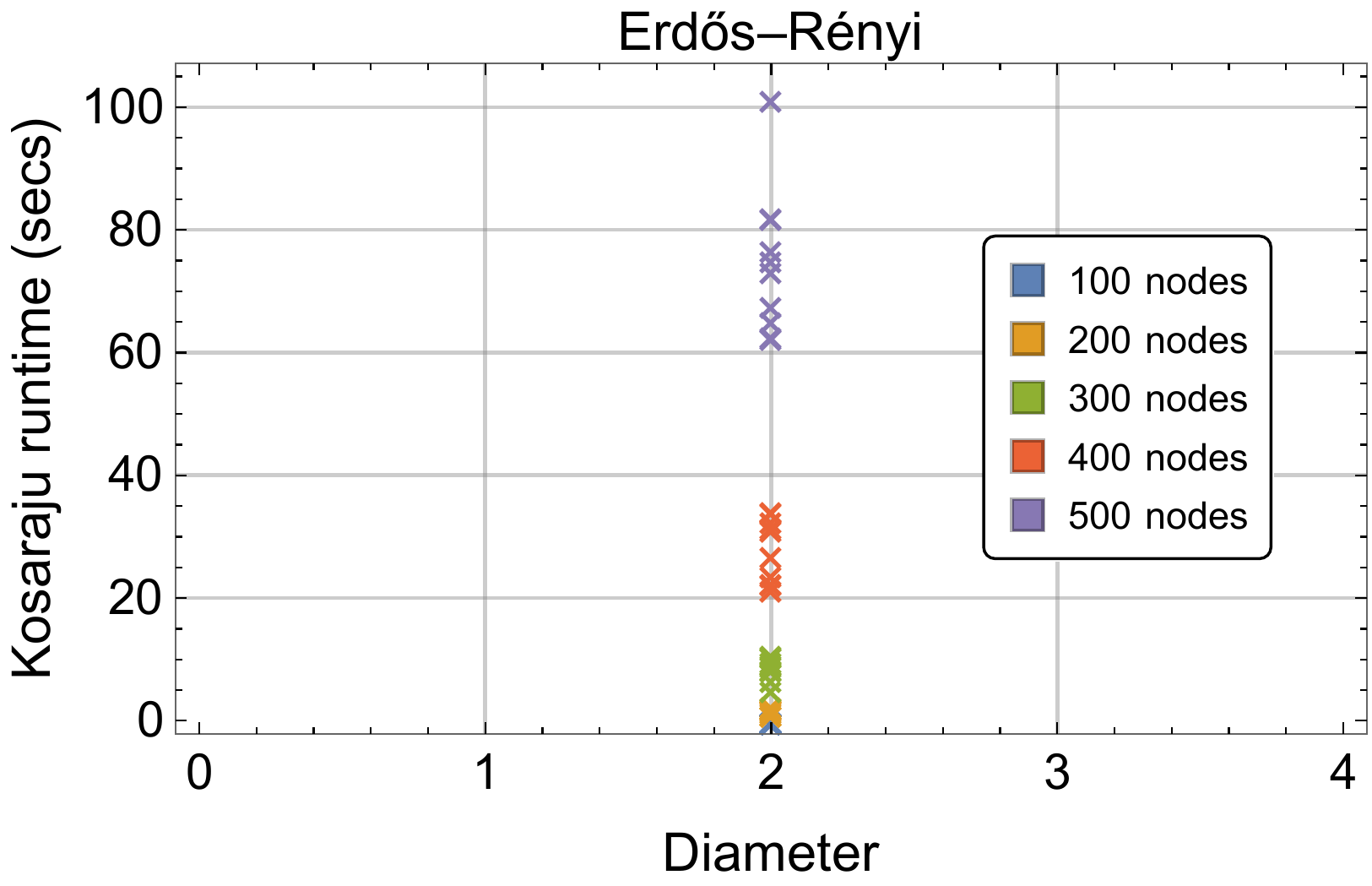}
    }
    \subfigure[]{
    \includegraphics[width=.46\linewidth]{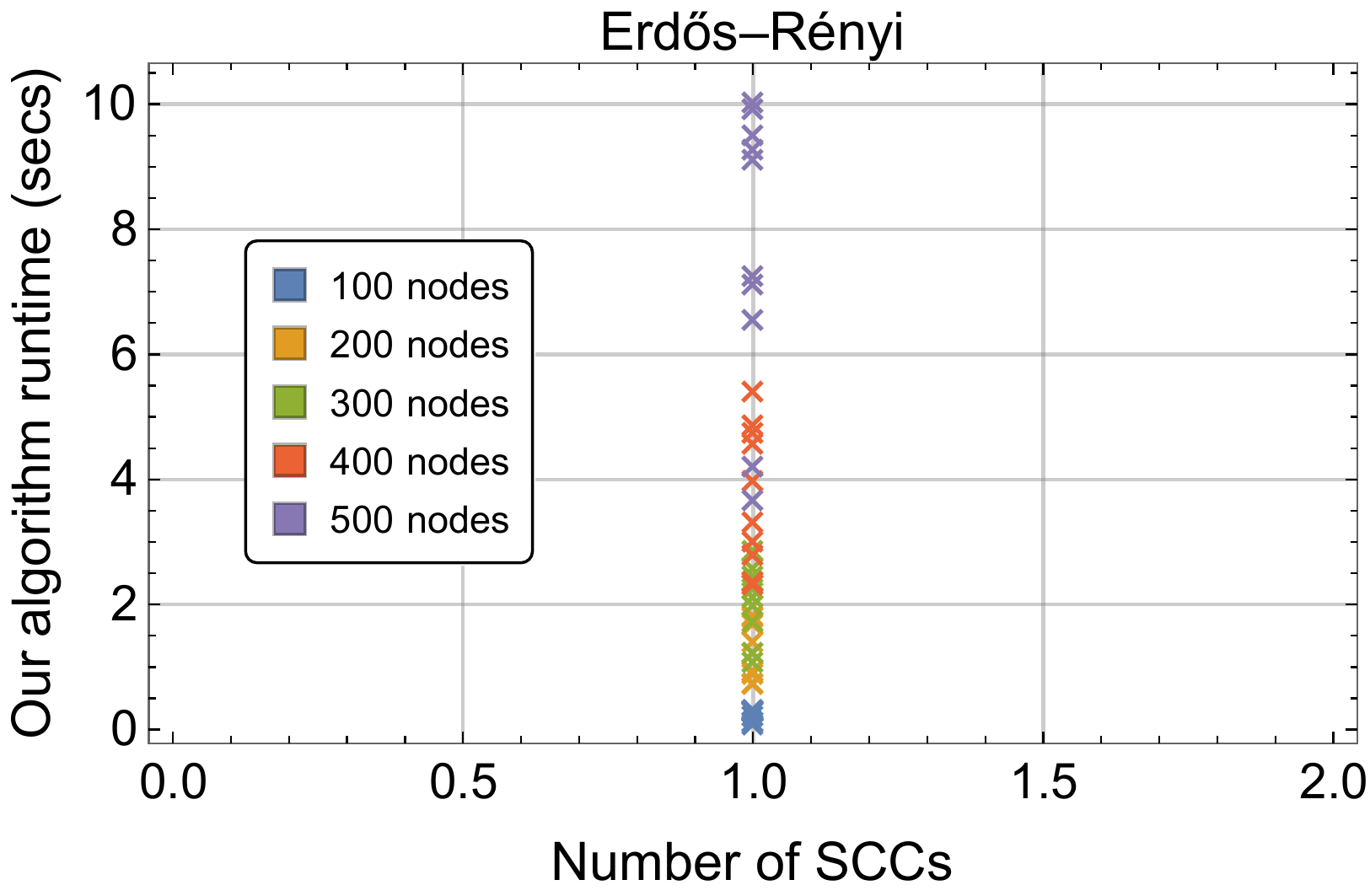}
    }
    \subfigure[]{
    \includegraphics[width=.46\linewidth]{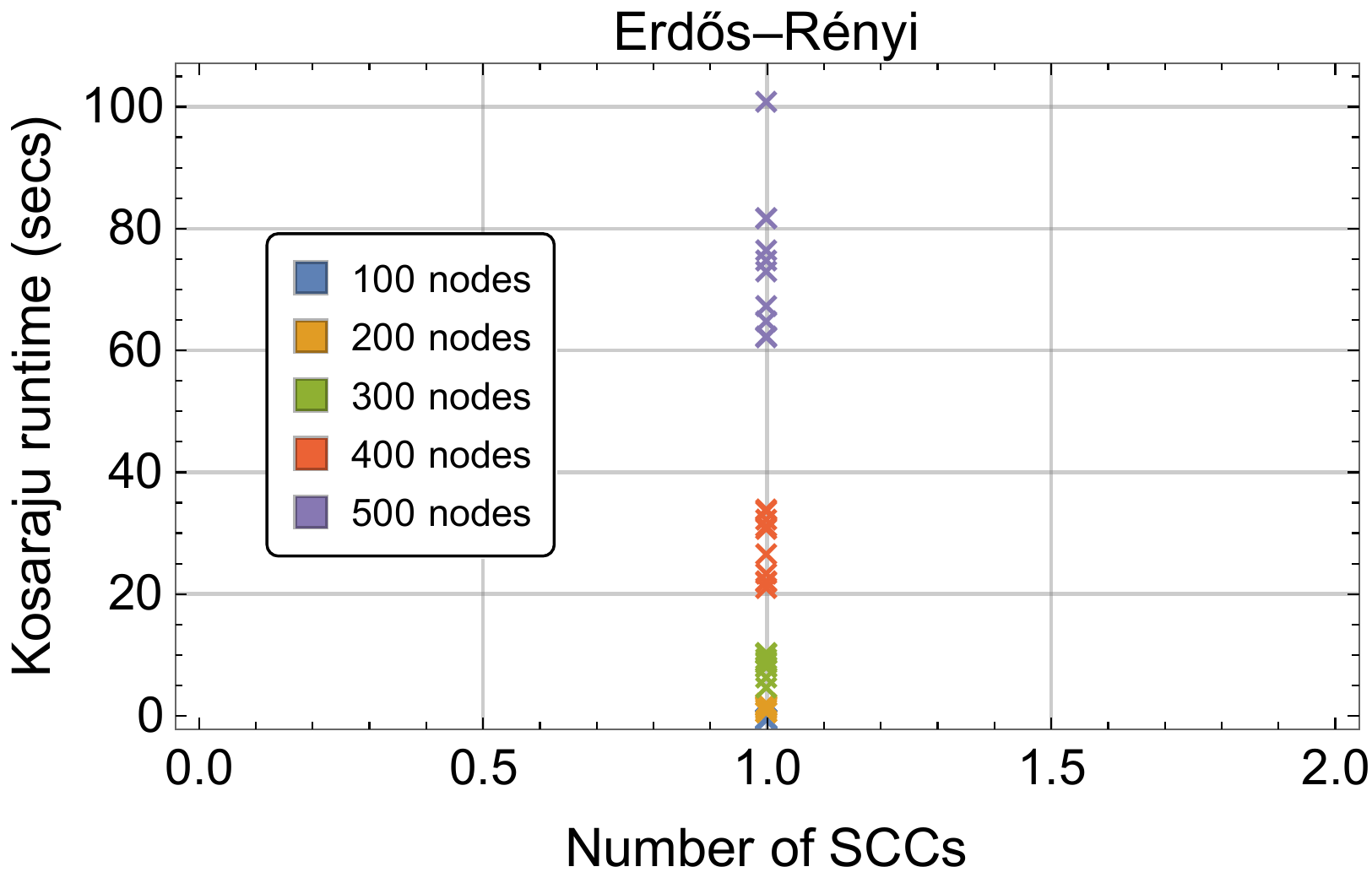}
    }
    \caption{These figures show the relationship between the network properties of some randomly generated Erdős-Rényi networks and their run times for both our proposed algorithm and Kosaraju's algorithm.}
    \label{fig:ERparam1}
\end{figure}
The Erdős-Rényi network requires two parameters, including the number of nodes and the number of edges. For the first set of parameters (Figures \ref{fig:ERparam1} and \ref{fig:ERruntimesparam1}), the number of nodes were chosen to be 100, 200, 300, 400, 500, and the  number of edges were chosen to be the number of nodes raised to the $2/3$ power. In the second set of parameters (Figures \ref{fig:ERparam2} and \ref{fig:ERruntimesparam2}), again the number of nodes remained the same, but the number of edges was fixed to 500 for all the sets of nodes.      


\begin{figure}[H]
    \centering
    \includegraphics[width =0.46 \linewidth]{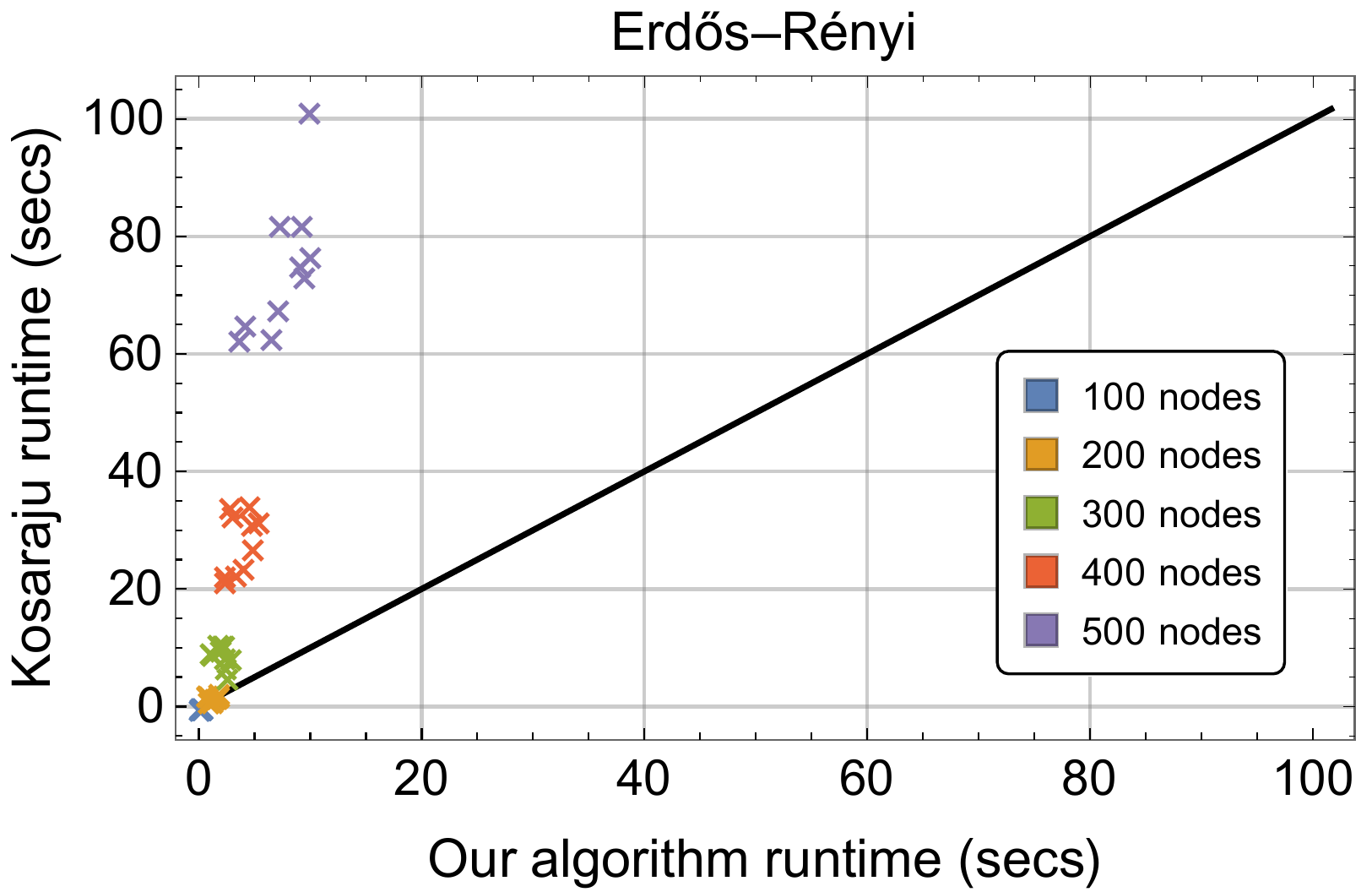}
    \caption{This figure compares the run times of  both our proposed algorithm and Kosaraju's algorithm for several randomly generated Erdős-Rényi networks. We see that our algorithm performs better on networks with a higher number of nodes.}
    \label{fig:ERruntimesparam1}
\end{figure}

\begin{figure}[H]
    \centering
    \subfigure[]{
    \includegraphics[width=.46\linewidth]{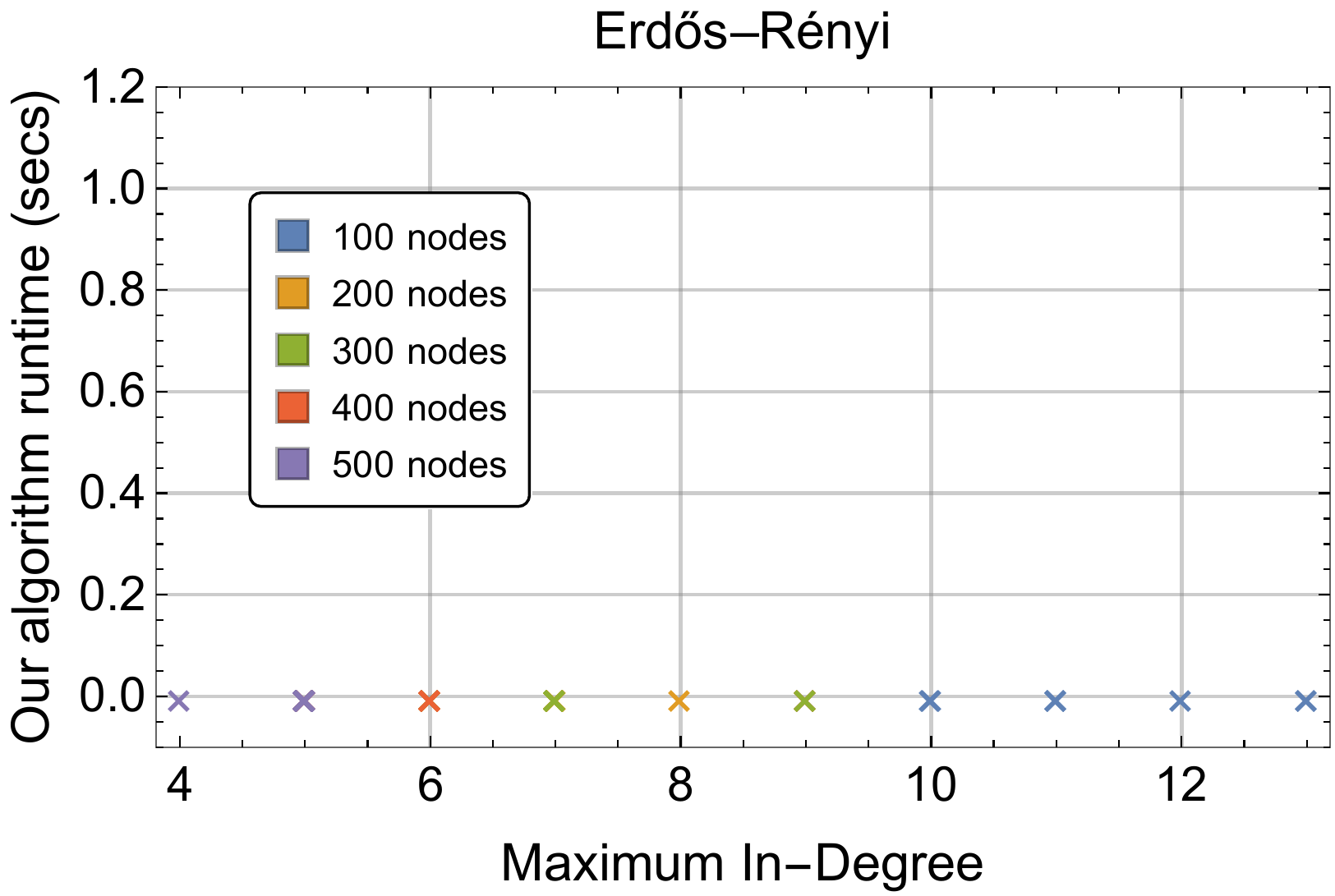}
    }
    \subfigure[]{
    \includegraphics[width=.46\linewidth]{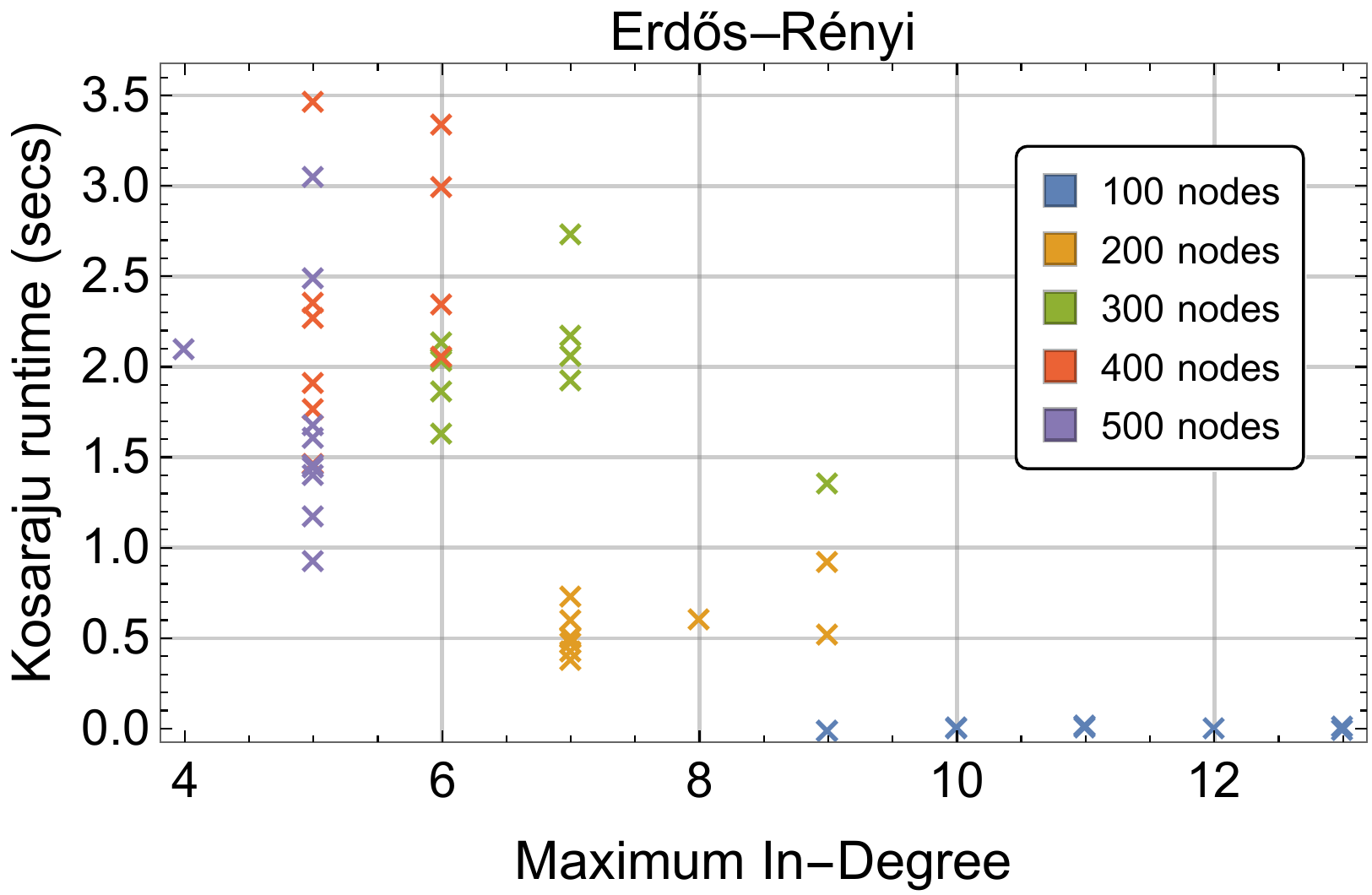}
    }
    \subfigure[]{
    \includegraphics[width=.46\linewidth]{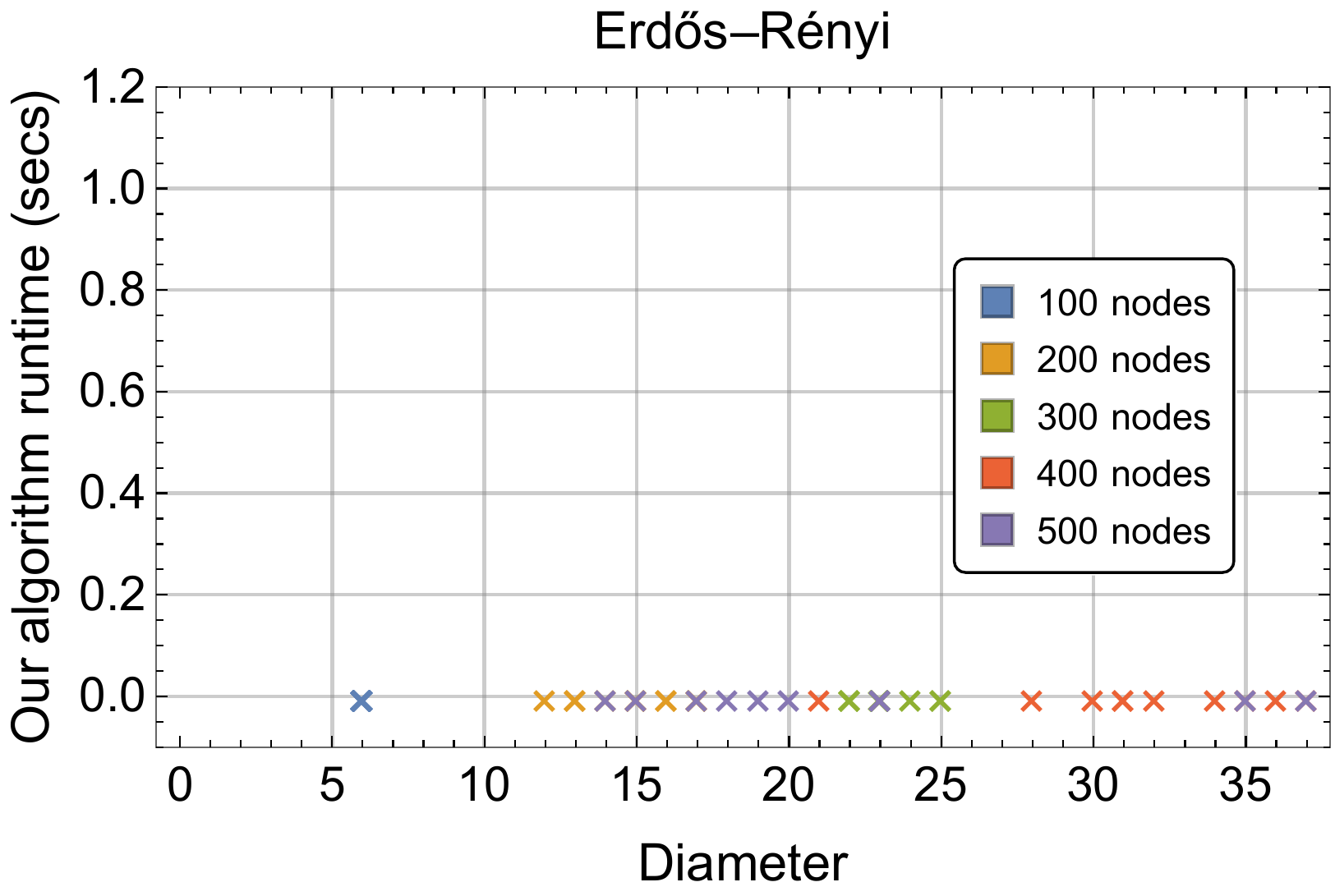}
    }
    \subfigure[]{
    \includegraphics[width=.46\linewidth]{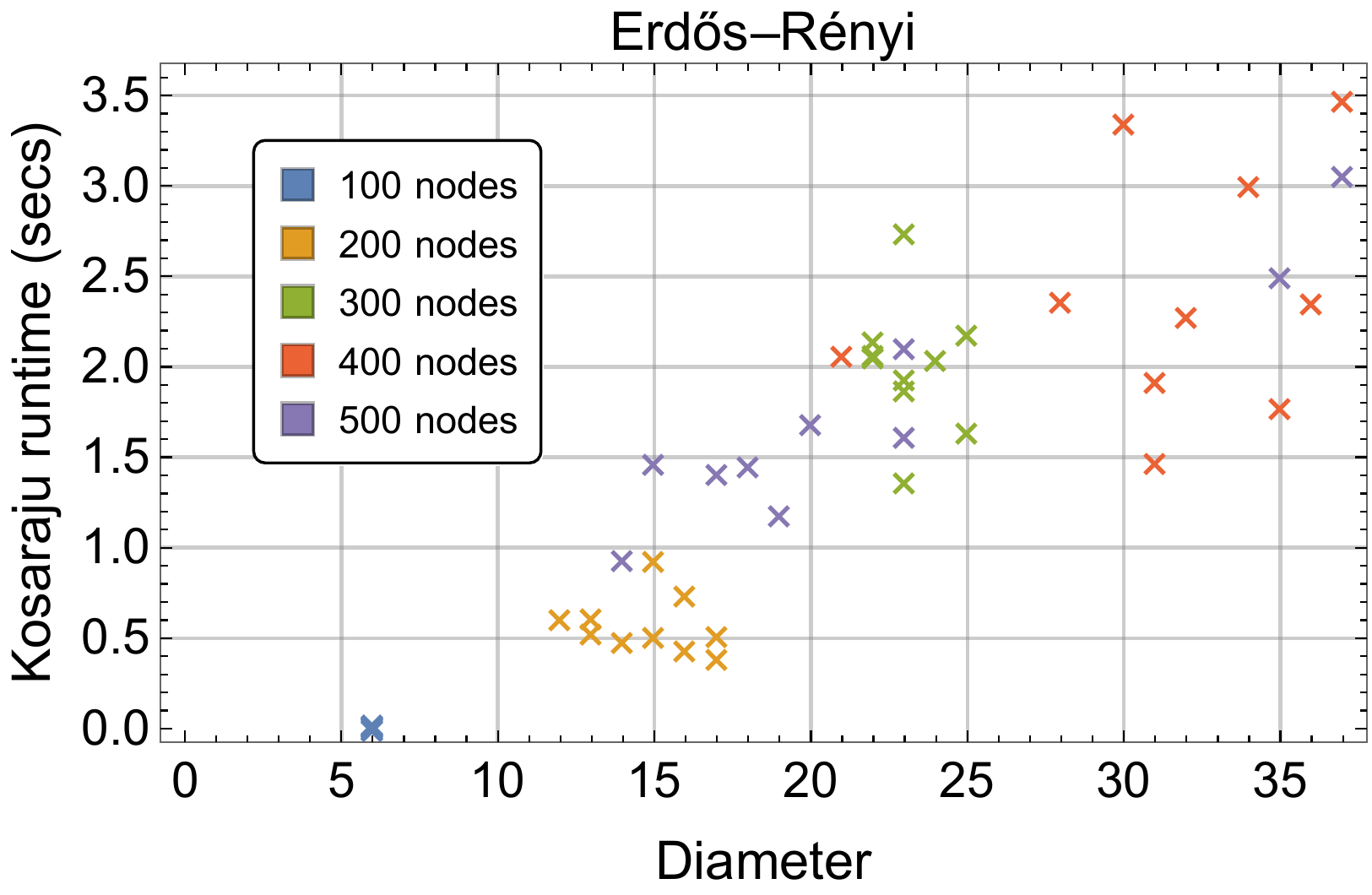}
    }
    \subfigure[]{
    \includegraphics[width=.46\linewidth]{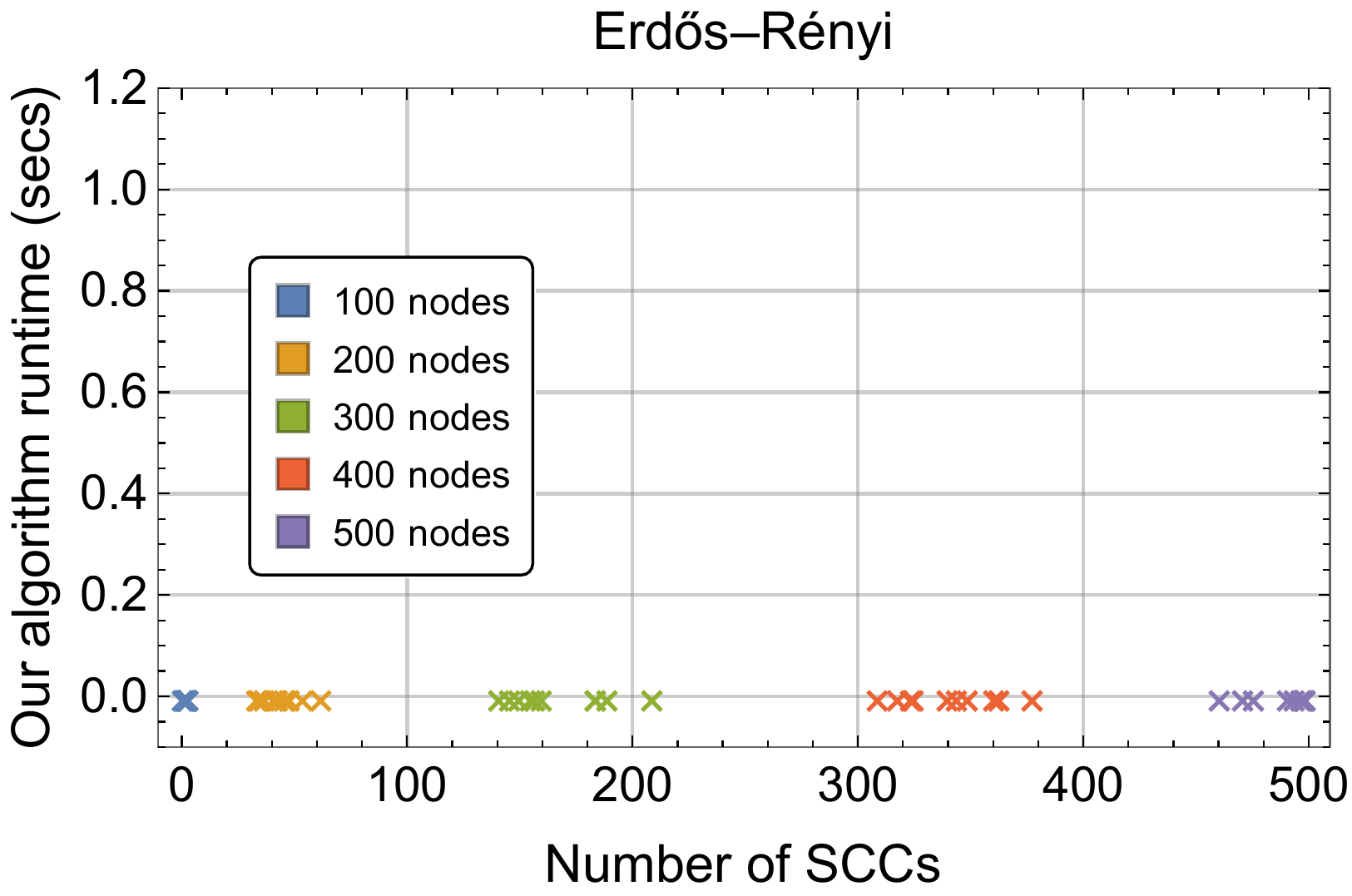}
    }
    \subfigure[]{
    \includegraphics[width=.46\linewidth]{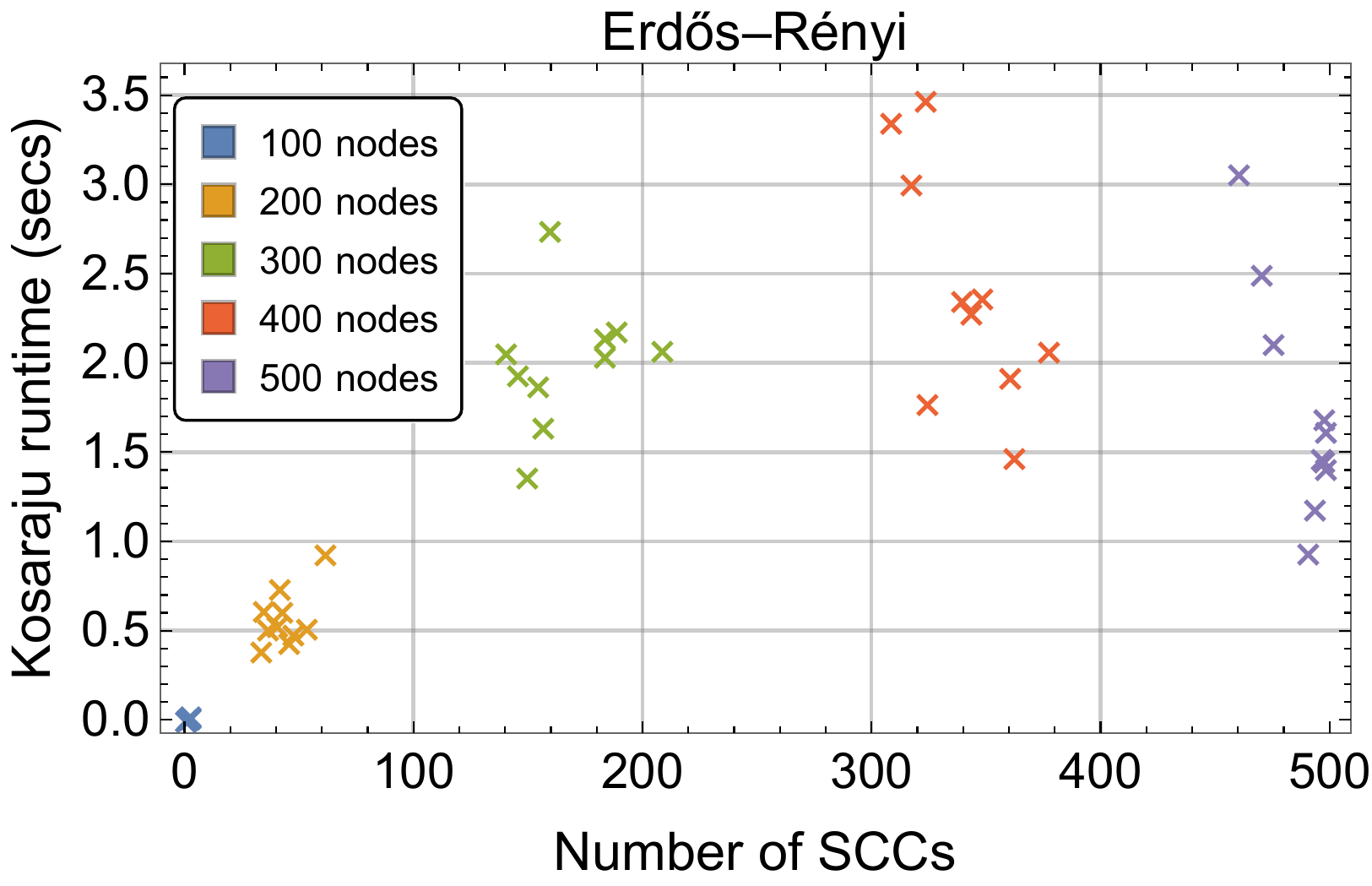}
    }
    \caption{These figures show the relationship between the network properties of some randomly generated Erdős-Rényi networks and their run times for both our proposed algorithm and Kosaraju's algorithm.}
    \label{fig:ERparam2}
\end{figure}

In Figure \ref{fig:ERparam1}, we see the comparison between different properties of the network, including the maximum in-degree, diameter, and total number of SCCs, with the run times of both our algorithm and Kosaaraju's algorithm. With this set of parameters, the maximum in-degree plays a much larger role in determining the run-time of the algorithm. As such, the centralized algorithm was run instead of the distributed algorithm because in order to run the algorithm in parallel, we have to create a shared memory data structure, which increases the run time. Because there are only two iterations before terminating,  parallelization is not computationally viable in this case. In other words, parallelization becomes advantageous when the diameter of the network is large.  

In Figure \ref{fig:ERruntimesparam1}, we see the comparison between the run times of both our algorithm and Kosaraju's algorithm on the different randomly generated Erdős-Rényi networks. Our algorithm performs better overall, especially for networks with more nodes. 
\begin{figure}[H]
    \centering
    \includegraphics[width =0.46 \linewidth]{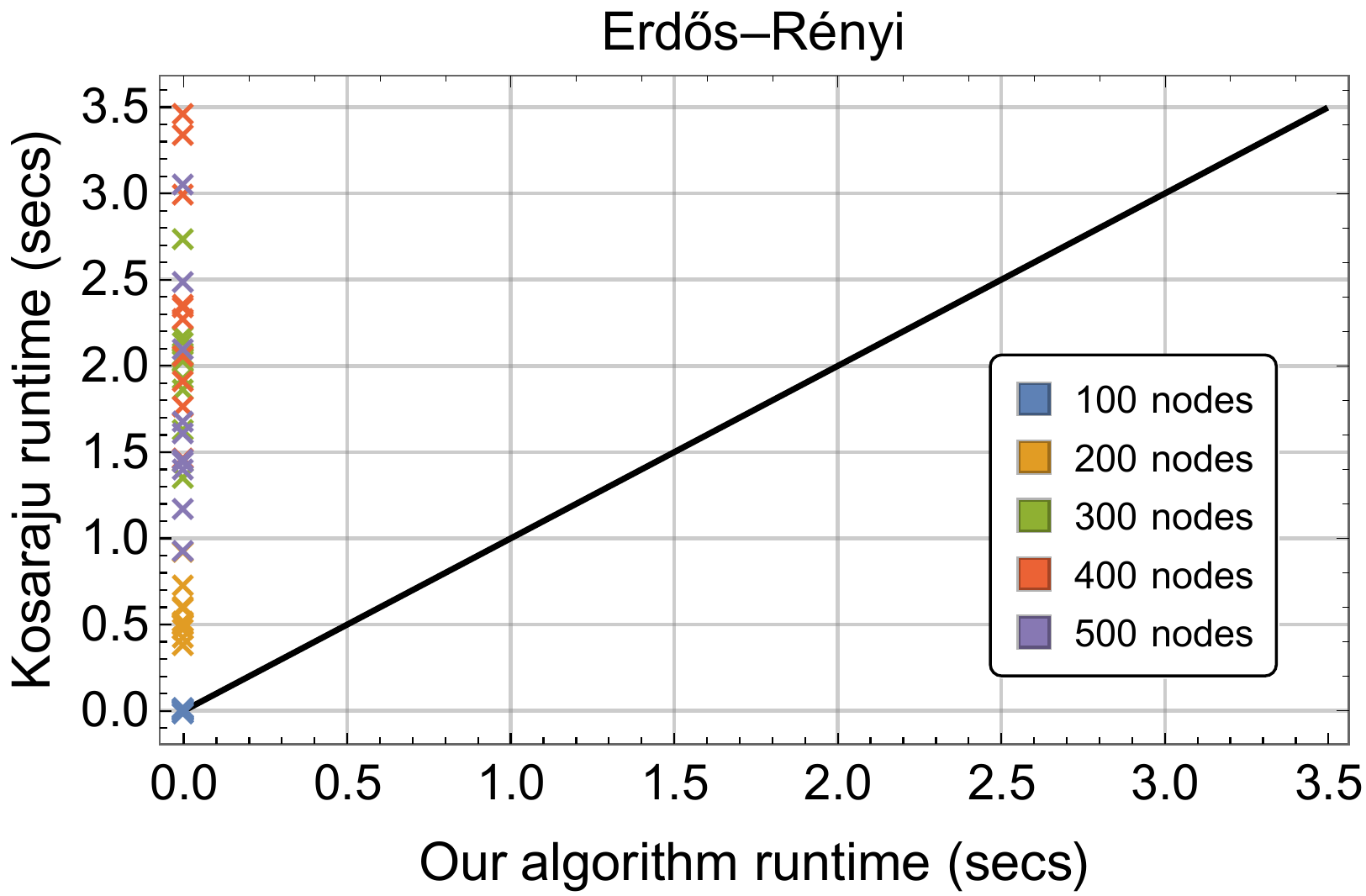}
    \caption{This figure compares the run times of  both our proposed algorithm and Kosaraju's algorithm for several randomly generated Erdős-Rényi networks. We see that our algorithm performs better on networks with a higher number of nodes.}
    \label{fig:ERruntimesparam2}
\end{figure}

The results from the second set of parameters for Erdős-Rényi networks are shown in Figures \ref{fig:ERparam2} and \ref{fig:ERruntimesparam2}. In these networks, the diameter is much larger, so it dominates the runtime of our algorithm. Hence, we used the distributed algorithm to find the SCCs of these networks. Figure \ref{fig:ERruntimesparam2} shows that the runtime of our distributed algorithm is far superior to that of Kosaraju when executed on the same \mbox{Erdős-Rényi} random networks. 

\subsection{Barabási-Albert}


\begin{figure}[H]
    \centering
    \subfigure[]{
    \includegraphics[width=.46\linewidth]{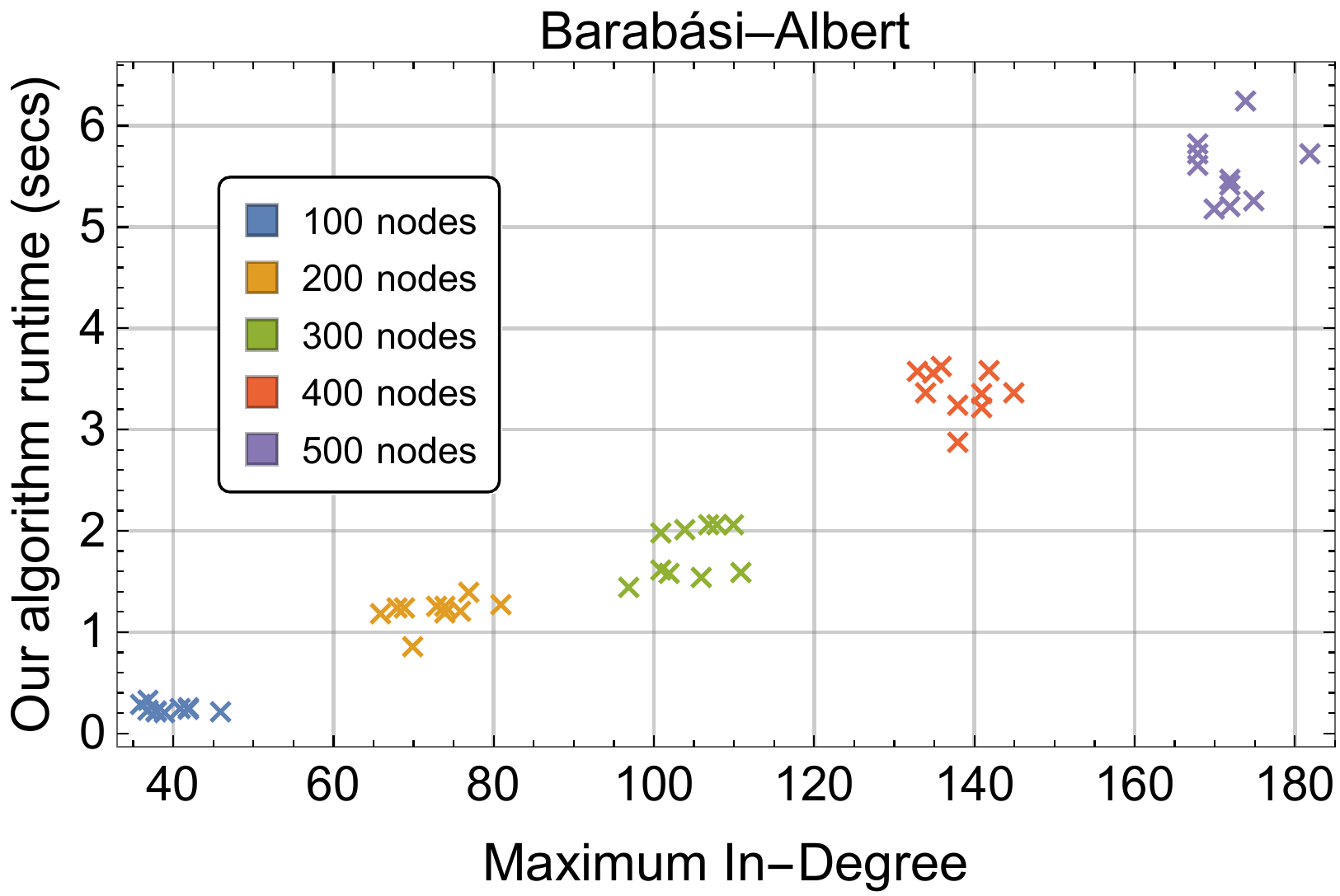}
    }
    \subfigure[]{
    \includegraphics[width=.46\linewidth]{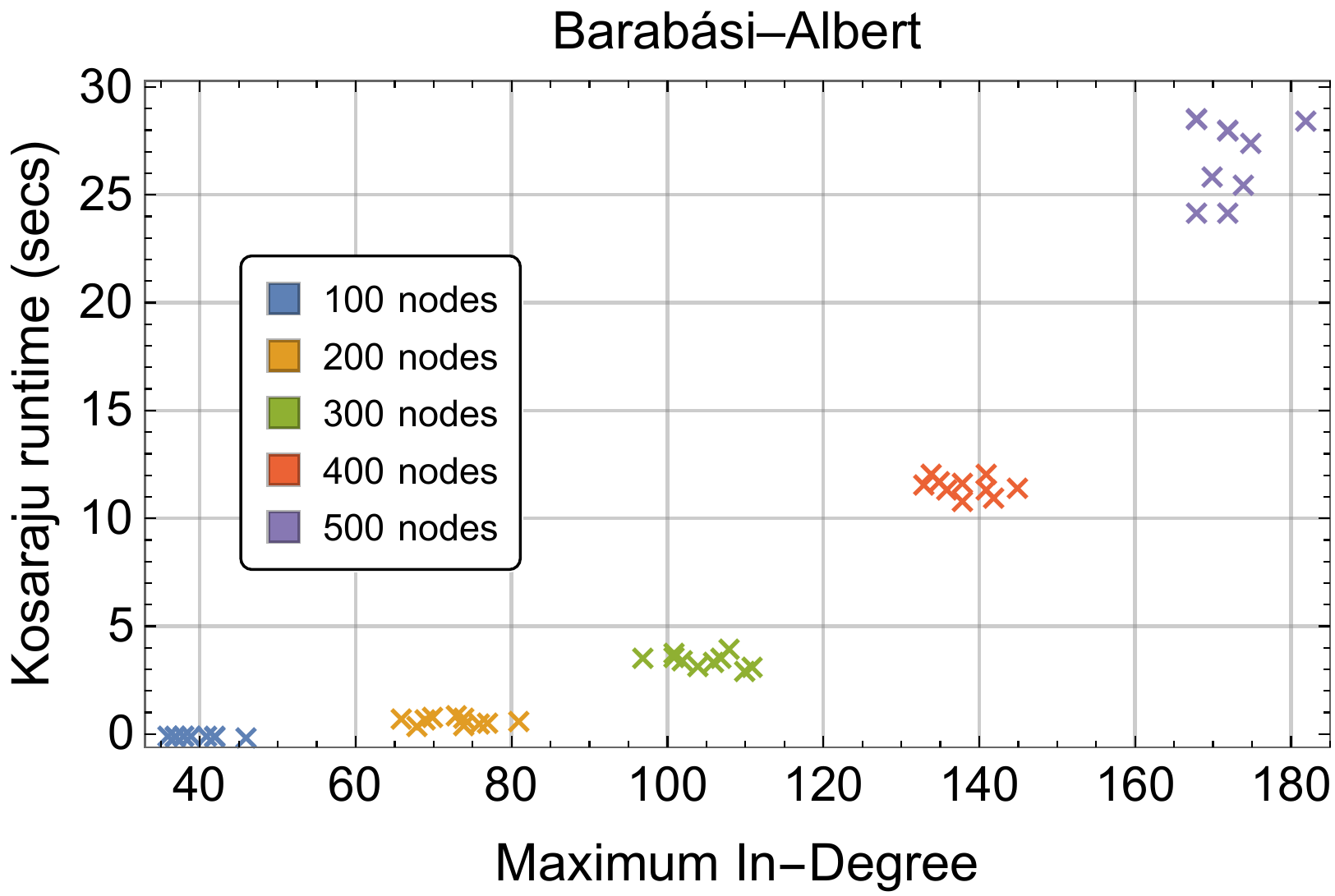}
    }
    \subfigure[]{
    \includegraphics[width=.46\linewidth]{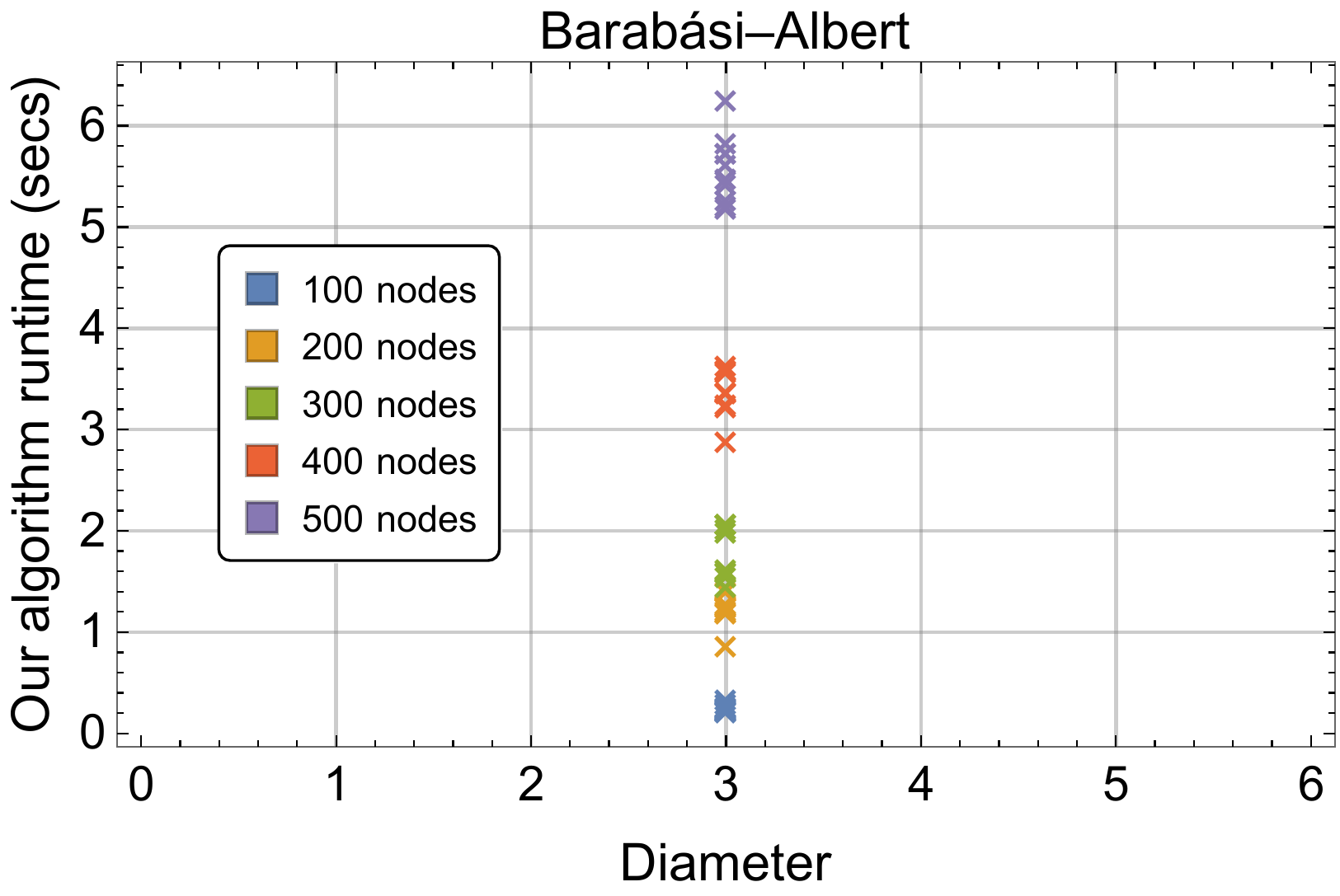}
    }
    \subfigure[]{
    \includegraphics[width=.46\linewidth]{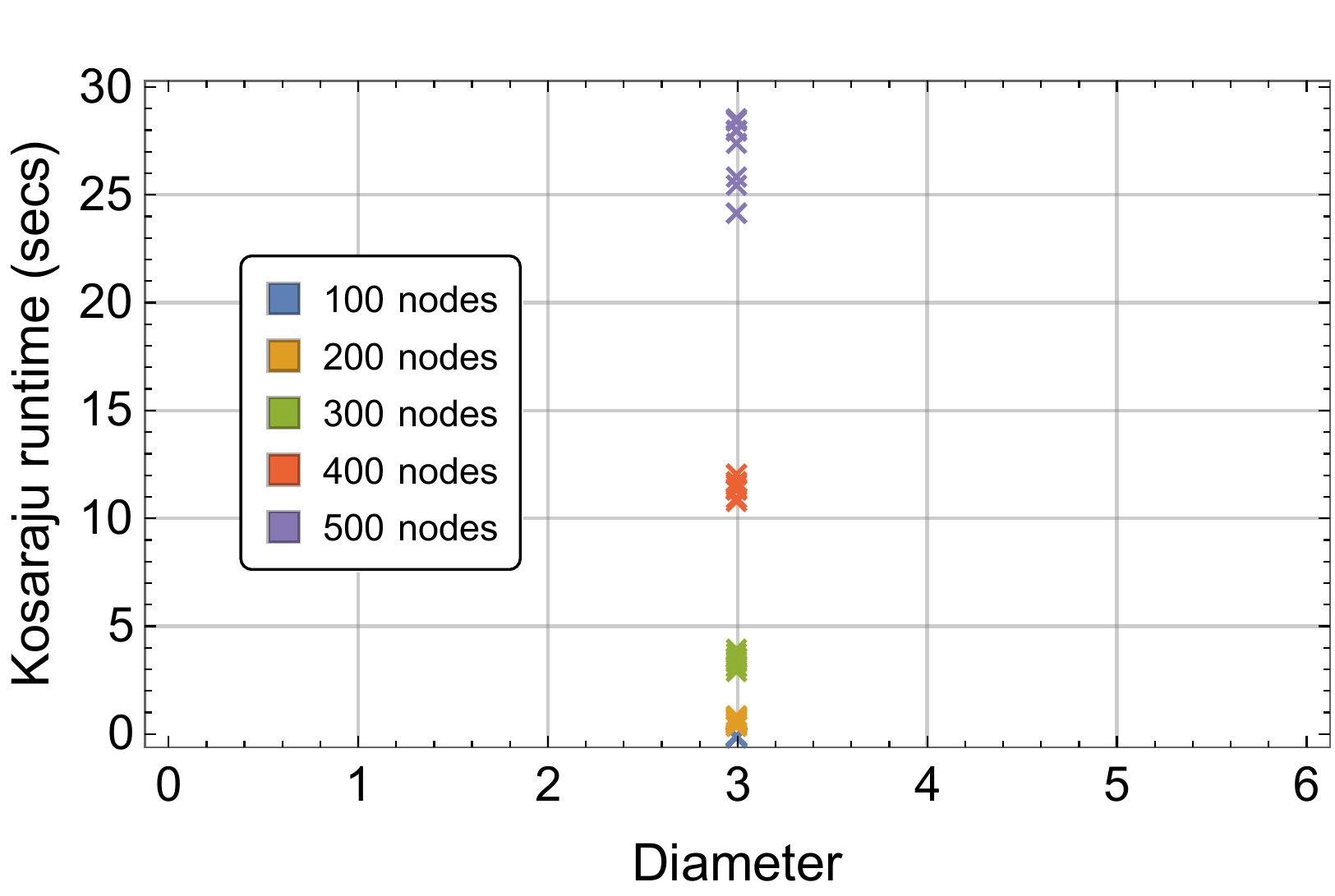}
    }
    \subfigure[]{
    \includegraphics[width=.46\linewidth]{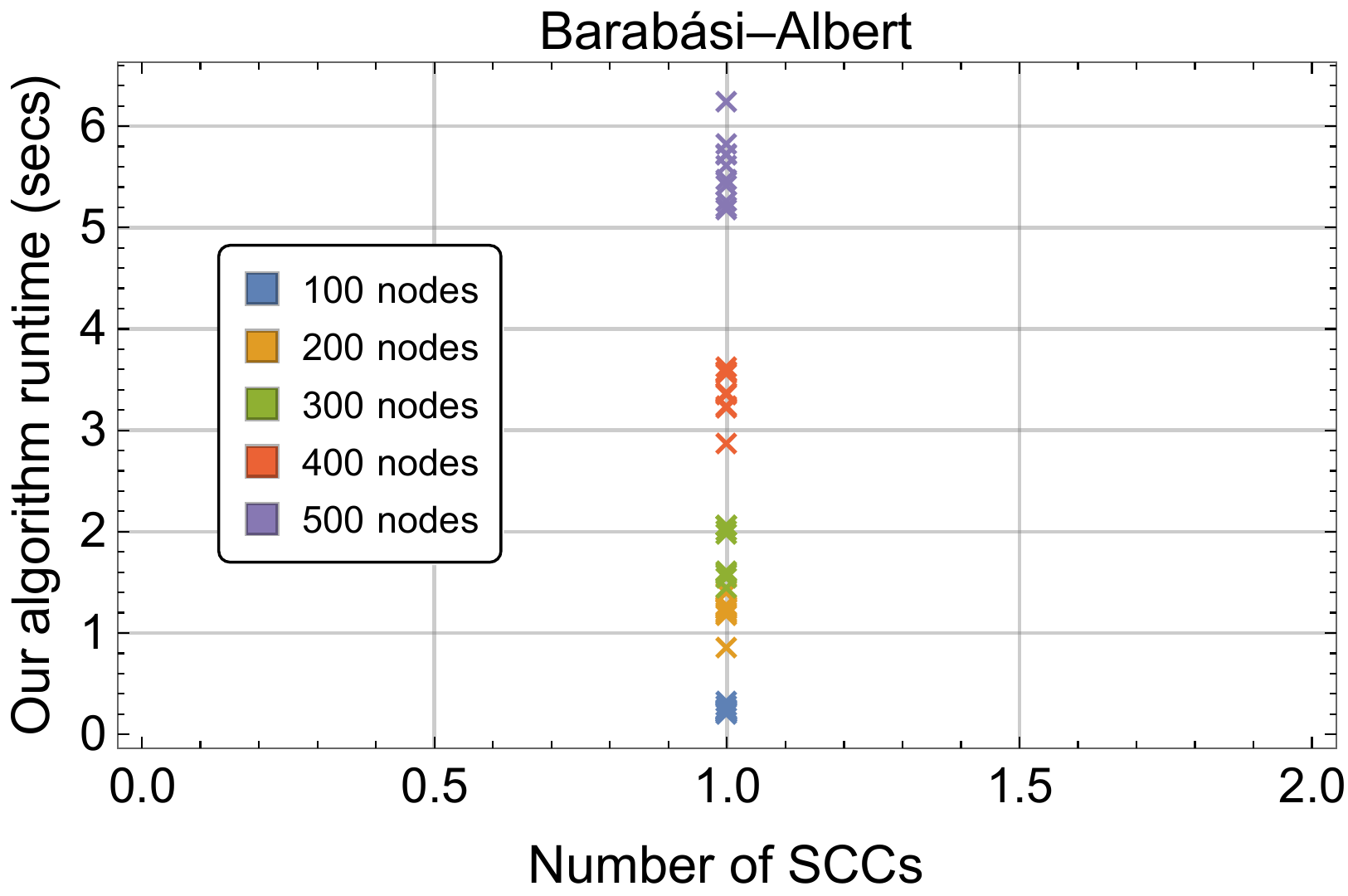}
    }
    \subfigure[]{
    \includegraphics[width=.46\linewidth]{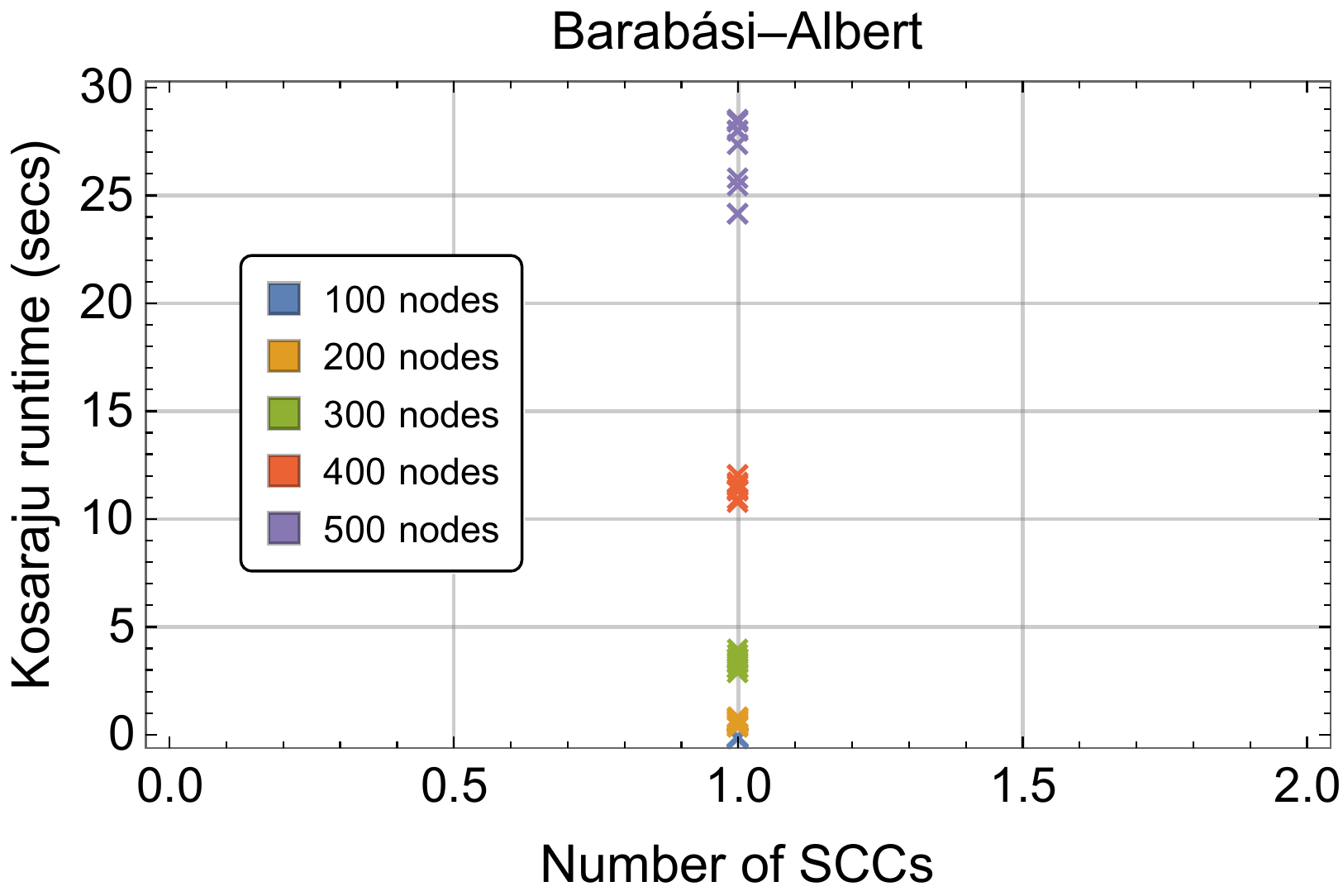}
    }
    \caption{These figures show the relationship between the network properties of some randomly generated Barabási-Albert networks and their run times for both our proposed algorithm and Kosaraju's algorithm.}
    \label{fig:BAparam1}
\end{figure}
The Barabási-Albert networks require two parameters, including the number of nodes and the number of edges added to a new vertex at each step.  

For the first set of parameters (Figures \ref{fig:BAparam1} and \ref{fig:BAruntimesparam1}), the number of nodes were fixed to 100, 200, 300, 400, 500, and the  number of edges were chosen to be the number of nodes divided by 5. In the second set of parameters (Figures \ref{fig:BAparam2} and \ref{fig:BAruntimesparam2}), again the number of nodes remained the same, but the number of edges was fixed to 50 for all the sets of nodes.

In Figure \ref{fig:BAparam1}, we see that the maximum in-degree plays a much larger role in determining the run-time of the algorithm. As such, the centralized algorithm was run instead of the distributed algorithm.  

\begin{figure}[H]
    \centering
    \includegraphics[width =0.46 \linewidth]{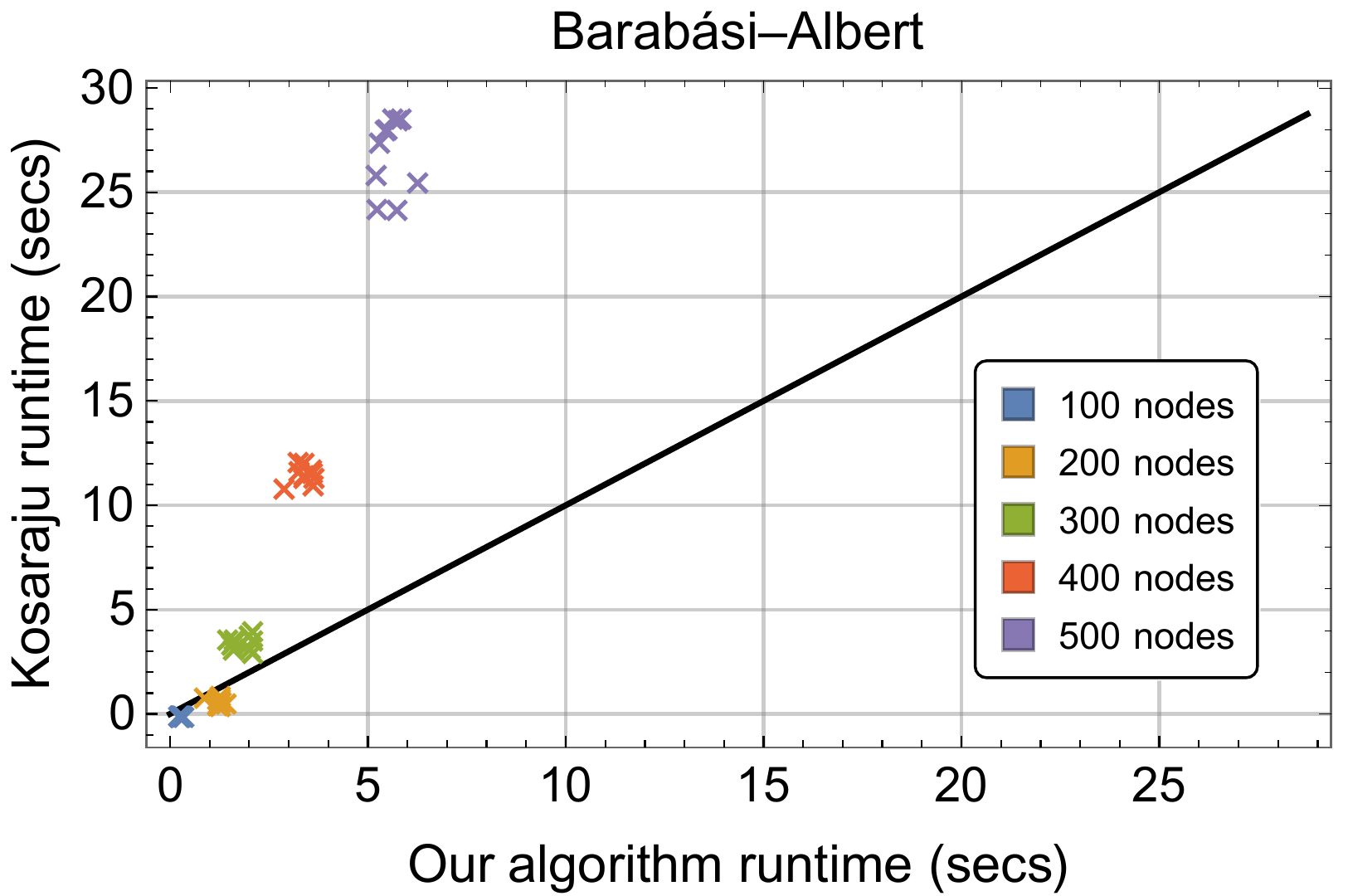}
    \caption{This figure compares the run times of  both our proposed algorithm and Kosaraju's algorithm for several randomly generated Barabási-Albert networks. We see that our algorithm performs better on networks with a higher number of nodes.}
    \label{fig:BAruntimesparam1}
\end{figure}

In Figure \ref{fig:BAruntimesparam1}, we see the comparison between the run times of both our algorithm and Kosaraju's algorithm on the different randomly generated Barabási-Albert networks. Even with the centarlized algorithm, our method performs better for networks with more nodes. 

\begin{figure}[H]
    \centering
    \subfigure[]{
    \includegraphics[width=.46\linewidth]{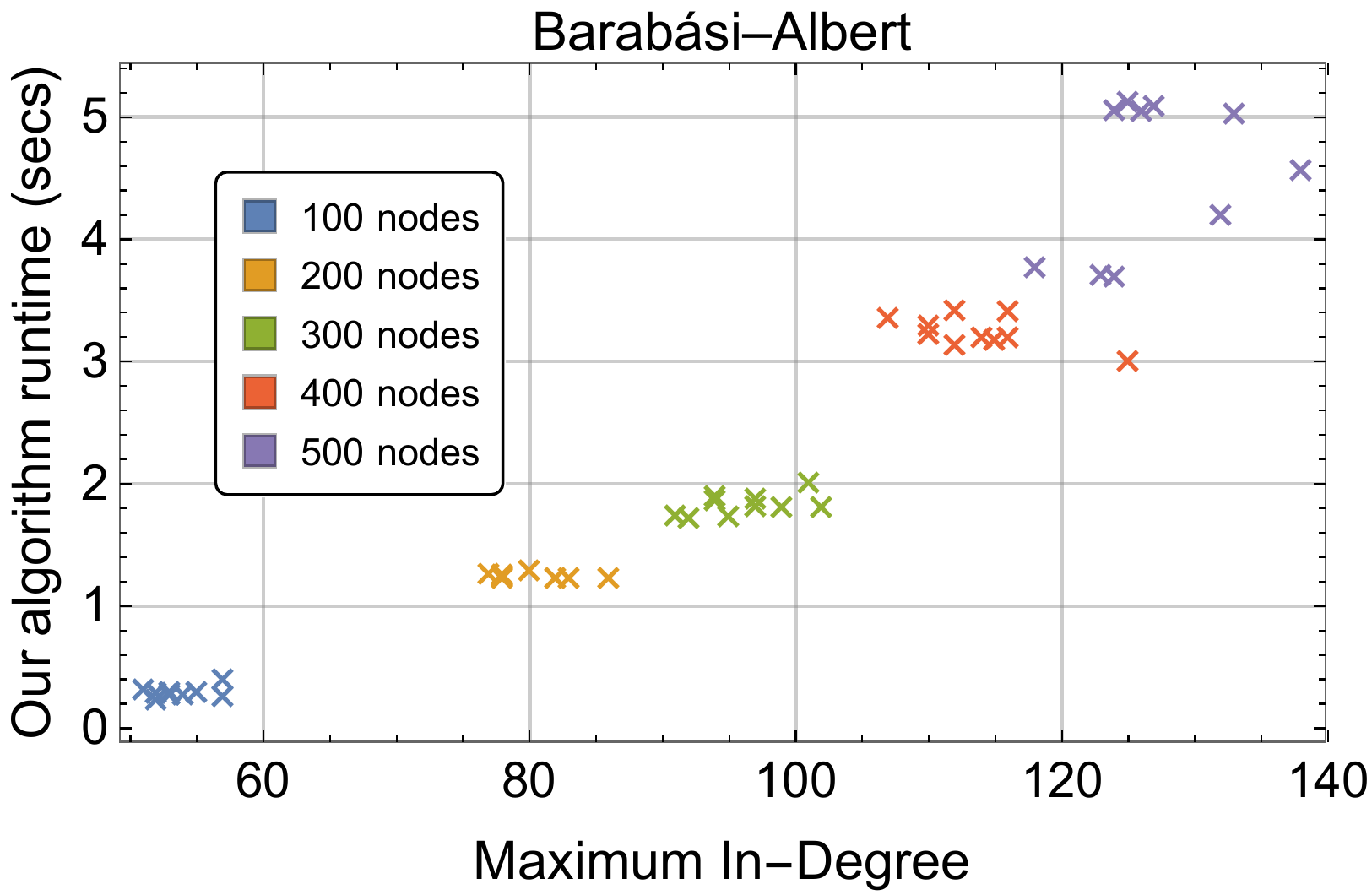}
    }
    \subfigure[]{
    \includegraphics[width=.46\linewidth]{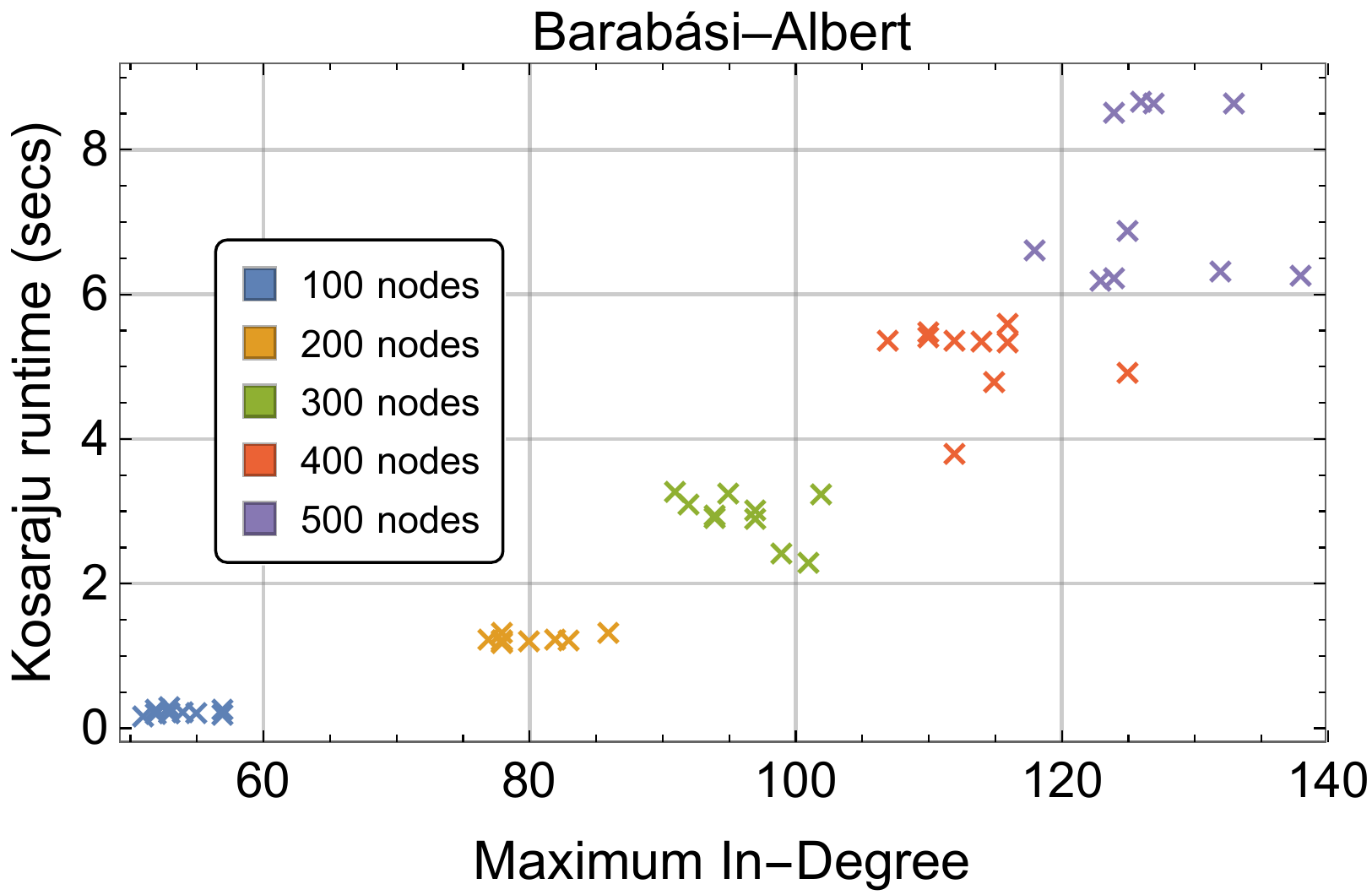}
    }
    \subfigure[]{
    \includegraphics[width=.46\linewidth]{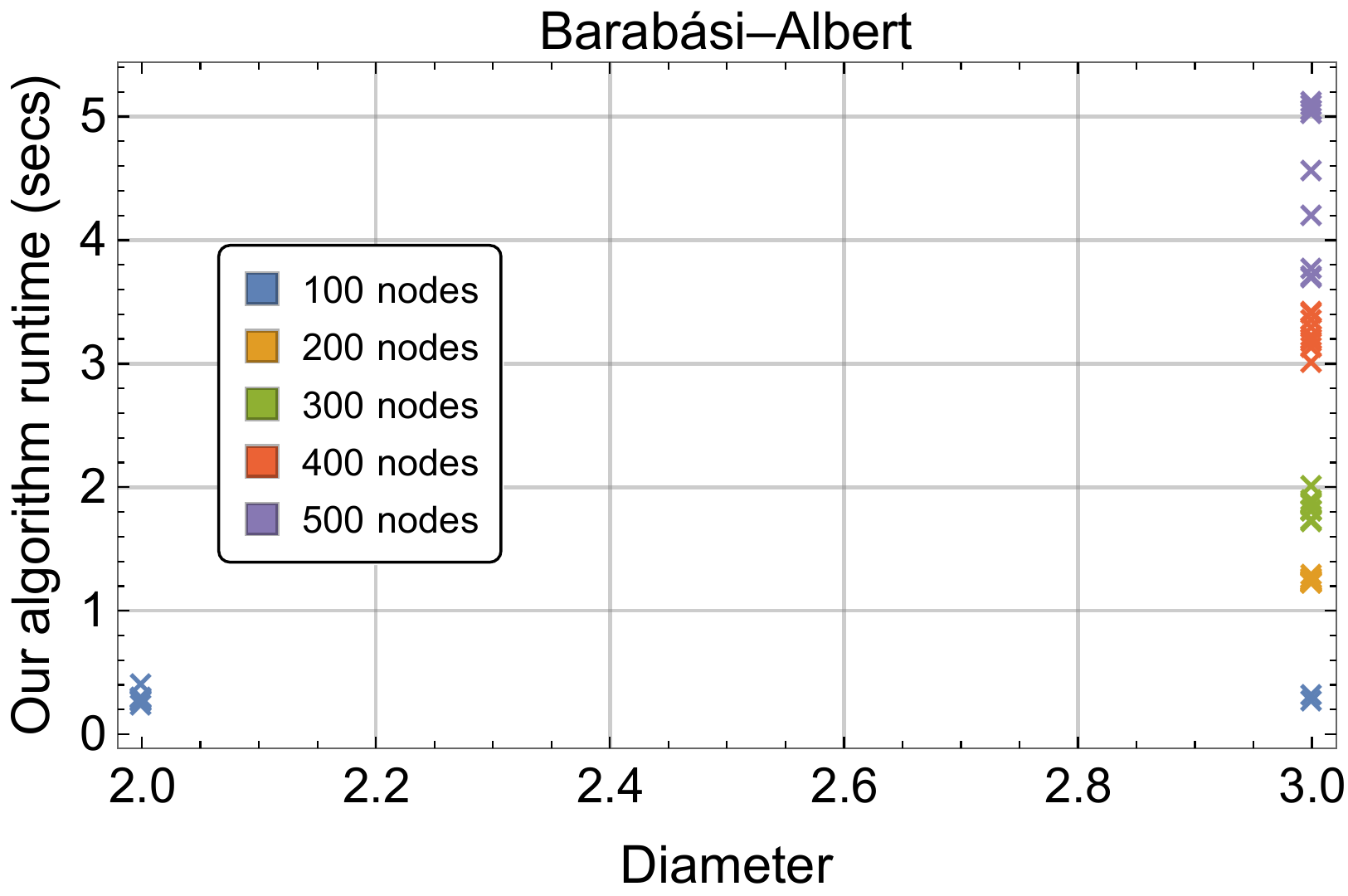}
    }
    \subfigure[]{
    \includegraphics[width=.46\linewidth]{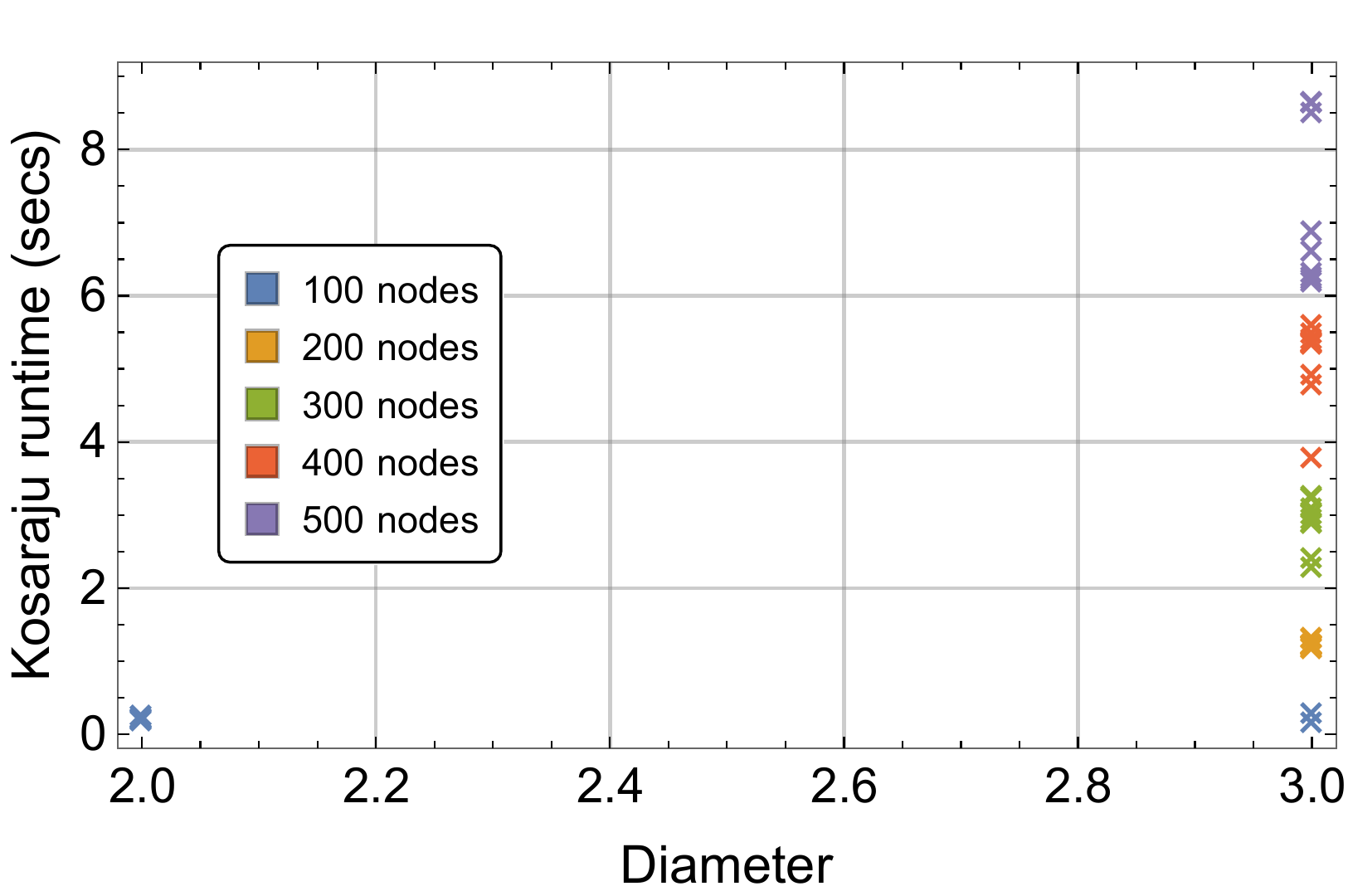}
    }
    \subfigure[]{
    \includegraphics[width=.46\linewidth]{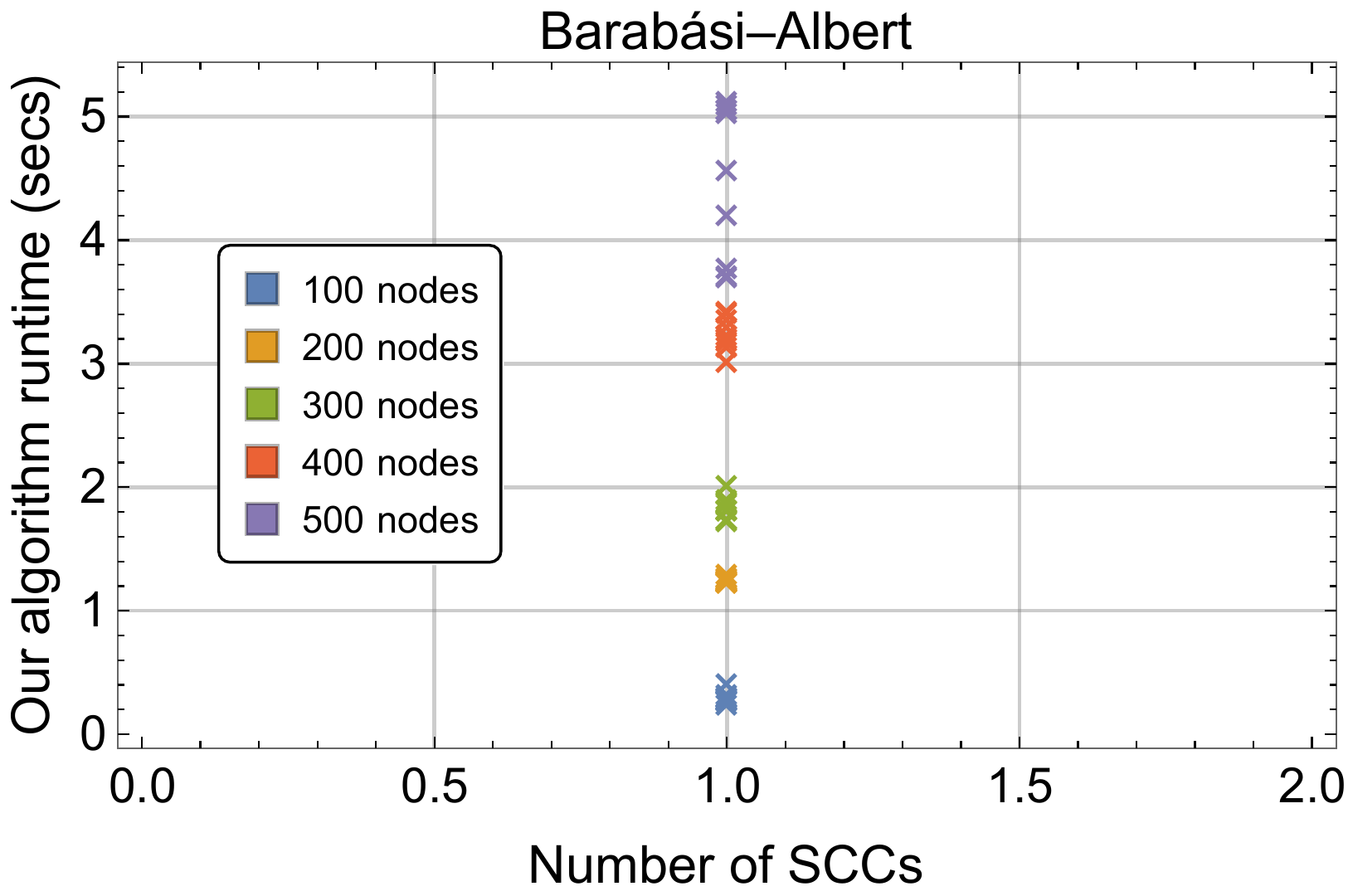}
    }
    \subfigure[]{
    \includegraphics[width=.46\linewidth]{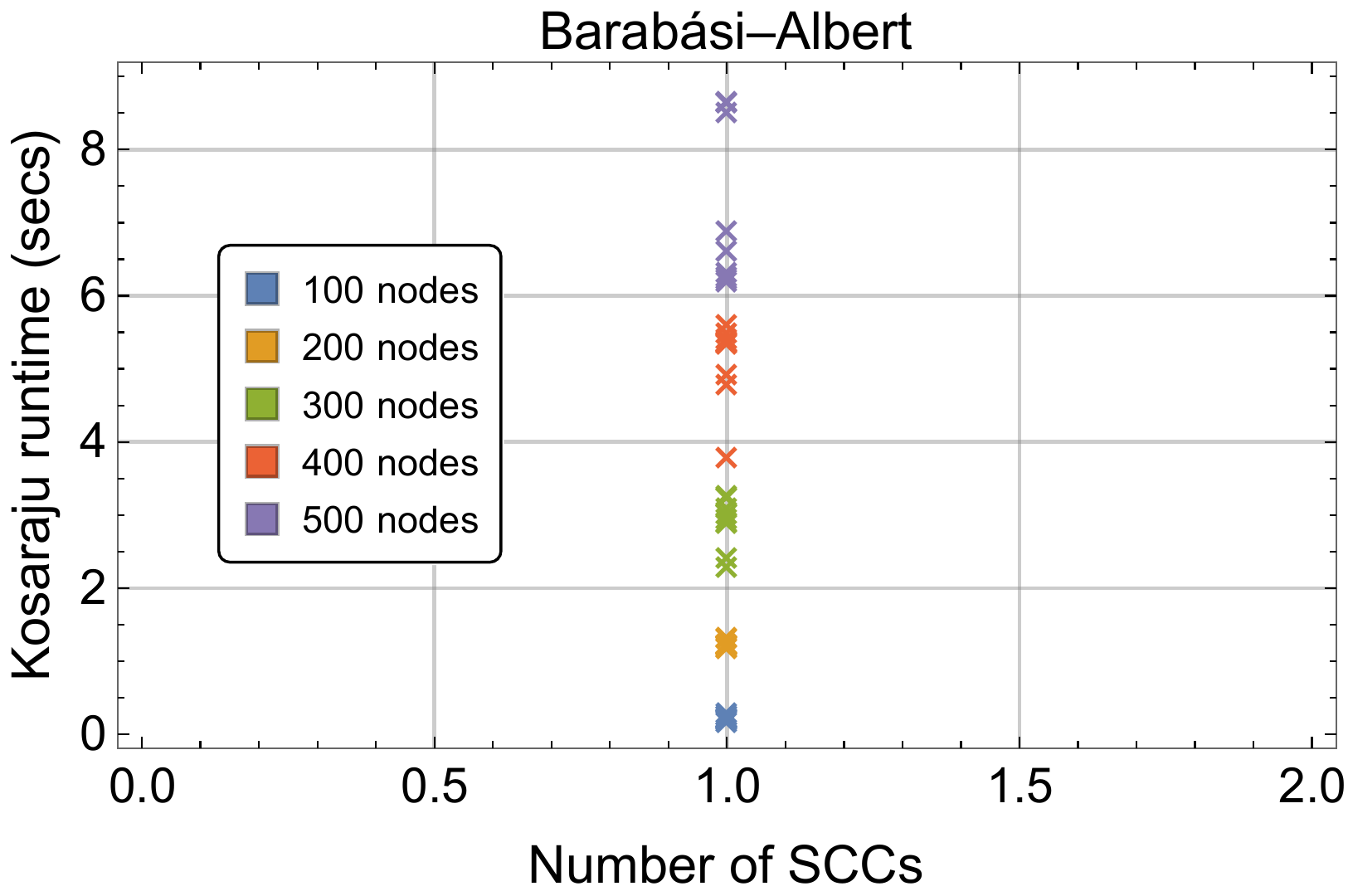}
    }
    \caption{These figures show the relationship between the network properties of some randomly generated Barabási-Albert networks and their run times for both our proposed algorithm and Kosaraju's algorithm.}
    \label{fig:BAparam2}
\end{figure}

The results from the second set of parameters for Barabási–Albert networks are shown in Figures \ref{fig:BAparam2} and \ref{fig:BAruntimesparam2}. The results from the two sets of parameters do not present much difference. It is clear that our algorithm outperforms the Kosaraju algorithm when there is a large number of nodes.

\begin{figure}[H]
    \centering
    \includegraphics[width =0.46 \linewidth]{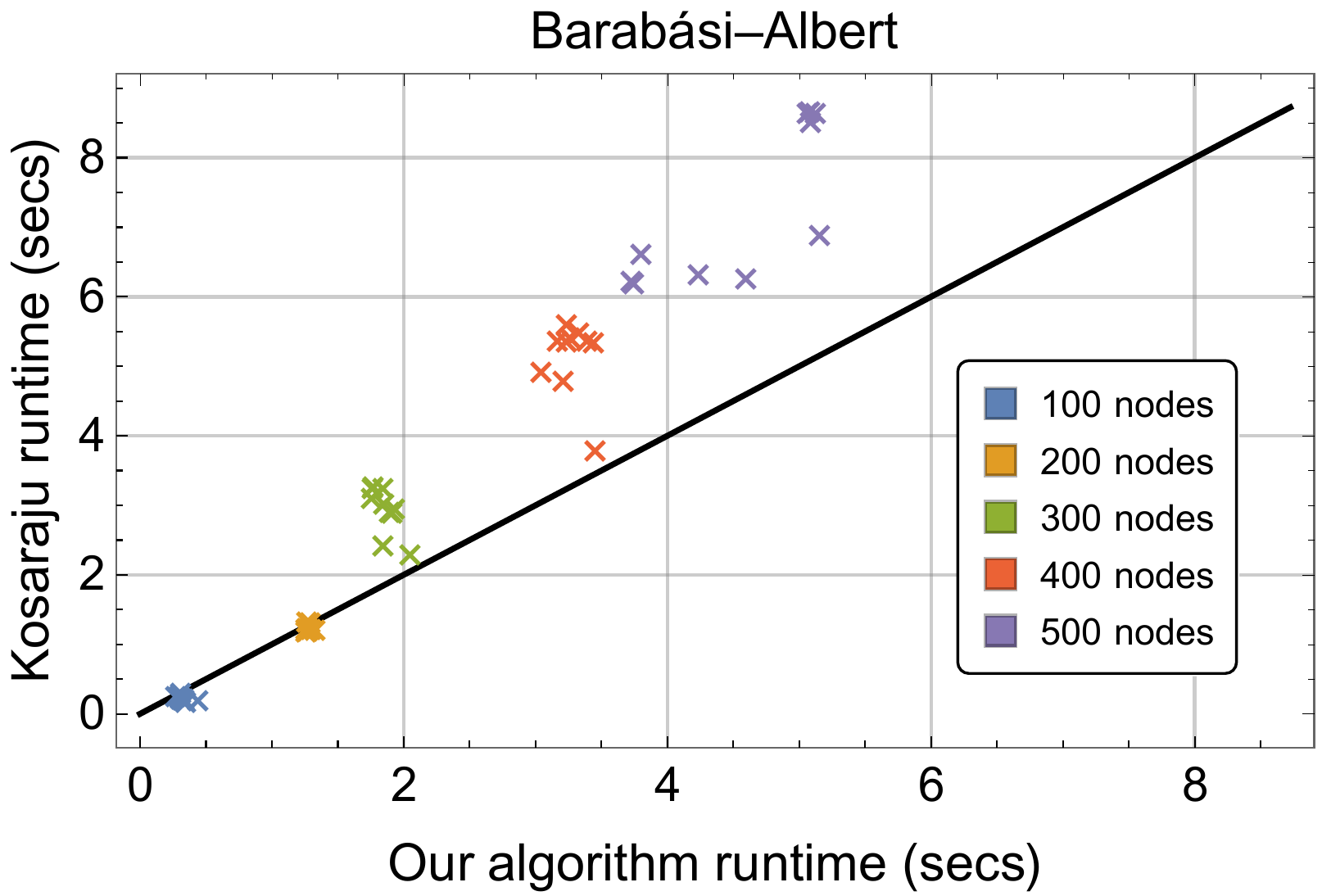}
    \caption{This figure compares the run times of  both our proposed algorithm and Kosaraju's algorithm for different randomly generated Barabási-Albert networks. We see that our algorithm performs better on networks with a higher number of nodes.}
    \label{fig:BAruntimesparam2}
\end{figure}

\subsection{Watts-Strogatz}

Finally, the Watts-Strogatz networks require the following two parameters, the number of nodes and the linkage probability (i.e., the probability that there is an edge between any two vertices). 

The first set of parameters included the nodes 100, 200, 300, 400, and 500 with a linkage probability of 0.8, and the results are shown in Figures \ref{fig:WSparam1} and \ref{fig:WSruntimesparam1}. For the second set of parameters, the set of nodes remain the same and the linkage probability is reduced to 0.2 with the results shown in Figures \ref{fig:WSparam2} and \ref{fig:WSruntimesparam2}. For all of the Watt-Strogatz networks, the distributed algorithm is used.

\begin{figure}[H]
    \centering
    \subfigure[]{
    \includegraphics[width=.46\linewidth]{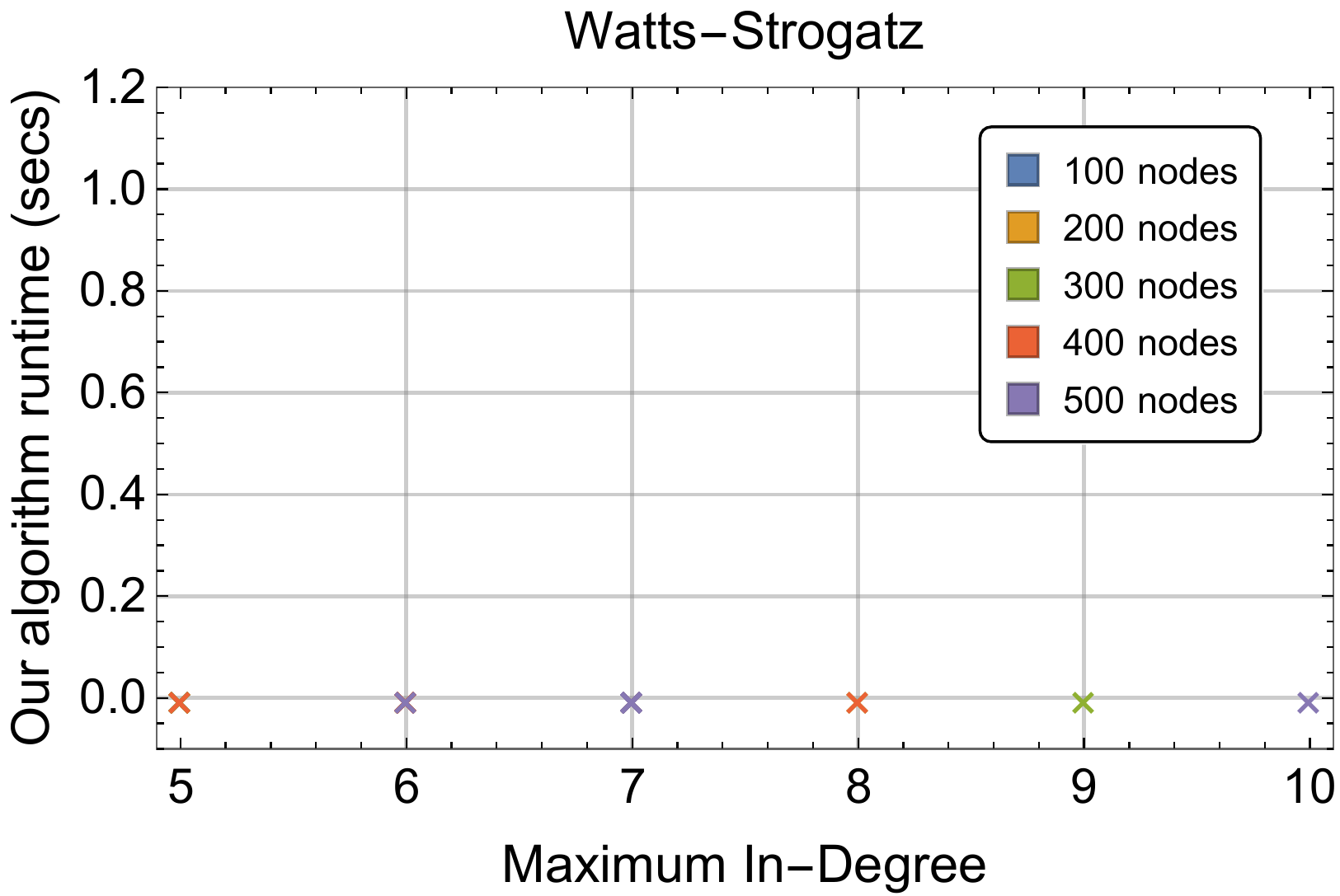}
    }
    \subfigure[]{
    \includegraphics[width=.46\linewidth]{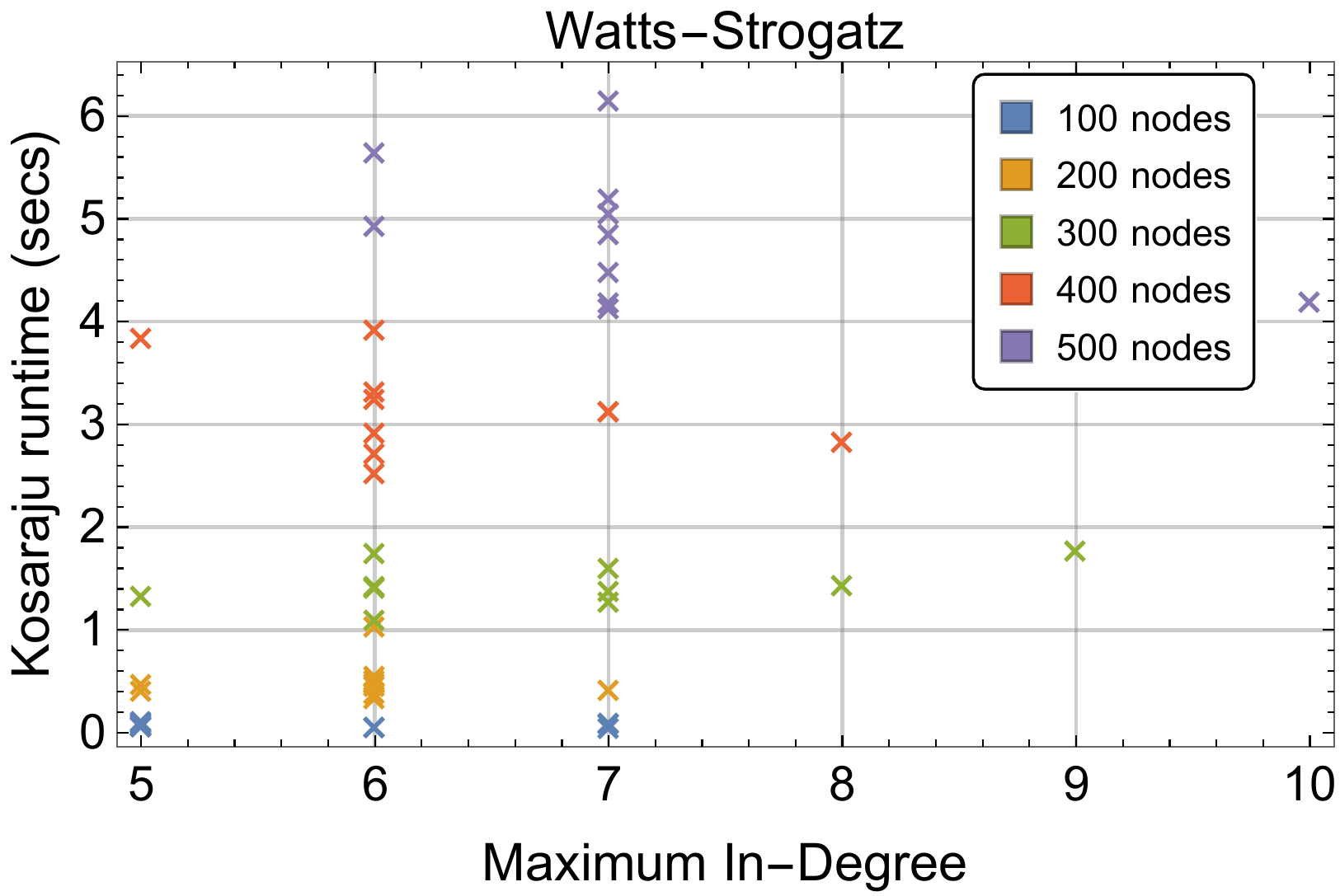}
    }
    \subfigure[]{
    \includegraphics[width=.46\linewidth]{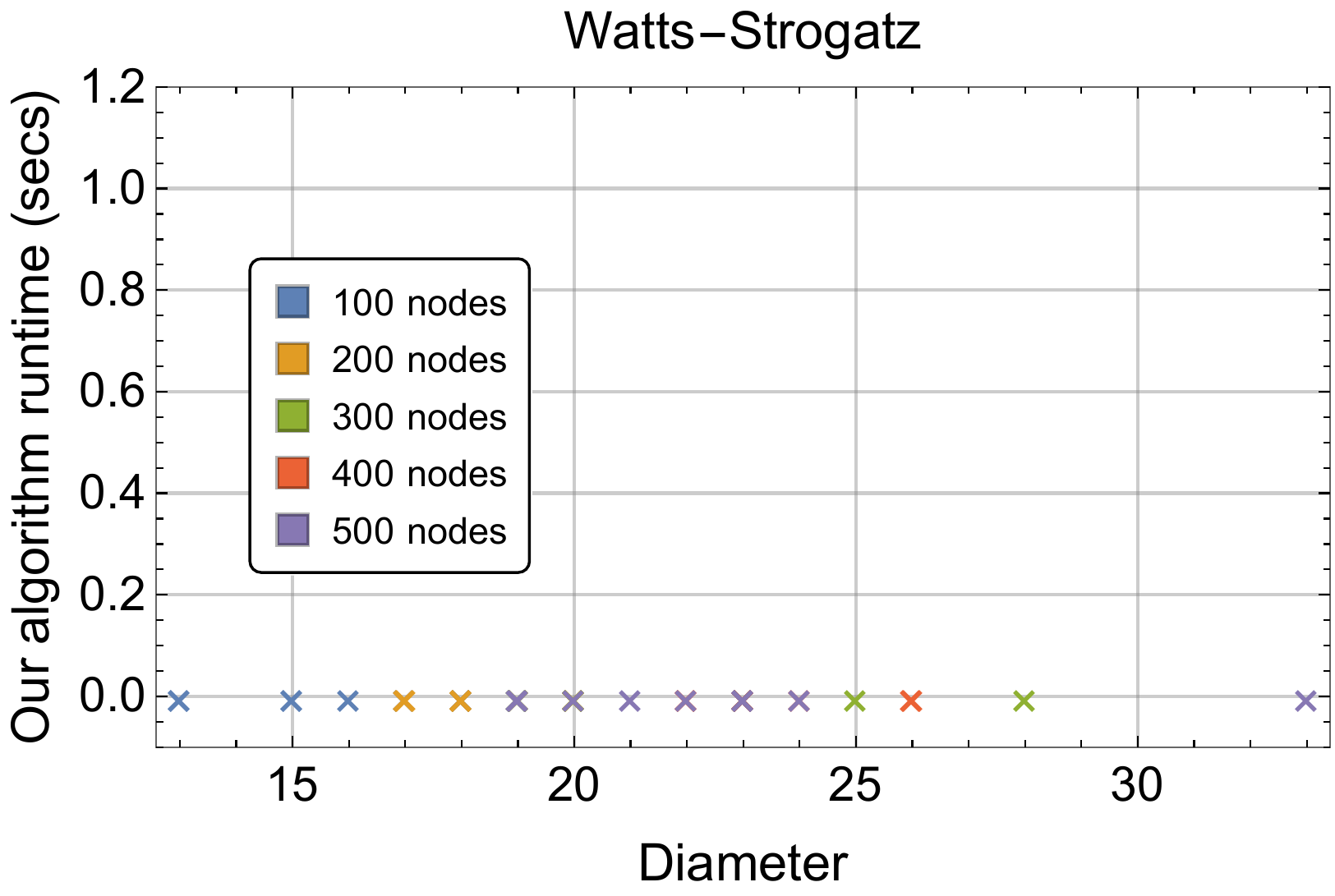}
    }
    \subfigure[]{
    \includegraphics[width=.46\linewidth]{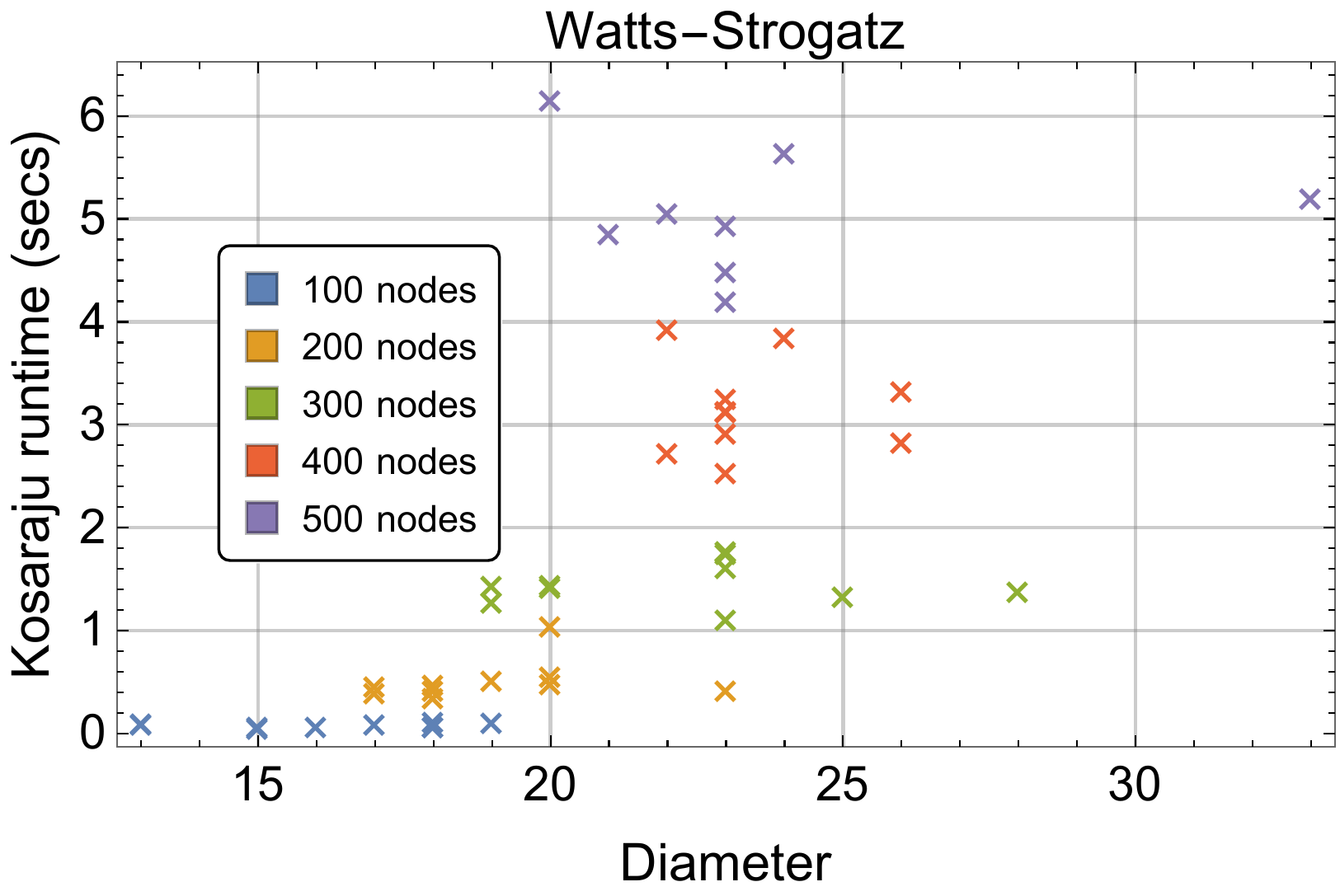}
    }
    \subfigure[]{
    \includegraphics[width=.46\linewidth]{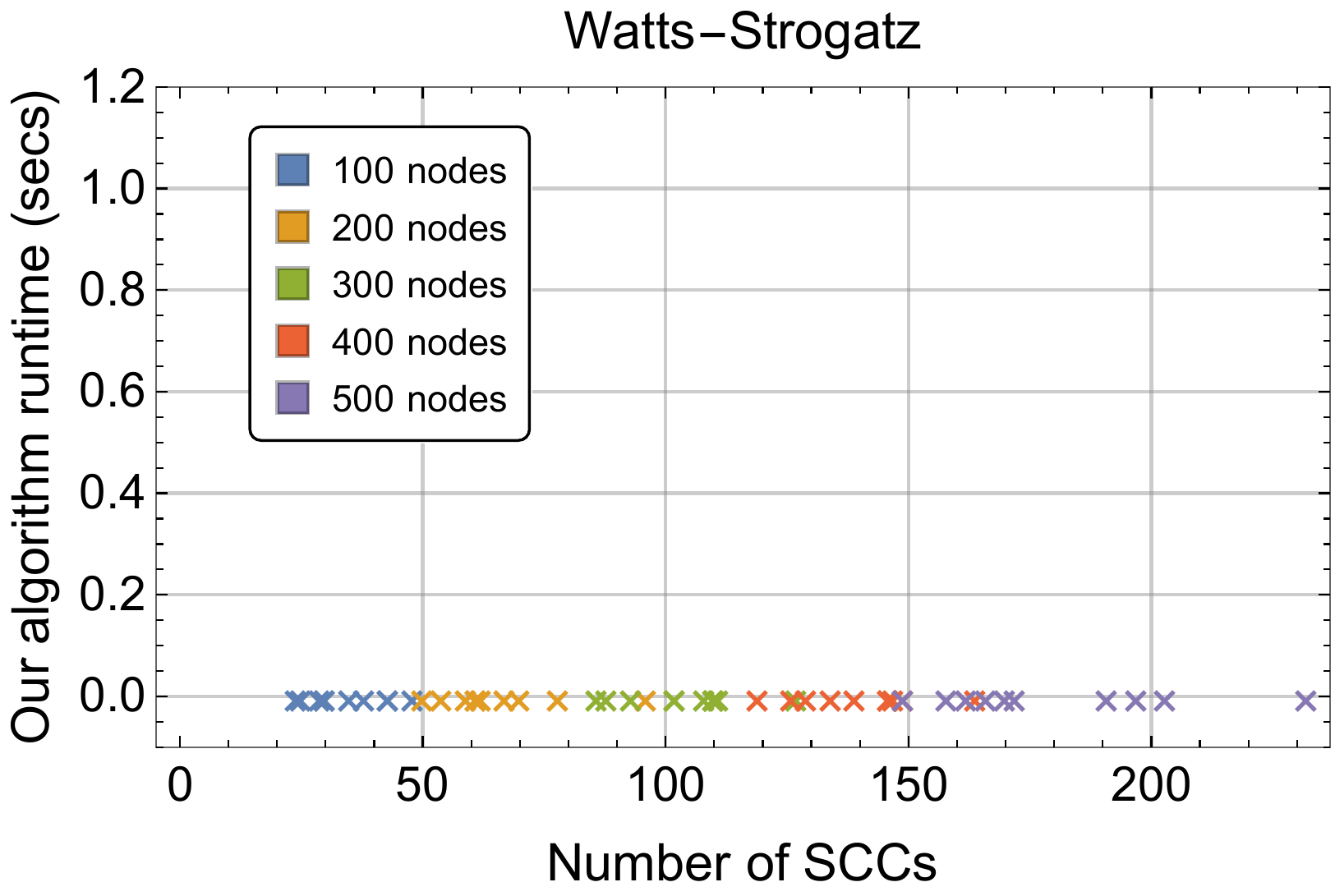}
    }
    \subfigure[]{
    \includegraphics[width=.46\linewidth]{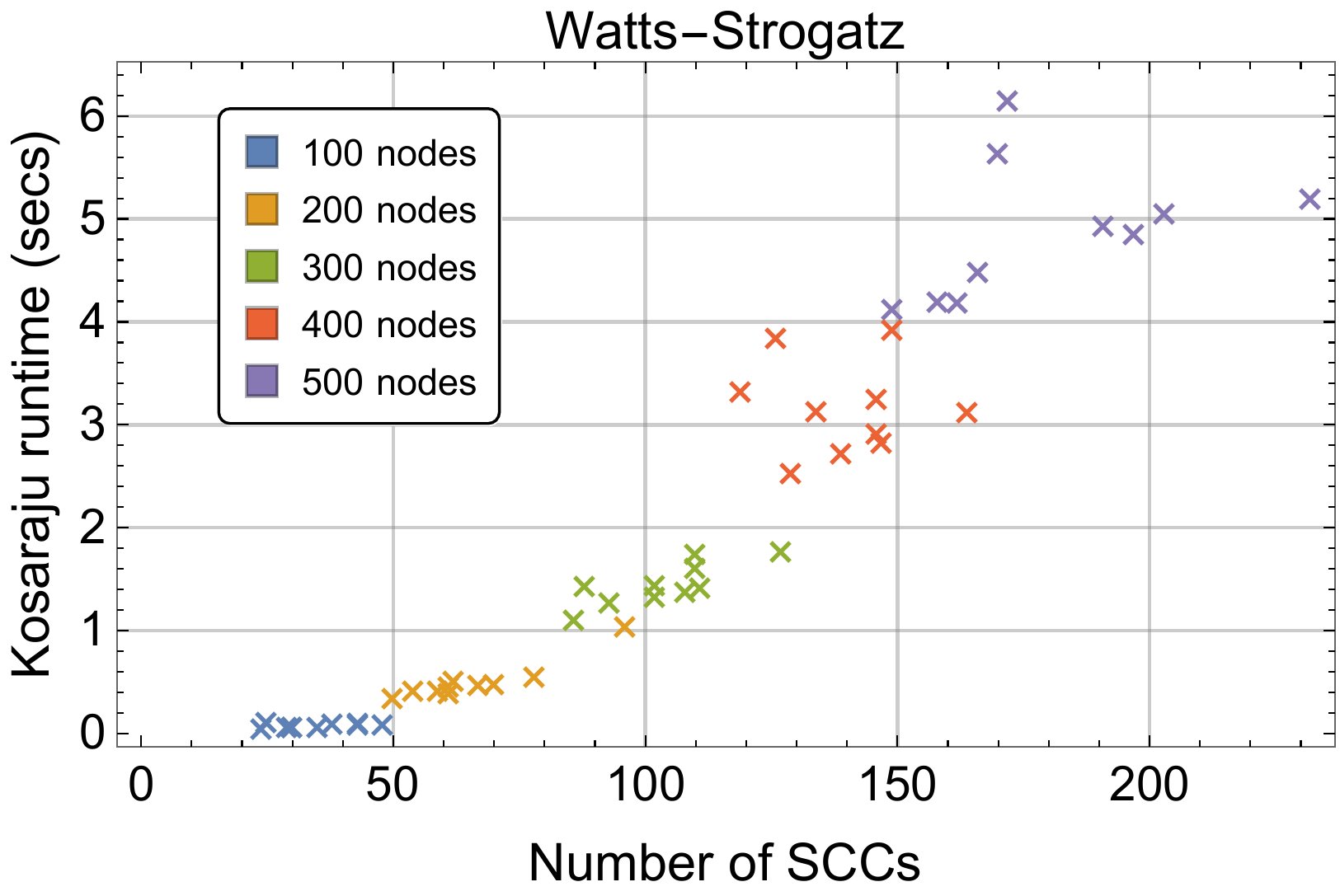}
    }
    \caption{These figures show the relationship between the network properties of some randomly generated Watts-Strogatz networks and their run times for both our proposed algorithm and Kosaraju's algorithm.}
    \label{fig:WSparam1}
\end{figure}

\begin{figure}[H]
    \centering
    \includegraphics[width =0.46 \linewidth]{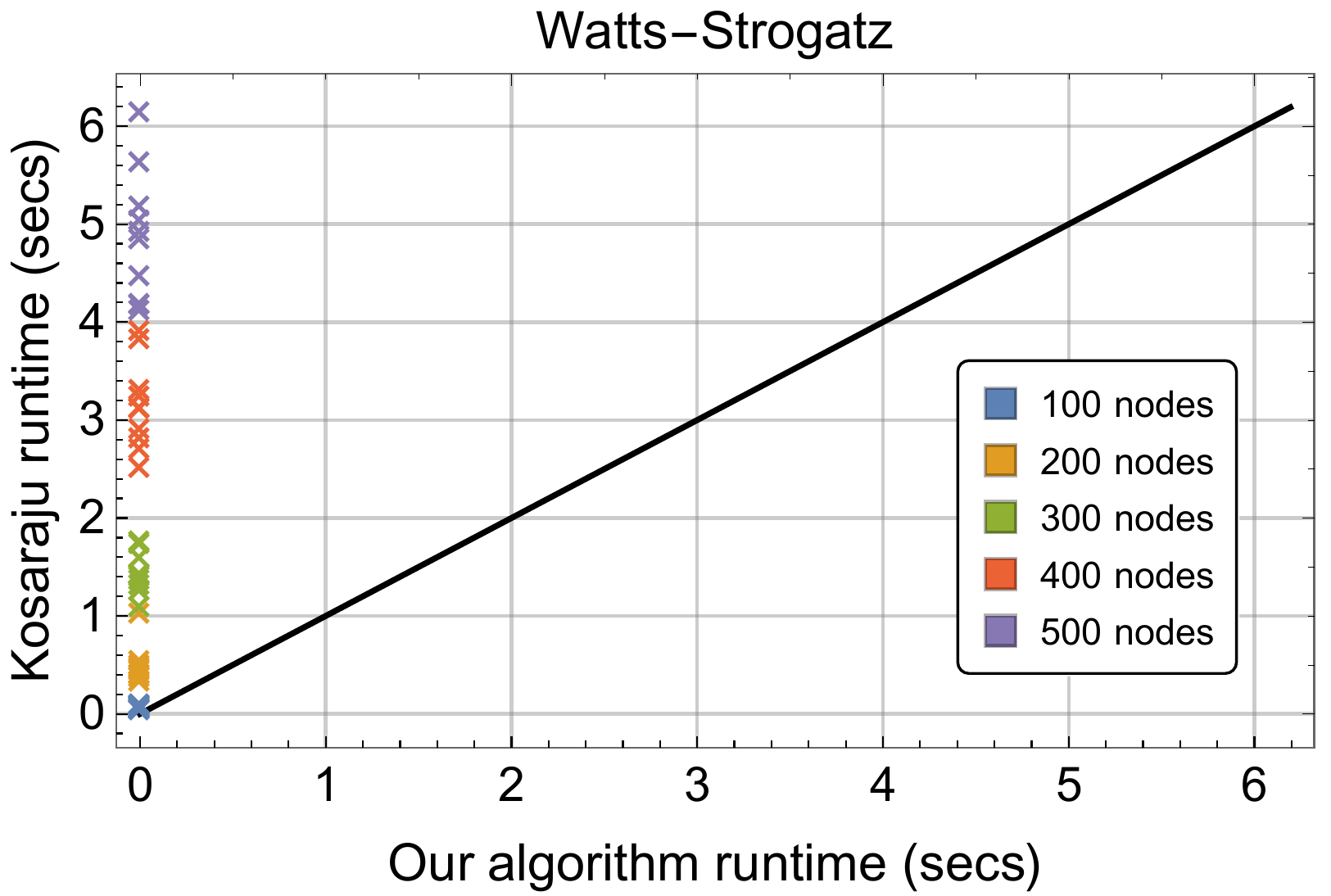}
    \caption{This figure compares the run times of  both our proposed algorithm and Kosaraju's algorithm for several randomly generated Watts-Strogatz networks.}
    \label{fig:WSruntimesparam1}
\end{figure}

\begin{figure}[H]
    \centering
    \subfigure[]{
    \includegraphics[width=.46\linewidth]{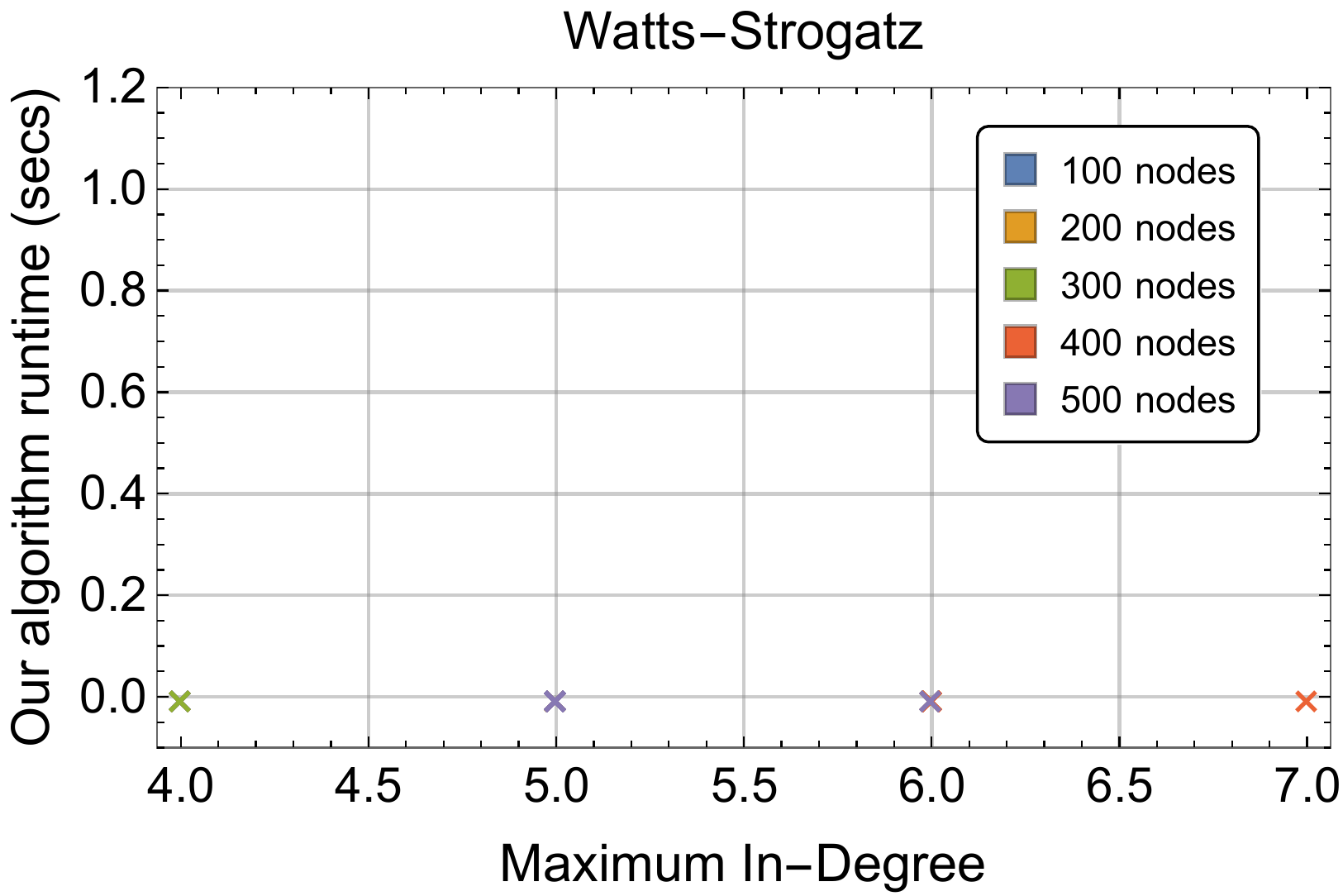}
    }
    \subfigure[]{
    \includegraphics[width=.46\linewidth]{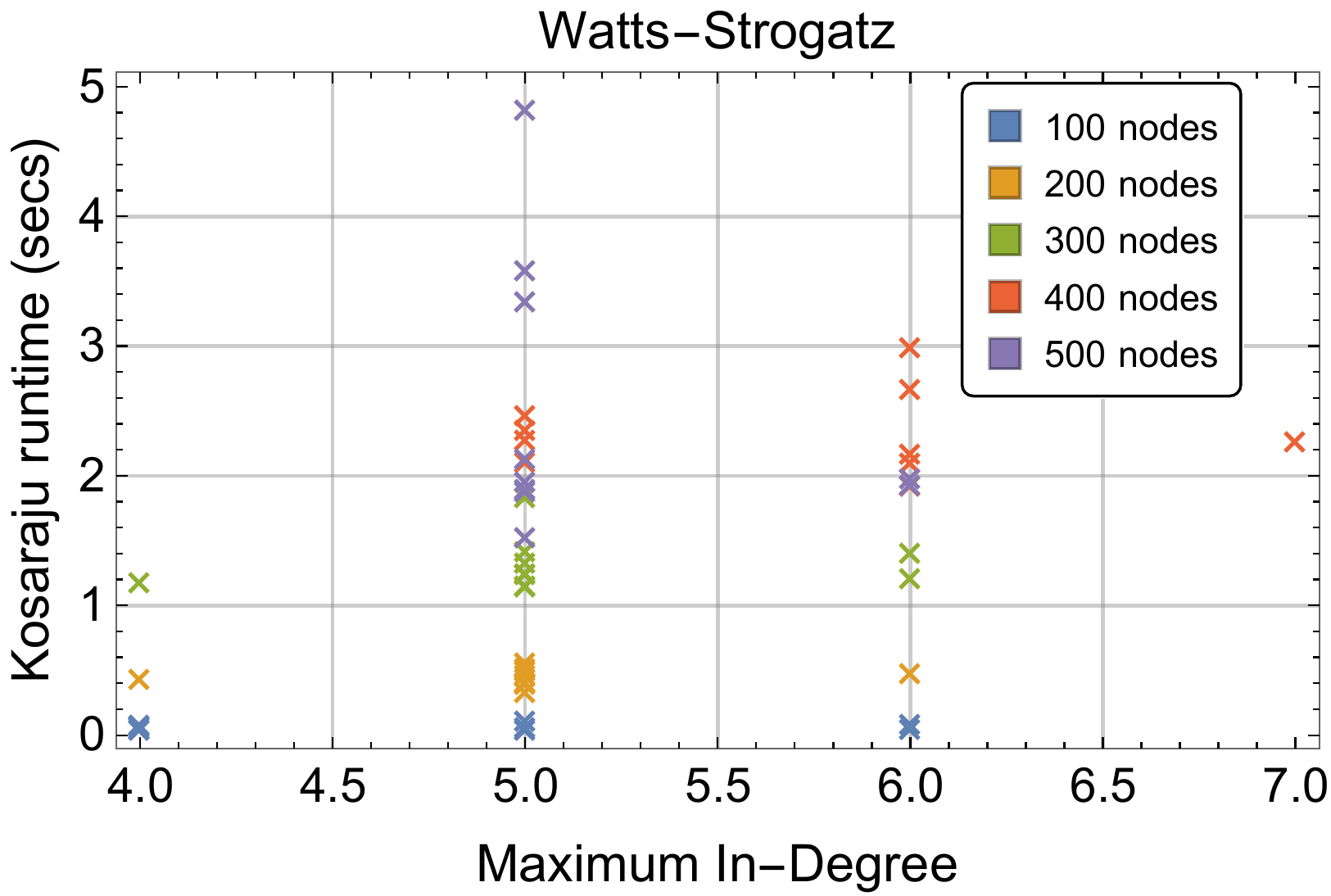}
    }
    \subfigure[]{
    \includegraphics[width=.46\linewidth]{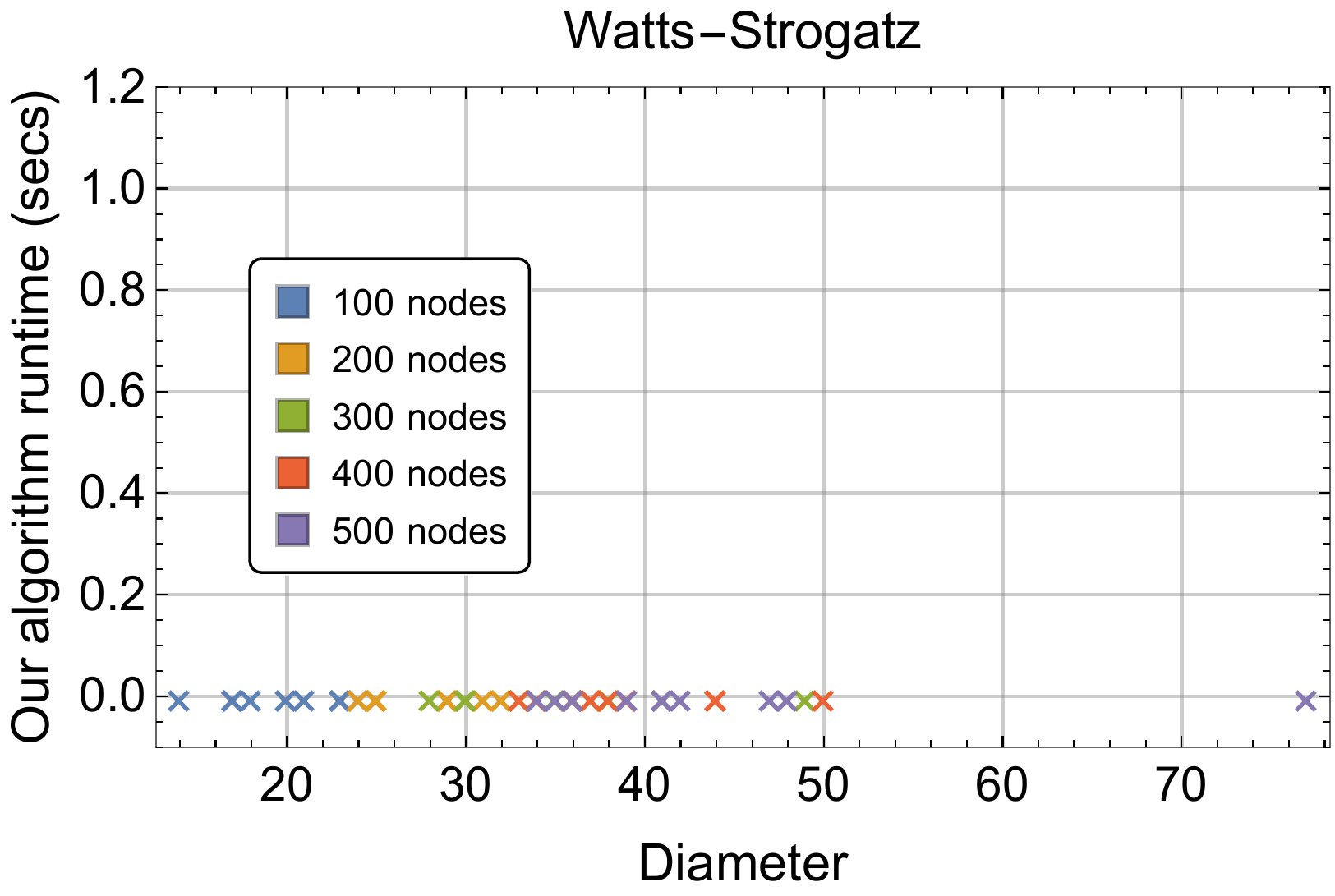}
    }
    \subfigure[]{
    \includegraphics[width=.46\linewidth]{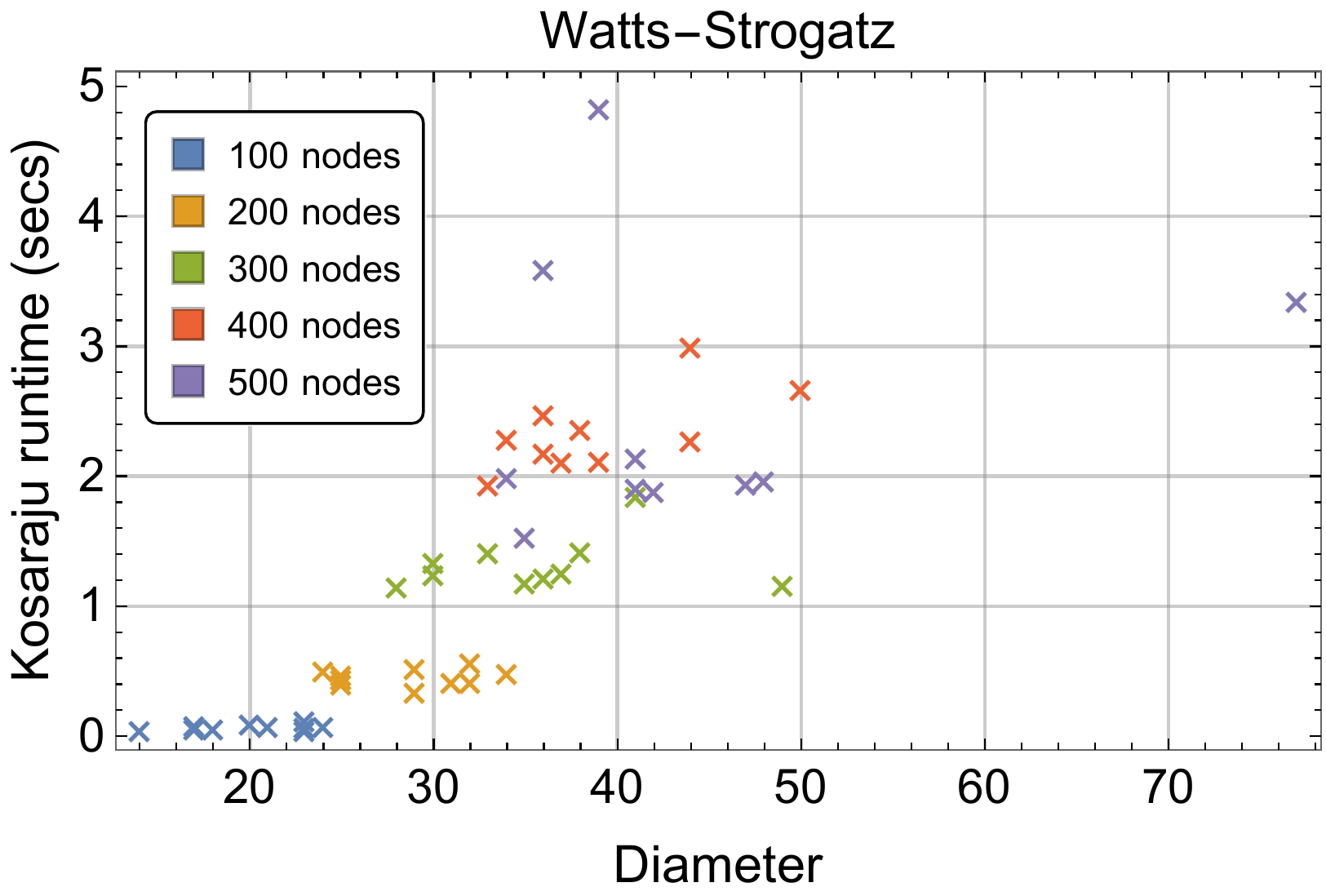}
    }
    \subfigure[]{
    \includegraphics[width=.46\linewidth]{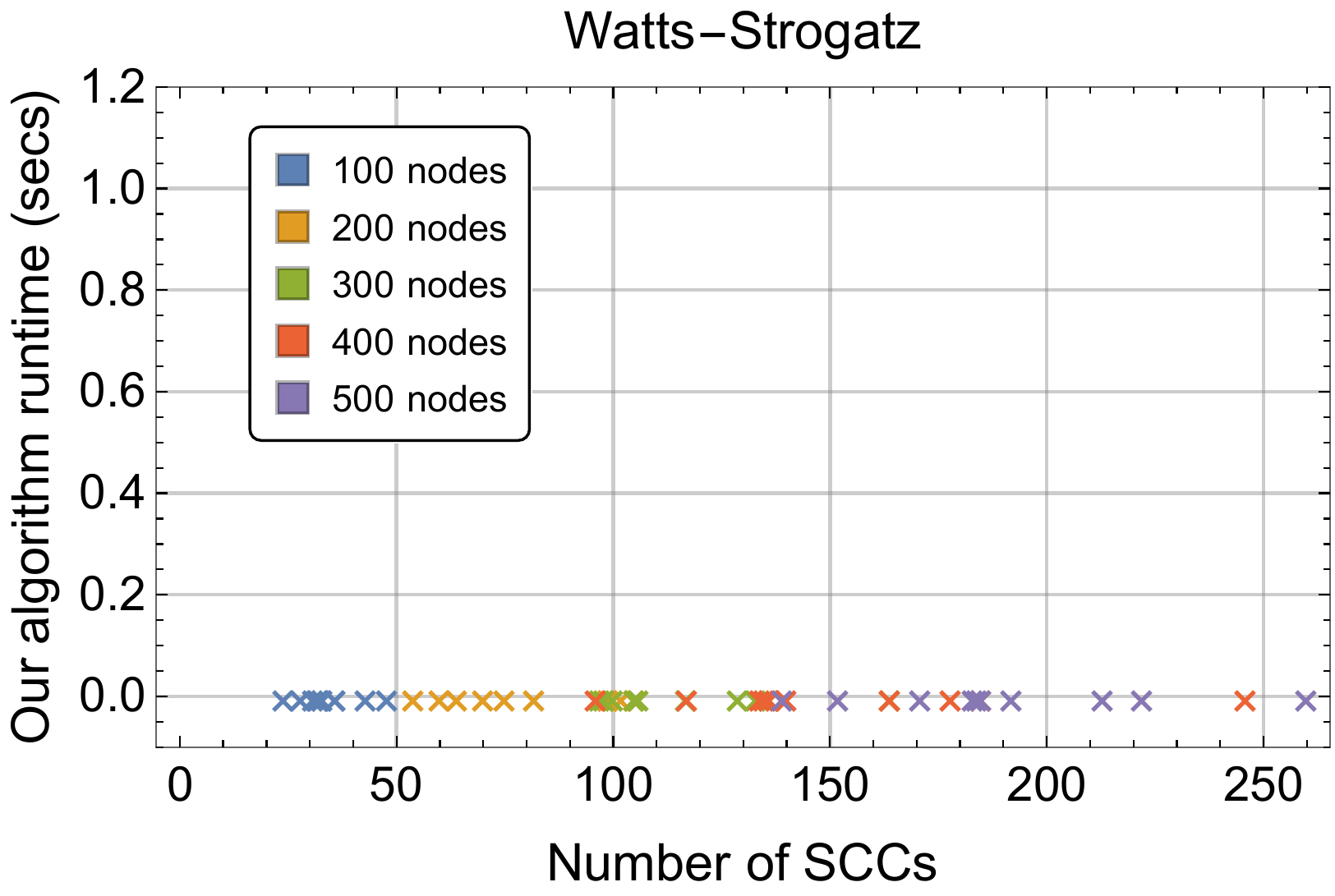}
    }
    \subfigure[]{
    \includegraphics[width=.46\linewidth]{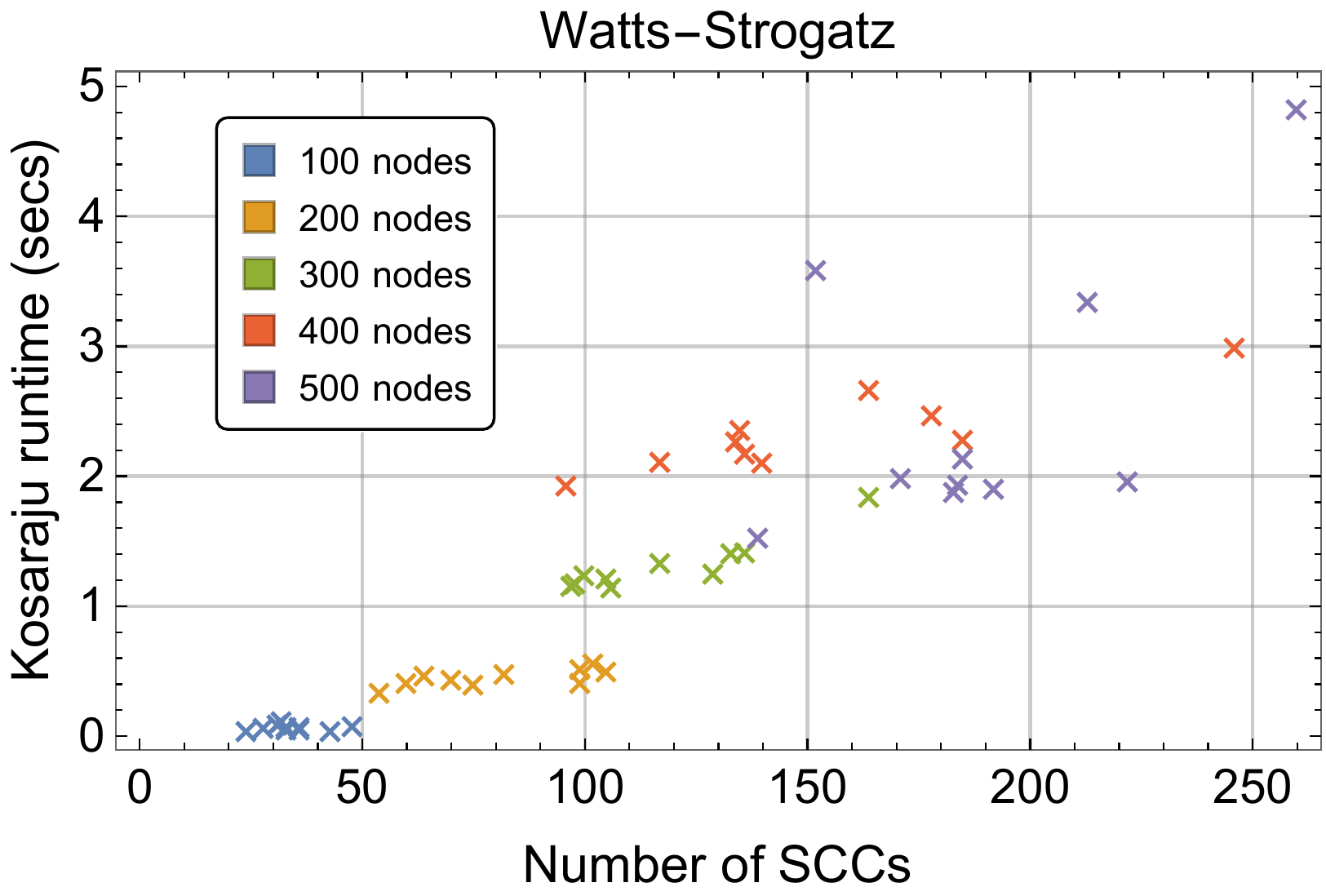}
    }
    \caption{These figures show the relationship between the network properties of some randomly generated Watts-Strogatz networks and their run times for both our proposed algorithm and Kosaraju's algorithm.}
    \label{fig:WSparam2}
\end{figure}

\begin{figure}[H]
    \centering
    \includegraphics[width =0.46 \linewidth]{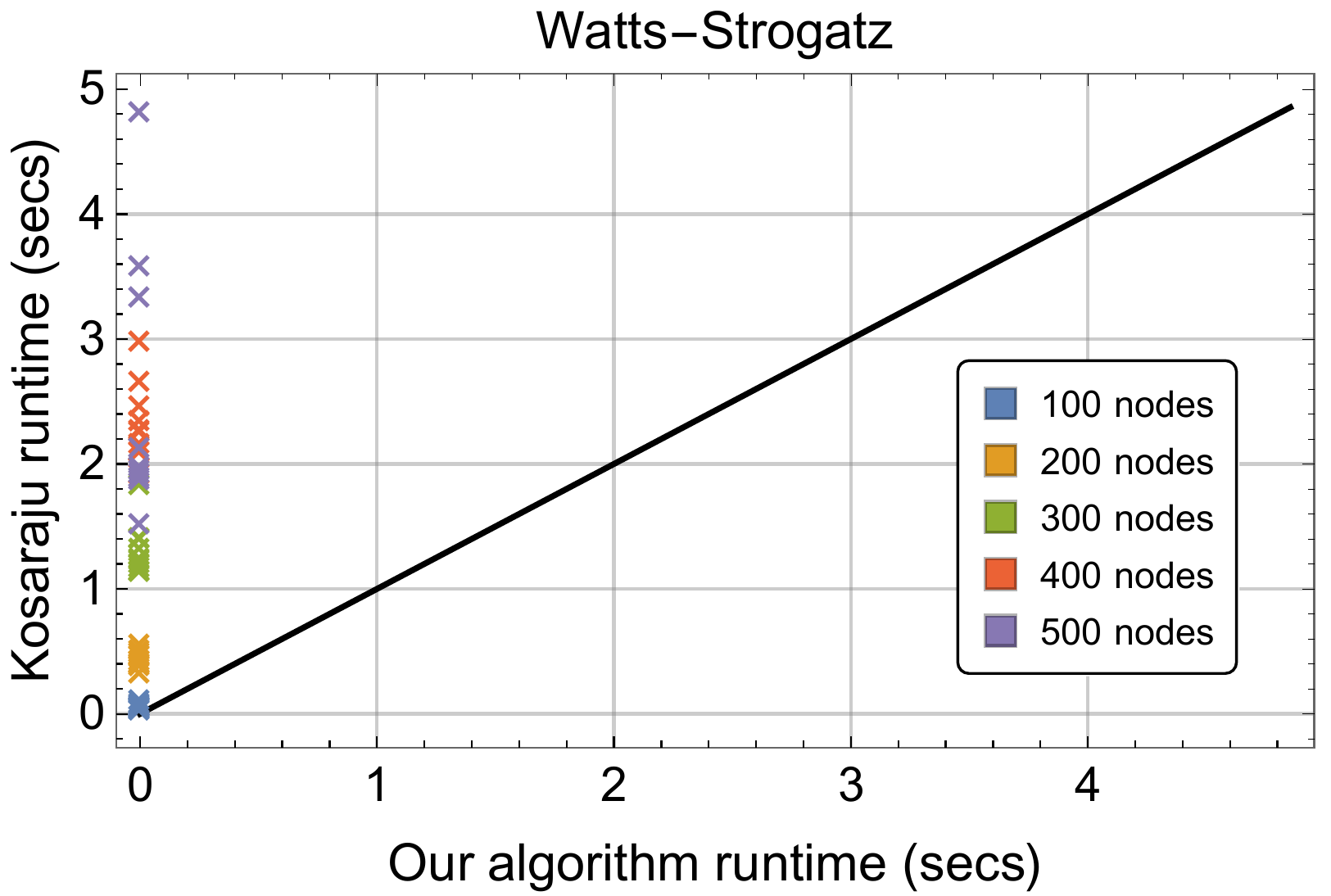}
    \caption{This figure compares the run times of  both our proposed algorithm and Kosaraju's algorithm for several randomly generated Watts-Strogatz networks.}
    \label{fig:WSruntimesparam2}
\end{figure}
In Figure \ref{fig:WSparam1}, we see that the diameter dominates the complexity, so the distributed algorithm was used.  

In Figure \ref{fig:WSruntimesparam1}, we see the comparison between the run times of both our algorithm and Kosaraju's algorithm on the different randomly generated Watts-Strogatz networks. Our distributed algorithm outperforms Kosaraju's. 

The results from the second set of parameters for Watts-Strogatz networks are shown in Figures \ref{fig:WSparam2} and \ref{fig:WSruntimesparam2}. The results from the two sets of parameters do not present much difference. It is clear that our algorithm outperforms the Kosaraju algorithm.

\subsection{Determining the Diameter of a Network}
\begin{figure}[H]
    \centering
    \subfigure[]{
    \includegraphics[width =0.46 \linewidth]{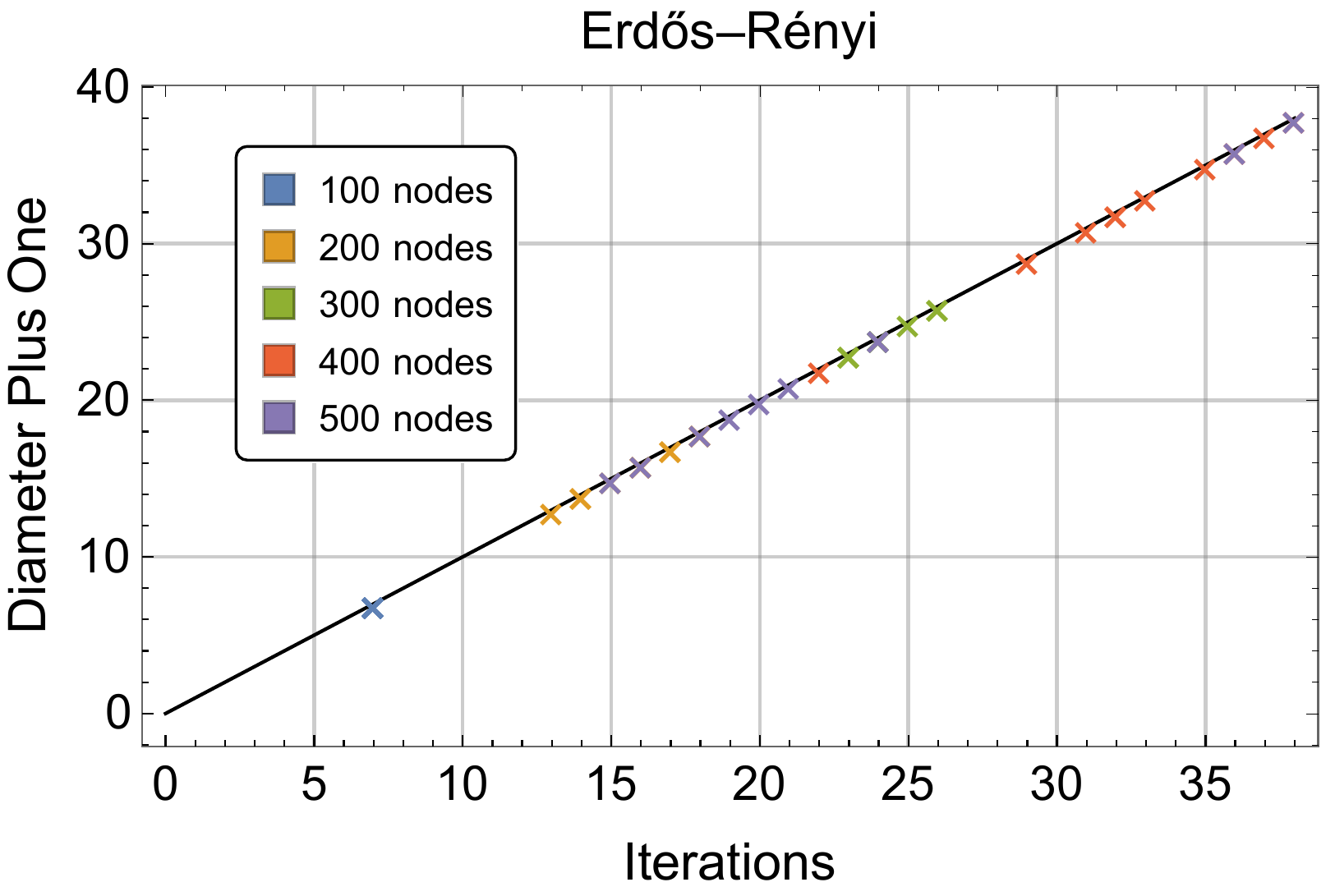}
    }
    \subfigure[]{
    \includegraphics[width =0.46 \linewidth]{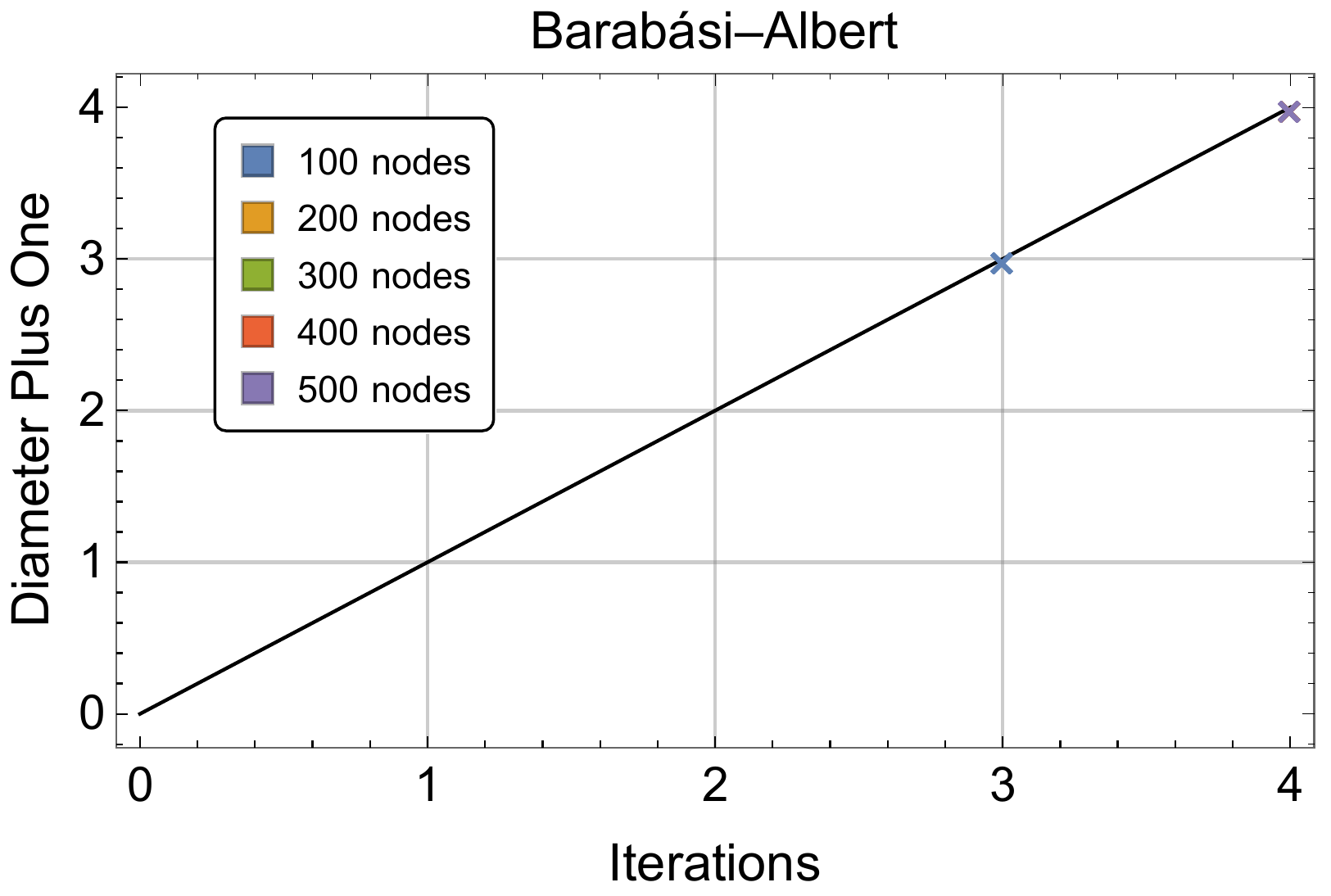}
    }
    \subfigure[]{
    \includegraphics[width =0.46 \linewidth]{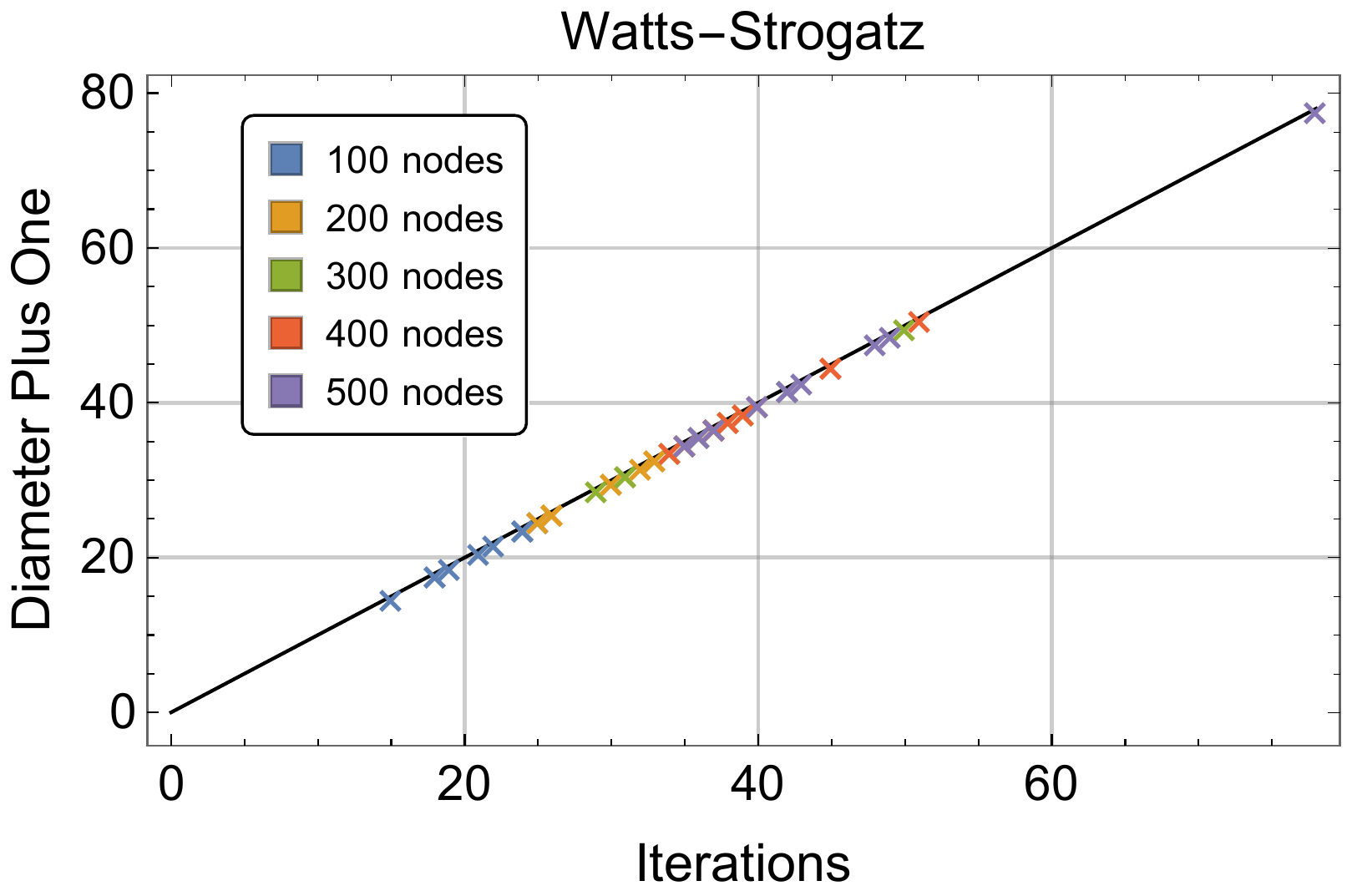}
    }
    \caption{This figure shows the relationship between the number of iterations needed before terminating our algorithm and the diameter plus one of several randomly generated networks.}
    \label{fig:diameter}
\end{figure}

From the results of Theorem~\ref{theorem:diameter}, our algorithm can determine the finite digraph diameter of the network. Here, we illustrate the relationship between the number of iterations required before terminating our algorithm compared with the finite digraph diameter plus one.

Figure \ref{fig:diameter} shows the results from running our algorithm on the random networks using the second set of parameters. We see that the number of required iterations is identical to the diameter of the network plus one. 

Finally, we compared the runtime of our algorithm with the Floyd-Warshall algorithm on the Erdős-Rényi, Barabási-Albert, and Watts-Strogatz networks. We randomly generated ten different Erdős-Rényi networks, using 25 nodes and 50 edges. For the Barabási-Albert network, we used 25 nodes and 3 edges added to each new vertex at each time step to generate ten different networks. Finally, we generated ten different Watts-Strogatz networks using 25 nodes and a linkage probability of 0.2. In Figure \ref{fig:comparediam}, we see that our algorithm outperforms the Floyd-Warshall algorithm for all networks.

\begin{figure}[H]
    \centering
    \includegraphics[width =0.46 \linewidth]{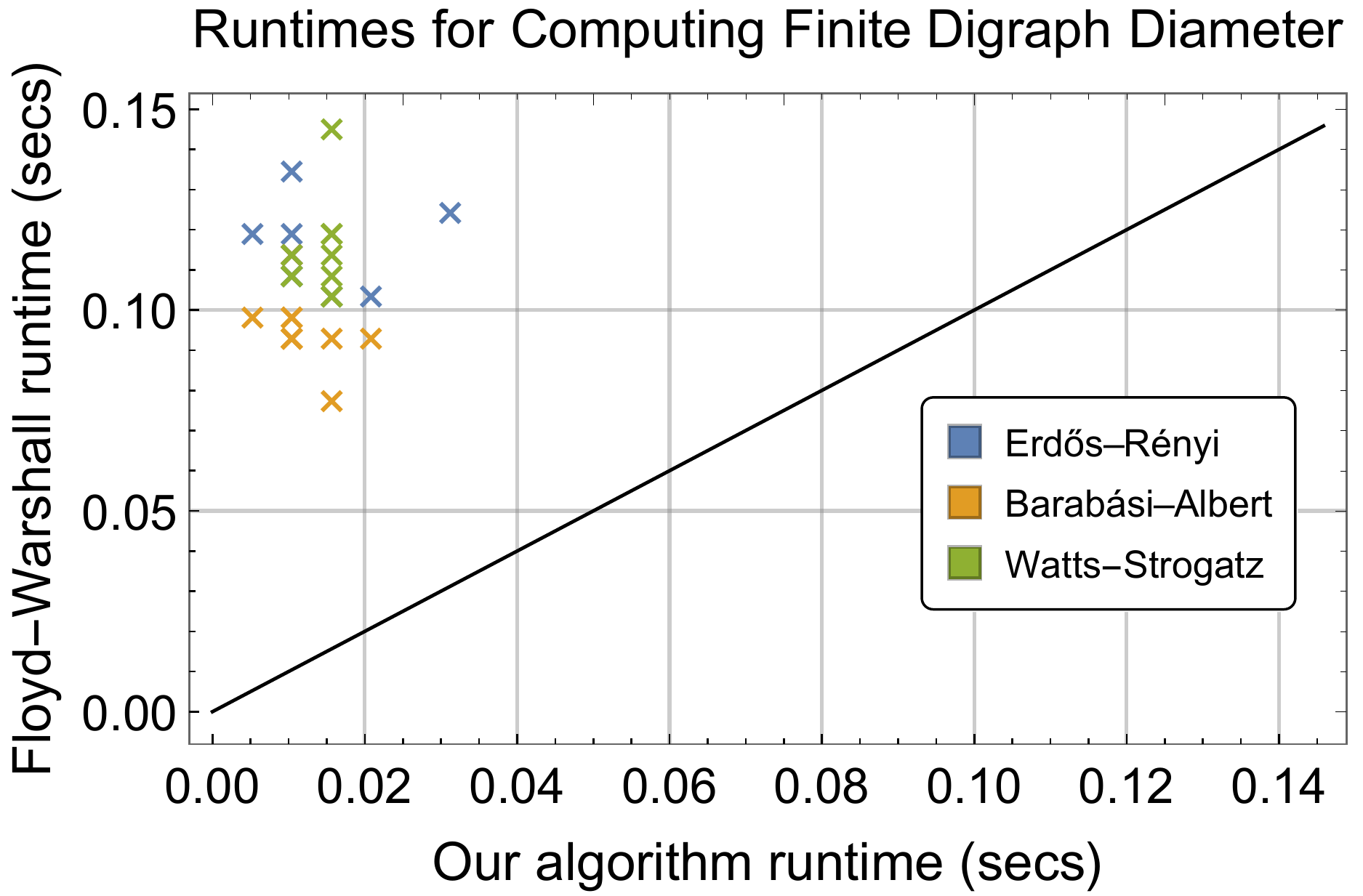}
    \caption{This figure compares the runtimes of computing the finite digraph diameter when using our algorithm against the Floyd-Warshall algorithm on several randomly generated networks, including Erdős-Rényi, Barabási-Albert, and Watts-Strogatz.}
    \label{fig:comparediam}
\end{figure}

\section{Conclusions}\label{sec:conc}

We provided a scalable and distributed algorithm to find the strongly connected components of a directed network with time-complexity $\mathcal O\left(Dd_{\text{in-degree}}^{\max}\right)$, where $D$ is the finite diameter and $d_{\text{in-degree}}^{\max}$ is the maximum in-degree of the network. Furthermore, our algorithm can be used to calculate the finite diameter of a directed network and outperforms the current state-of-the art, e.g., the Floyd-Warshall algorithm. We demonstrated the performance of our algorithm on several random networks. We compared the runtime of our algorithm against Kosaraju's algorithm and found that our algorithm outperformed Kosaraju's in nearly every tested network. Similar results readily follow regarding the computation of the finite digraph diameter on different random networks.

\vspace{-0.2cm}
{\small
\bibliographystyle{IEEEtran}
\bibliography{IEEEabrv,references}
}

\end{document}